\documentclass[11pt,leqno]{amsart}
\usepackage{pkgfile}

\newcommand{\pr}[2]{\langle {#1} , {#2} \rangle}
\newcommand{\norm}[1]{\left \| #1 \right \|}

\newtheorem{step}{Step}

\def \etc {,\ldots,}

\def \a {\alpha}
\def \b {\beta}
\def \t {\tau}

\def \dist {{\rm dist}}

\def \rank {{\rm rank }}

\def \d {\delta}
\def\Adj{\mathrm{Adj}}

\begin{document}

\title[{Invertibility of Sparse Matrices}]{{Invertibility of Sparse non-Hermitian matrices}}

\author[A.\ Basak]{Anirban Basak$^*$}
\thanks{${}^*$ Most of the work was done while A.B.~was a visiting assistant professor at Duke University, USA} 
 \author[M.\ Rudelson]{Mark Rudelson$^\dagger$}
 \thanks{${}^\dagger$Research partially supported by NSF grant DMS 1464514 and USAF Grant FA9550-14-1-0009.}


\address{$^{*}$Department of Mathematics,
Weizmann Institute of Science,
POB 26, Rehovot 76100, Israel.}

\address{$^\dagger$ Department of Mathematics, University of Michigan, East Hall, 530 Church Street, Ann Arbor, Michigan 48109, USA.}

\date{\today}

\subjclass[2010]{46B09, 60B20.}

\keywords{Random matrices, sparse matrices, smallest singular value, spectral norm, small ball probability.}

\maketitle

\begin{abstract}
We consider a class of sparse random matrices of the form $A_n =(\xi_{i,j}\delta_{i,j})_{i,j=1}^n$, where $\{\xi_{i,j}\}$ are i.i.d.~centered random variables, and $\{\delta_{i,j}\}$ are i.i.d.~Bernoulli random variables taking value $1$ with probability $p_n$, and prove a quantitative estimate on the smallest singular value for $p_n = \Omega(\frac{\log n}{n})$, under a suitable assumption on the spectral norm of the matrices. This establishes the invertibility of a large class of sparse matrices.
For $p_n =\Omega( n^{-\alpha})$ with some $\alpha \in (0,1)$, we deduce that the condition number of $A_n$ is of order $n$ with probability tending to one
under the optimal moment assumption on $\{\xi_{i,j}\}$.
 This in particular, extends a conjecture of von Neumann about the condition number to sparse random matrices with heavy-tailed entries. In the case that the random variables  $\{\xi_{i,j}\}$ are i.i.d.~sub-Gaussian, we further show that a sparse random matrix is singular with probability at most $\exp(-c n p_n)$ whenever $p_n$ is above the critical threshold $p_n = \Omega(\frac{ \log n}{n})$. The results also extend to the case when $\{\xi_{i,j}\}$ have a non-zero mean. 
 We further find quantitative estimates on the smallest singular value of the adjacency matrix of a directed Erd\H{o}s-R\'{e}yni graph whenever its edge connectivity probability  is above the critical threshold $\Omega(\frac{\log n}{n})$.
\end{abstract}


\section{Introduction}\label{sec:intro}

This paper establishes the bounds on the condition number of a sparse random matrix with independent identically distributed (i.i.d.) entries and on the probability that such matrix is singular.

For a $n \times n$ real matrix $A_n$ its {\em singular values} $s_k(A_n), k=1,2,\ldots,n$, are the eigenvalues of $|A_n|=\sqrt{A_n^*A_n}$ arranged in non-increasing order. The maximum and the minimum singular values are often of particular interest, and they can be defined as
\[
s_{\max}(A_n):= s_1(A_n):= \sup_{x \in S^{n-1}}\|A_n x\|_2, \quad s_{\min}(A_n):= s_n(A_n):= \inf_{x \in S^{n-1}}\|A_n x\|_2,
\]
where $S^{n-1}:=\{x\in \R^n : \|x\|_2=1\}$ and $\|\cdot \|_2$ denotes the Euclidean norm of a vector.
This definition means that the largest singular value $s_{\max}(A_n)$ is the operator or spectral norm of the matrix $A_n$, and the smallest singular value $s_{\min}(A_n)$ provides a {\em quantitative measure} of the invertibility of $A_n$:
  \[
   s_{\min}(A_n) = \inf \left\{ \norm{A_n-B}: \ \text{det}(B)=0 \right\},
  \]
where $\norm{A_n-B}$ denotes the operator norm of the $n \times n$ matrix $A_n-B$. Another such measure is the {\em condition number} defined as
\[
\sigma(A_n):=\f{s_{\max}(A_n)}{s_{\min}(A_n)},
\]
 which often serves a measure of stability of matrix algorithms  in numerical linear algebra.

In this paper we obtain lower bounds on the smallest singular value of a class of {\em sparse} random matrices, and then finding appropriate upper bounds on the maximum singular value, we deduce that the condition number of such matrices are well controlled, and therefore they are well invertible (see Theorem \ref{thm: smallest singular + norm}, Corollary \ref{thm: smallest singular heavy tail} and Corollary \ref{thm: smallest singular}). 

Another class of random matrices which are of interest in combinatorics and graph theory are the adjacency matrices of random graphs. Graphs, more precisely, their edges can be either undirected or directed.  Both directed and undirected graphs are abundant in real life. One of the simplest, and widely studied model in the undirected random graph literature is the Erd\H{o}s-R\'{e}yni random graph. Here we consider the directed version of that model (see Definition \ref{dfn:ER}), and show that the smallest singular value and condition number of the adjacency matrix of such random graphs are well controlled (see Theorem \ref{thm:bernoulli}). 

Analysis of extremal singular values of random matrices of large but fixed dimensions has received a lot of interest in recent years, due to its application in compressed sensing, geometric functional analysis, theoretical computer science, and other fields of science. Moreover, the bounds on the extreme singular values, especially the one on the smallest singular value, play a key role in obtaining the limiting spectral distribution of various non-Hermitian random matrix ensembles. For example, see \cite{b_dembo_unitary, bcc, gotze_tikhomirov, gkz, nguyen, RV_unitary, tao_vu, wood}. Likewise, the bounds on the smallest singular value obtained here  play a crucial role in establishing the {\em circular law} for sparse non-Herimitian random matrices, which is derived in a companion paper \cite{BR} (see also Remark \ref{rmk:circular law}).

The study of the smallest singular value of a random matrix was initiated back in 1940's when von Neumann and his collaborators used random matrices to test their algorithm for the inversion of large matrices, and they speculated that
\beq\label{eq:JvN}
s_{\min} (A_n) \sim n^{-1/2}, \quad  s_{\max} (A_n) \sim n^{1/2}\text{ with high probability}
\eeq
(see \cite[pp. 14, 477, 555]{von_neu} and \cite[Section 7.8]{vNG}).
That is,
\beq\label{eq:condition_number_conjecture}
\sigma(A_n) \sim n \text{ with high probability}.
\eeq
A more precise version of this conjecture appeared in \cite{smale}. For {\em Gaussian} random matrices it was proved that
\[
\P(s_{\min}(A) \le \vep n^{-1/2}) \sim \vep, \text{ for every } \vep \in (0,1),
\]
see \cite{edelman, szarek}.
However, the conjecture about the smallest singular value of a general random matrix remained open for a long time. For example, the result was not known even for {\em random sign} matrix, i.e. for the matrix with i.i.d~$\pm 1$ symmetric random variables. The first bound in this direction was proved in \cite{R} for matrices with i.i.d.~sub-Gaussian entries. Later in \cite{RV1} this was improved to prove lower bound on $s_{\min}$ under the finiteness of the fourth moment assumption. In particular, it was shown that for every $\delta >0$, there exists $\vep>0$ such that
\[
\P(s_{\min}(A_n) \le \vep n^{-1/2}) \le \delta.
\]
Restricting to the i.i.d.~sub-Gaussian entries their arguments also give the following strong probability bound:
\beq\label{eq:RV1_smin}
\P(s_{\min}(A_n) \le \vep n^{-1/2}) \le C \vep + c^n, \text{ for every } \vep \ge 0,
\eeq
where $C$ and $c \in (0,1)$ are some constants depending polynomially on the {\em sub-Gaussian moment} of the entries. Finally a matching upper bound was proved for sub-Gaussian entries in \cite{RV11}, and improved under finite fourth moment assumption in \cite{V2}. The necessary bounds on the largest singular value follows from \cite{L2} for entries with finite fourth moment, and from \cite{DS} for i.i.d.~sub-Gaussian entries. This establishes \eqref{eq:JvN}-\eqref{eq:condition_number_conjecture} for random matrices with centered i.i.d.~entries of unit variance with finite fourth moment.

Another line of research is directed towards proving the universality of the smallest singular value under small perturbation. This is largely motivated by its application in establishing the {\em circular law}. Considering a random matrix of i.i.d.~ entries with finite second moment Tao and Vu in \cite{tao_vu_s_min} established that for every $C'>0$ there exists a $C>0$ such that
\beq\label{eq:tv_smin}
\P(s_{\min}(A_n+M_n) \le n^{-C}) \le n^{-C'},
\eeq
where $M_n$ is a deterministic $n \times n$ matrix with $s_{\max}(M_n)= n^{O(1)}$.

The results described so far are only for {\em dense} matrices. However, {\em sparse} matrices are more abundant in statistics, neural network, financial modeling, electrical engineering, wireless communications, and in many other fields. We refer  the reader to \cite[Chapter 7]{bai} for other examples, and their relevant references. It is therefore natural to ask if there is an analogue of \eqref{eq:JvN}-\eqref{eq:condition_number_conjecture} for sparse matrices. Analysis of sparse matrices is usually more challenging than its dense counterparts because of presence of a large number of zeros. Litvak and Rivasplata in \cite{LR} considered a class of random sparse matrices. They imposed certain conditions on the columns and rows of those matrices which prevent a large number of zeros, and then under the finiteness of $(2+\vep)$ moments they show \eqref{eq:JvN}-\eqref{eq:condition_number_conjecture} hold.

Another way to construct sparse random matrix is to multiply each of the entries by i.i.d. Bernoulli entries denoted below by  ~$\dBer(p_n)$, where $p_n \ra 0$. For such matrices it was shown in \cite{tao_vu_s_min} that \eqref{eq:tv_smin} holds (a similar result appeared in \cite{gotze_tikhomirov}), as long as $p_n = \Omega(n^{-\alpha})$ for some $\alpha \in (0,1)$ (Recall that $a_n =\Omega(b_n)$ iff $\liminf_{n \ra \infty} {a_n}/{b_n} \ge K$ for some $K>0$). In \cite{gotze_tikhomirov}, under a minimal moment assumption, it was also shown that $s_{\max}(A_n) \le n \sqrt{p_n}$ with probability tending to $1$. This implies that $\sigma(A_n)= O(n^{C})$, for a large constant $C$, which is weaker than the conjecture \eqref{eq:condition_number_conjecture} for these sparse matrices.

On the other hand, it is straightforward to check that when $p_n \le \f{\log n}{n}$, the probability of the matrix containing a zero row is positive and bounded below uniformly in $n$, thereby making it singular. Thus the analogue of \eqref{eq:JvN} cannot be extended beyond the $\f{\log n}{n}$ barrier. Therefore it would be interesting to check if analogue of \eqref{eq:JvN}-\eqref{eq:condition_number_conjecture} hold for all $p_n = \Omega(\f{\log n}{n})$.

In our first result below we provide an affirmative answer to the question above, under a suitable  assumption on the maximal singular value. Note that it only requires the finiteness of the fourth moment. In the theorem below we consider a slightly different set-up, where we allow the entries on the diagonal to be arbitrary as long as they are not too big. This generalization is motivated by its role in the analysis of the adjacency matrix of a random directed graph as well as in the proof of the circular law (see Remark \ref{rmk:circular law} for more details). The case of matrices with i.i.d.~entries follows by conditioning on the diagonal entries, and showing that, with high probability, they satisfy the requirements of the main theorem. This is established in  Corollary \ref{thm: smallest singular heavy tail} and Corollary \ref{thm: smallest singular}. Before stating the main theorem, for ease of writing, let us introduce the notation $[n]:=\{1,2,\ldots,n\}$.

\begin{thm}  \label{thm: smallest singular + norm}
 Let $\bar{A}_n$ be an $n \times n$ matrix with zero on the diagonal and has i.i.d.~off-diagonal entries $a_{i,j}= \delta_{i,j} \xi_{i,j}$, where $\delta_{i,j}, \ i,j \in [n], i \ne j$, are independent Bernoulli random variables taking value 1 with probability $p_n \in (0,1]$, and $\xi_{i,j}, \ i,j \in [n], i \ne j$ are i.i.d.~random variables with zero mean, unit variance, and finite fourth moment. Fix $K \ge 1$, and let $\Omega_K:= \Big\{\norm{\bar{A}_n} \le K \sqrt{np_n}\Big\}$. Further let $D_n$ be a real non-random diagonal matrix with $\norm{D_n} \le R \sqrt{n p_n}$, for some positive constant $R$. 
 Then there exist constants $0< c_{\ref{thm: smallest singular + norm}}, c'_{\ref{thm: smallest singular + norm}}, C_{\ref{thm: smallest singular + norm}}, \ol{C}_{\ref{thm: smallest singular + norm}} < \infty$, depending on $K, R$, and on the  fourth moment of $\xi_{i,j}$, such that for any $\vep>0$, and
 \beq
p_n \ge  \frac{\ol{C}_{\ref{thm: smallest singular + norm}} \log n}{n},\label{eq:p_assumption}
 \eeq
 \beq\label{eq: smallest singular + norm}
  \P \bigg( \Big\{s_{\min}(\bar{A}_n + D_n) \le C_{\ref{thm: smallest singular + norm}} \vep \exp \left(-c_{\ref{thm: smallest singular + norm}} \frac{\log (1/p_n)}{\log (np_n)} \right) \sqrt{\frac{p_n}{n}} \Big\} \bigcap   \Omega_K \bigg)
  \le \vep +  \exp(-c'_{\ref{thm: smallest singular + norm}}np_n).
 \eeq
\end{thm}

\begin{rmk}\label{rmk:circular law}
In Theorem \ref{thm: smallest singular + norm} we studied the smallest singular value of $\bar{A}_n +D_n$  instead of considering $A_n$, the matrix with i.i.d. entries. Since in directed Erd\H{o}s-R\'{e}yni graphs, we do not allow self-loops, the diagonal entries of its adjacency matrix are zero. This has motivated us to consider $\bar{A}_n$ in Theorem \ref{thm: smallest singular + norm} with zeros on the diagonal. Addition of an extra diagonal matrix $D_n$ to $\bar{A}_n$ has been motivated by its application in identifying the limiting spectral distribution of ${A}_n$. It is well known that in order to establish the convergence of empirical distribution of the eigenvalues  of $A_n$, one needs to prove the convergence of the integral of $\log (\cdot)$ with respect to the empirical distribution of the singular values of $A_n/\sqrt{np} -\omega I_n$ for Lebesgue a.e.~$\omega \in \C$ (for more details see \cite{bordenave_chafai}). Whenever, the limiting distribution is compactly supported, one can restrict $\omega$ in a ball in the complex plane.

Since $\log(\cdot)$ is unbounded near $0$, one must have a control on $s_{\min}(\cdot)$. Set $D_n= \omega \sqrt{np} I_n+ \Lambda_n$, where $\Lambda_n$ is the diagonal matrix consisting of the diagonal entries of $A_n$, in Theorem \ref{thm: smallest singular + norm}. Upon showing that  $\norm{\Lambda_n}=O(\sqrt{np_n})$ with high probability, we have the required estimate on $s_{\min}(\cdot)$ for all bounded real $\omega$ (recall that in Theorem \ref{thm: smallest singular + norm} we need $D_n$ to be a matrix with real entries).  The difficulty for complex $\omega$ arises because of an $\vep$-net argument. See Remark \ref{rmk:omega_complex} and Remark \ref{rmk:omega_complex_small_lcd} for more details.

In \cite{BR} we overcome this difficulty and extend Theorem \ref{thm: smallest singular + norm} for complex $\omega$.
Since such extension requires a significant additional work, we defer it to \cite{BR} where it is applied to proving the circular law for such matrices.
\end{rmk}

\begin{rmk}
We prove Theorem \ref{thm: smallest singular + norm} under the assumption of unit variance, and finite fourth moment of $\{\xi_{i,j}\}$. This assumption can easily be relaxed to unit variance, and bounded $(2+\eta)$-th moment, for any $\eta >0$. The boundedness of fourth moment is required in the proof of Lemma \ref{l: spread vector}, where it has been used to apply Paley-Zygmund inequality. However, Paley-Zygmund inequality continues to hold as long as the $(2+\eta)$-th moment is finite (see \cite[Lemma 3.5]{LPRT}). To apply this version of the Paley-Zygmund inequality in the proof of Lemma \ref{l: spread vector} we need to bound $\E[|\sum_{i=1}^n \theta_i x_i |^{2+\eta}]$, where $\{\theta_i\}_{i \in[n]}$ are symmetrized versions of $\{\xi_i\}_{i \in [n]}$, and $x\in S^{n-1}$. To this end, one can use \cite[Theorem 6.20]{LT} to obtain the necessary bounds. Finiteness of fourth moment has also been used in Proposition \ref{prop:lcd}. Since \cite[Assumption 1.4]{RV_no-gap} holds under the unit variance, bounded $(2+\eta)$-th moment, one can instead use \cite[Corollary 7.6]{RV_no-gap} to arrive at the same conclusion. For the clarity of presentation,  we work with the finite fourth moment assumption. 
\end{rmk}

\bigskip

\noindent
To obtain the necessary estimates on the spectral norm in Theorem \ref{thm: smallest singular  +  norm}, we first focus on {\em heavy-tailed} random variables, and establish the required bound when $p_n = \Omega(n^{-\alpha})$. For dense matrices, the finiteness of the fourth moment is sufficient (and also necessary) to guarantee the necessary bounds on $s_{\max}(\cdot)$ (see \cite{L2}). However, for sparse case one {\em needs}  finiteness of the higher moments depending on the choice of $\alpha$. This is established in the second part of the next theorem. Before stating this theorem let us recall that $a_n \sim b_n$ means that there exists positive constants $c, C$ such that $c b_n \le a_n \le C b_n$ for all large $n$.

\begin{thm}\label{lem:norm_heavy_tail}
 Fix $\a \in (0,1)$ and let $p_n=\Omega(n^{-\a})$. Denote
  \[
    q:=\frac{2(2-\a)}{1-\a}.
  \]
  Let ${A}_n$ be an $n \times n$ random matrix with i.i.d.~entries ${a}_{ij}=\d_{ij} \xi_{ij}$, where $\d_{ij}$ are Bernoulli random variables with $\P(\d_{ij}=1)=p_n$, and $\xi_{ij}$ are independent copies of a centered random variable $\xi$ of unit variance and finite fourth moment.

 \noindent
(i)  If $\E |\xi|^q <\bar{K}^q$, for some $\bar{K} < \infty$, then, for any $r<q$, there exists some positive constant $\bar{C}$ depending on $\a, \bar{K}, r$, and the fourth moment of $\xi$, such that
    \[
      \E \norm{{A}_n}^r \le \bar{C} (\sqrt{np_n})^r.
    \]

   \noindent
   (ii) Let $p_n \sim n^{-\a}$. For any $r<q$, there exist $\mu,\nu>0$, depending on $r$ and $q$, and a centered random variable $\xi$ with $\E |\xi|^r <\bar{K}$, and $\E|\xi|^q =\infty$, such that
    \[
     \P( \norm{{A}_n} \le n^\nu \sqrt{np_n} ) \le \exp (-\bar{C} n^\mu),
    \]
    where $\bar{C}$ is an absolute constant.
\end{thm}

\bigskip

Recall that $\Lambda_n$ is the diagonal matrix consisting of the diagonal entries of $A_n$. Thus denoting $\bar{A}_n:={A}_n-\Lambda_n$, we see that it  satisfies the conditions of Theorem  \ref{thm: smallest singular  +  norm}. To apply Theorem  \ref{thm: smallest singular  +  norm} for $A_n$ we also need to establish that $\norm{\Lambda_n}= O(\sqrt{np_n})$ with large probability. This can be done similarly as in Theorem \ref{lem:norm_heavy_tail} (see proof of Corollary \ref{thm: smallest singular heavy tail}). Moreover, when $p_n=\Omega(n^{-\alpha})$ we have
\[
\f{\log (1/p_n)}{\log (np_n)}= O(1).
\]
Therefore we obtain the following corollary.
\begin{cor}  \label{thm: smallest singular heavy tail}
Let $A_n$ be an $n \times n$ matrix with i.i.d.~entries $a_{i,j}= \d_{i,j} \xi_{i,j}$, where $\d_{i,j}, \ i,j \in [n]$ are independent Bernoulli random variables taking value 1 with probability $p_n$
  and $\xi_{i,j}, \ i,j \in [n]$ are {i.i.d.~centered random variables with variance at least one, and finite fourth moment}. Let $\{D_n\}_{n \in \N}$ be a sequence of real diagonal matrices such that $\norm{D_n} \le R \sqrt{np_n}$ for all $n$, and for some $R<\infty$. Assume that
  \[
  p_n = \Omega( n^{-\alpha}), \text{ for some }\alpha \in (0,1) \text{ and } \E|\xi_{i,j}|^q < \infty, \text{ where } q=\frac{2(2-\a)}{1-\a}.
  \]
  Then for every $\delta >0$, there exists an $\vep>0$ and $n_0$, depending on $R, \alpha, \delta$, and $q$-th moment of $|\xi_{i,j}|$, such that
 \beq\label{eq: smallest singular}
  \P \left( s_{\min}(A_n+D_n) \le \vep \sqrt{\frac{p_n}{n}} \right )
  \le \delta \text{ for all } n \ge n_0.
 \eeq
\end{cor}

\bigskip

Now note that combining Theorem \ref{lem:norm_heavy_tail}, and Corollary \ref{thm: smallest singular heavy tail} we immediately deduce that, for any $\delta >0$, there exists $K_0$, and $n_0$ depending on $\delta, \alpha$, and the $q$-th moment of $|\xi_{i,j}|$, such that
\[
\P(\sigma(A_n) \le K_0 n ) \ge 1- \delta \text{ for all } n \ge n_0,
\]
validating \eqref{eq:condition_number_conjecture} for heavy-tailed sparse random matrices.
Assertion (ii) of Theorem \ref{lem:norm_heavy_tail} shows that the moment condition $\E |\xi_{i,j} |^q < \infty$ is optimal.

Next we consider sparse matrices with a lighter tail. To this end, recall the definition of sub-Gaussian random variables.

\begin{dfn}\label{dfn:sub_gaussian}
For a random variable $\xi$, the sub-Gaussian norm of $\xi$, denoted by  $\norm{\xi}_{\psi_2}$, is defined as
\[
 \norm{\xi}_{\psi_2} := \sup_{k \ge 1} k^{-1/2} \norm{\xi}_k,
\]
where for every $k \in \N$, $\norm{\xi}_k:=(\E|\xi|^k)^{1/k}$.
If the sub-Gaussian norm is finite, the random variable $\xi$ is called sub-Gaussian.
\end{dfn}
This definition of the norm $\norm{\cdot}_{\psi_2}$ is equivalent to the cannonical one, see \cite{R_lec_notes}.

We now state our result about spectral norm of sparse random matrices with sub-Gaussian entries.

\begin{thm}\label{lem:norm_subgaussian}
There exists $C_0 \ge 1$ such that the following holds. Let $n \in \N$ and $p_n \in (0,1]$ be such that $p_n \ge C_0 \frac{\log n}{n}$.
 Let ${A}_n$ be an $n \times n$ random matrix with i.i.d.~entries ${a}_{ij}=\d_{ij} \xi_{ij}$, where $\d_{ij}$ are Bernoulli random variables with $\P(\d_{ij}=1)=p_n$ and $\xi_{ij}$ are centered sub-Gaussian random variables.
 Then   there exist   positive constants $C_{\ref{lem:norm_subgaussian}}, c_{\ref{lem:norm_subgaussian}}$, depending on  the sub-Gaussian norm of $\{\xi_{ij}\}$, so that
 \[
  \P( \norm{{A}_n} \ge {C}_{\ref{lem:norm_subgaussian}} \sqrt{np_n} ) \le \exp(-c_{\ref{lem:norm_subgaussian}} np_n).
 \]
\end{thm}

Proceeding as in Theorem \ref{lem:norm_subgaussian} we can also show that $\norm{\Lambda_n} =O(\sqrt{n p_n})$ with large probability. Therefore we obtain the following corollary.
\begin{cor}  \label{thm: smallest singular}
Let $A_n$ be an $n \times n$ matrix with i.i.d. entries $a_{i,j}= \d_{i,j} \xi_{i,j}$, where $\d_{i,j}, \ i,j \in [n]$ are independent Bernoulli random variables taking value 1 with probability $p_n$
  and $\xi_{i,j}, \ i,j \in [n]$ are {i.i.d.~centered sub-Gaussian random variables with variance at least one}. Let $\{D_n\}$ be a sequence of real diagonal matrices such that $\norm{D_n} \le R \sqrt{np_n}$ for all $n$, and for some $R<\infty$. Then there exist constants $0 < c_{\ref{thm: smallest singular}}, \ol{c}_{\ref{thm: smallest singular}}, C_{\ref{thm: smallest singular}}, \ol{C}_{\ref{thm: smallest singular}} < \infty$, depending on $R$, and the sub-Gaussian norm of $\xi_{i,j}$, such that for
\[
  p_n \ge \frac{\ol{C}_{\ref{thm: smallest singular}} \log n}{n},
\]
and any $\vep >0$,
 \beq\label{eq: smallest singular}
  \P \left( s_{\min}(A_n+D_n) \le {C}_{\ref{thm: smallest singular}} \, \vep \exp \left(-c_{\ref{thm: smallest singular}} \frac{\log (1/p_n)}{\log (np_n)} \right) \cdot \sqrt{\frac{p_n}{n}} \right )
  \le \vep +  \exp(-\ol{c}_{\ref{thm: smallest singular}}np_n).
 \eeq
\end{cor}

\bigskip
Since $p_n = \Omega (\f{\log n}{n})$, we have
\[
\f{\log (1/p_n)}{\log (np_n)}= O\left(\f{\log n}{\log \log n}\right).
\]
Thus we deduce that for all $p_n = \Omega (\f{\log n}{n})$, for every $\vep >0$,
\[
\P\left(\sigma(A_n) \ge C\vep^{-1} n^{1+\f{c}{\log \log n}}\right) \le \vep + \exp(-cnp_n)
\]
where $c, C$ are some constants, depending only on the sub-Gaussian norm of $\xi_{i,j}$. This validates \eqref{eq:condition_number_conjecture} upto a factor of  $n^{\f{c}{\log \log n}}$.

Also, letting $\vep \to 0$, we obtain the optimal bound for the probability that the sparse random matrix is singular:
\[
\P \big( \text{det}(A_n)=0\big) \le \exp(-cnp_n)
\quad \text{whenever } p_n = \Omega \left(\f{\log n}{n} \right).
\]

\begin{rmk}  \label{as: Lipschitz conc}
In Theorem \ref{lem:norm_subgaussian} we considered only sub-Gaussian random variables. One can consider a more general class of light tailed random variables. Namely, we can consider random variables $\xi$ such that
\beq\label{eq:beta}
\E|\xi|^h \le C^h h^{\beta h}, \text{ for all } h \ge 1, \text{ and for some constants } C \text{ and }\beta.
\eeq
Note that $\beta=1/2$ yields the sub-Gaussian random variables. Considering sparse random matrices with i.i.d. copies of $\xi$ satisfying \eqref{eq:beta} for $\be \ge 1/2$, one can show that $\norm{A_n} = O(\sqrt{np_n})$, for all $p_n$ satisfying $np_n = \Omega((\log n)^{2\beta})$. For an outline of the proof see Remark \ref{rmk:norm_general_beta}.
\end{rmk}

We now extend our result for the adjacency matrix of directed Erd\H{o}s-R\'{e}yni random graph. Let us begin with the relevant definitions.

\begin{dfn}\label{dfn:ER}
Let $\mathsf{G}_n$ be a random directed graph on $n$ vertices, with vertex set $[n]$, such that for every $i \ne j$, a directed edge from $i$ to $j$ is present with probability $p$, independently of everything else. Assume that the graph $\mathsf{G}_n$ is simple, i.e.~no self-loops and multiple edges are present. We call this graph $\mathsf{G}_n$ a directed Erd\H{o}s-R\'{e}yni graph with edge connectivity probability $p$. For any such graph $\mathsf{G}_n$ we denote $\Adj_n:=\Adj(\mathsf{G}_n)$ to be its adjacency matrix. That is, for any $i, j \in [n]$,
\[
\Adj_n(i,j)=\left\{\begin{array}{ll}
1 & \mbox{if a directed edge from $i$ to $j$ is present in $\mathsf{G}_n$}\\
0 & \mbox{otherwise}.\end{array} \right. \] 
\end{dfn}

We now have the following theorem on the smallest singular value of the adjacency matrix of a directed Erd\H{o}s-R\'{e}yni graph.

\begin{thm}\label{thm:bernoulli}
Let $\Adj_n$ be the  adjacency matrix of a directed Erd\H{o}s-R\'{e}yni graph, with edge connectivity probability $p_n \in (0,1)$. Fix $R \ge 1$, and let $D_n$ be a non-random real valued diagonal matrix with $\norm{D_n} \le R\sqrt{n p_n}$. Then there exist constants $0 < c_{\ref{thm:bernoulli}}, \ol{c}_{\ref{thm:bernoulli}}, C_{\ref{thm:bernoulli}}, \ol{C}_{\ref{thm:bernoulli}} < \infty$, depending only on $R$, such that for
\[
 \frac{\ol{C}_{\ref{thm:bernoulli}} \log n}{n}\le p_n \le  1- \frac{\ol{C}_{\ref{thm:bernoulli}} \log n}{n},
\]
and any $\vep >0$,
 \beq\label{eq: smallest singular bernoulli}
  \P \left( s_{\min}(\Adj_n+D_n) \le {C}_{\ref{thm:bernoulli}} \, \vep \exp \left(-c_{\ref{thm:bernoulli}} \frac{\log (1/p_n)}{\log (np_n)} \right) \cdot \sqrt{\frac{p_n}{n}} \right )
  \le \vep +  \exp(-\ol{c}_{\ref{thm:bernoulli}}np_n).
 \eeq

\end{thm}

\bigskip
In Theorem \ref{thm: smallest singular + norm} the entries of the matrix under consderation have zero mean. So we cannot apply those results directly to prove Theorem \ref{thm:bernoulli}. We extend Theorem \ref{thm: smallest singular + norm} for non-centered random variables (see Theorem \ref{thm: smallest singular + norm + non-centered}) which yields the desired result for the directed Erd\H{o}s-R\'{e}yni graphs.

\vskip10pt
\subsection*{Outline of the paper }

 \begin{itemize}
 \item In Section \ref{sec: proof outline}, we introduce the necessary concepts and provide an outline of the proof of Theorem \ref{thm: smallest singular + norm}. The proof is based on decomposing the unit sphere into compressible, dominated, and incompressible vectors, and controlling the infimum of $\norm{(\bar{A}_n+D_n)x}_2$ for each these three parts.

 \item The main result in Section \ref{sec: compressible} is a lower bound of the infimum over compressible and dominated vectors (see Proposition \ref{p: dominated and compressible}).
     The idea of splitting the sphere into compressible (close to sparse) and incompressible ones originated in \cite{LPRT} and was further developed in \cite{R,RV1}.
     Yet, for sparse random matrices, it can be implemented only for vectors with a relatively large support.
     To treat the vectors with a very small support, we had to introduce a new class, namely \emph{dominated vectors}.
     Handling these vectors requires a new technique based on sparsity of the matrix. \\
      First we prove a concentration result in Lemma \ref{l: pattern}. Using this lemma, we derive a lower bound for the infimum of $\norm{(\bar{A}_n+D_n)x}_2$ for $O(p^{-1})$-compressible and dominated vectors  in Lemma \ref{l: sparse vectors}, and Lemma \ref{l: compressible}. To deal with $cn$-compressible and dominated vectors, we first derive a result in Corollary \ref{c: spread vector}, using which we prove the desired for dominated vectors in Lemma \ref{l: dominated vectors}, and then we finally prove Proposition \ref{p: dominated and compressible}. Before concluding the section we point out that the techniques in this section allow us to consider $D_n$ with complex entries, in Proposition \ref{p: dominated and compressible} (see Remark \ref{rmk:omega_complex}). 

 \item In Section \ref{sec: small LCD} we prove a result about the infimum for vectors with small \abbr{LCD} (see Proposition \ref{p: norm on S_L}). Before proving this result, we first recall few preliminary facts about \abbr{LCD}. Unlike in Section \ref{sec: compressible}, 
it does not extend for $D_n$ with complex entries (see Remark \ref{rmk:omega_complex_small_lcd}).

 \item In Section \ref{sec:main_thm_proof}, we combine results from Section \ref{sec: compressible}, and Section \ref{sec: small LCD}, and complete the proof of Theorem \ref{thm: smallest singular + norm}.

\item In Section \ref{sec:norm}, we prove  Theorem \ref{lem:norm_heavy_tail} and Theorem \ref{lem:norm_subgaussian} establishing the necessary estimates on spectral norm for sparse random matrices with heavy tail and sub-Gaussian random variables. Combining these results with Theorem \ref{thm: smallest singular + norm}, we prove Corollary \ref{thm: smallest singular heavy tail} and  Corollary \ref{thm: smallest singular}. Finally in Remark \ref{rmk:norm_general_beta} we outline an extension of  Theorem \ref{lem:norm_subgaussian} for random variables satisfying \eqref{eq:beta}.

\item {Section \ref{sec:smallest singular+non-centered} is devoted to the proof of Theorem \ref{thm:bernoulli}. We begin with extending Theorem \ref{thm: smallest singular + norm}, to matrices with non-centered random entries (see Theorem \ref{thm: smallest singular + norm + non-centered}).  To handle random variables with non-zero mean we need a folding trick, which we explain in detail. The rest of the proof of Theorem \ref{thm: smallest singular + norm + non-centered}  largely follows from that of Theorem \ref{thm: smallest singular + norm}. We provide a detailed outline about how to extend the results of Section \ref{sec: compressible} and Section \ref{sec: small LCD} to this more general setup. Finally we show that Theorem \ref{thm: smallest singular + norm + non-centered} can be appropriately adapted to prove Theorem \ref{thm:bernoulli}.}
 \end{itemize}

\section{Preliminaries and Proof Outline} \label{sec: proof outline}

Without loss of generality, we may assume that $p_n \le c(K+R)^{-2}$, for some small positive constant $c$, since for  larger values of $p_n$, the entries $a_{i,j}$ have variance bounded below by an absolute constant. In such case, Theorem \ref{thm: smallest singular + norm} follows from \cite{RV1}.

Since
\[
  s_{\min}(\bar{A}_n+D_n)=\inf_{x \in S^{n-1}} \norm{(\bar{A}_n+D_n)x}_2,
 \]
to prove Theorem \ref{thm: smallest singular + norm}, we need to find a lower bound on this infimum. For dense matrices this is done via decomposing the unit sphere into  {\em compressible} and {\em incompressible} vectors, and obtaining necessary bounds on the infimum on both of these parts separately (cf.~\cite{R, RV1, RV2, V}). To carry out the argument for sparse matrices we introduce another class of vectors which we call {\em dominated vectors}. Below we define the necessary concepts, and explain the necessity of the dominated vectors along with a outline of the proof.

\noindent
We start with the definition of compressible and incompressible vectors.
\begin{dfn}
Fix $m<n$. The set of $m$-sparse vectors is given by
 \[
  \text{Sparse}(m):=\{ x \in \R^{n} \mid |\text{supp}(x)| \le m \},
 \]
 where $|S|$ denotes the cardinality of a set $S$.
 Furthermore, for any $\delta>0$, the vectors which are $\delta$-close to $m$-sparse vectors in Euclidean norm, are called $(m, \delta)$-compressible vectors. The set of all such vectors, hereafter will be denoted by $\text{Comp}(m, \delta)$. Thus,
 \[
   \text{Comp}(m, \delta):= \{x \in S^{n-1} \mid \exists y \in \text{Sparse}(m) \text{ such that } \norm{x-y}_2 \le \delta \}.
  \]
  The vectors in $S^{n-1}$ which are not compressible, are defined to be incompressible, and the set of all incompressible vectors is denoted as $\text{Incomp}(m,\delta)$.
  \end{dfn}
 Next we define the dominated vectors. These are also close to sparse vectors, but in a different sense.
 \begin{dfn}
 For any $x \in S^{n-1}$, let $\pi_x: [n] \to [n]$ be a permutation which arranges the absolute values of the coordinates of $x$ in an non-increasing order. For $1 \le m \le m' \le n$, denote by $x_{[m:m']} \in \R^n$ the vector with coordinates
  \[
    x_{[m:m']}(j)=x(j) \cdot \mathbf{1}_{[m:m']}(\pi_x(j)).
  \]
  In other words, we include in $x_{[m:m']}$ the coordinates of $x$ which take places from $m$ to $m'$ in the non-increasing rearrangement.

  \noindent
  For $\alpha<1$ and $m \le n$ define the set of vectors with dominated tail as follows:
  \[
    \text{Dom}(m, \alpha):= \{ x \in S^{n-1} \mid\norm{x_{[m+1:n]}}_2 \le \alpha \sqrt{m} \norm{x_{[m+1:n]}}_{\infty} \}.
  \]
  \end{dfn}
  Note that by definition, $\text{Sparse}(m) \cap S^{n-1} \subset \text{Dom}(m, \alpha)$, since for $m$-sparse vectors, $x_{[m+1:n]}=0$.  We now provide an outline of the proof. For the ease of writing, hereafter, we will often drop the sub-script in $p_n$, and will write $p$ instead.

  \bigskip

 \noindent
  The proof of Theorem \ref{thm: smallest singular + norm} proceeds by first bounding the infimum over compressible and dominated vectors, and then the same for the incompressible vectors. As in \cite{RV1}, the first step is to control the infimum of $\norm{\bar{A}_n x}_2$ for sparse vectors (for clarity of explanation we take $D_n =0$ in rest of the section). This was done in \cite{RV1} using a small ball probability estimate, and an $\vep$-net argument (see \cite[Corollary 2.7]{RV1} and \cite[Proposition 2.5]{RV1}). However, the sparseness of the entries prevents us to use these techniques here. For example, adapting \cite[Proposition 2.5]{RV1} to the sparse set-up one can at best hope to obtain that for any fixed $x \in S^{k-1}$
  \[
  \P\left(\|{\tilde{A}_n x}\|_2 \le \eta \sqrt{np}\right) \le e^{-cnp},
  \] for a {\em tall} $n \times k$ matrix sparse matrix $\tilde{A}_n$, and for some $\eta, c>0$. However, when one tries to use the $\vep$-net argument, then it is clear we must have $k=O(np)$. Since in the sparse regime $p \ra 0$, this is not enough. Moreover to uplift the result for tall matrices to square matrices and sparse $x \in S^{n-1}$ one needs to take another union bound (see proof of \cite[Lemma 3.3]{RV1}), which also fails here.

  Instead, using Chernoff's bound we show that there are large submatrices inside $\bar{A}_n$ such that one part of those submatrices contain only one non-zero entry per row, and the rest of them are zero (see Lemma \ref{l: pattern}). This essentially means that $(\bar{A}_n x)_i$ is just $a_{i,j}x_j$, for some $j \ne i$, when $x$ is a sparse vector.
  Thus contributions of different coordinates of $x$ do not cancel, which allows to avoid using the $\vep$-net argument at this step.
   This is enough to control $\norm{\bar{A}_n x}_2$ for very sparse vectors. More specifically, this argument works for $O(p^{-1})$-sparse vectors of unit norm. These estimates automatically extend to compressible and dominated vectors with $m= O(p^{-1})$. To carry out the program, one needs to improve these estimates for $cn$-sparse vectors, for some $c \in (0,1)$. To this end, we need some estimates on the {\em small ball probability}. For such estimates,  the following definition of the L\'{e}vy concentration function turns out to be useful.

\begin{dfn}\label{dfn:levy}
Let $Z$ be random variable in $\R^n$. For every $\vep >0$, the L\'{e}vy concentration function of $Z$ is defined as
\beq
\cL(Z,\vep):=\sup_{u \in \R^n}\P(\norm{ Z-u}_2 \le \vep),\notag
\eeq
where $\norm{\cdot}_2$ denotes the Euclidean norm.
\end{dfn}
Once we obtain necessary estimates for $cn$-sparse vectors, we extend them for compressible and dominated vectors  using the $\vep$-net argument and the union bound.

 \bigskip

Next we need to bound the infimum for incompressible vectors. To this end, we need the following Lemma of \cite{RV1} (see \cite[Lemma 3.5]{RV1}).
 \begin{lem}[Invertibility via distance]  \label{l: via distance}
     For $j \in [n]$, let $\tilde{A}_{n,j} \in \R^n$ be the $j$-th column of $\tilde{A}_n$, and let $H_{n,j}$ be the subspace of $\R^n$ spanned by $\{\tilde{A}_{n,i}, i \in [n]\setminus\{j\}\}$. Then for any $\vep, \rho>0$, and $M<n$,
   \beq\label{eq:invertibility_distance}
    \P\left(  \inf_{x \in  \text{\rm Incomp}(M, \rho)} \norm{\tilde{A}_nx}_2 \le \vep \rho^2 \sqrt{\frac{p}{n}} \right)
    \le \frac{1}{M} \sum_{j=1}^n \P \left( \dist(\tilde{A}_{n,j},H_{n,j}) \le \rho\sqrt{p} \vep \right).
   \eeq
 \end{lem}
 \begin{rmk}\label{rmk:l via distance}
 Lemma \ref{l: via distance} can be extended to the case when the event on the \abbr{LHS} of \eqref{eq:invertibility_distance} is intersected with an event $\Omega$, and in that case Lemma \ref{l: via distance} continues to hold if the \abbr{RHS} of \eqref{eq:invertibility_distance} is replaced by intersecting each of the event under the summation sign with the same event $\Omega$. In the proof of Theorem \ref{thm: smallest singular + norm}, we will use this slightly more general version of Lemma \ref{l: via distance}. Since the proof this general version of Lemma \ref{l: via distance} is a straightforward adaptation of the proof of \cite[Lemma 3.5]{RV1}, we omit the details.
 \end{rmk}
Proceeding similarly as in \cite{RV1} we see that we need to find small ball probability estimates for incompressible vectors.  However, the small ball probability estimates used in the proof of Proposition \ref{p: dominated and compressible} is too weak for this purpose. The rich additive structure of the incompressible  vectors  is helpful here. For a vector $x \in \R^n$, when each coordinate of $x$ is rational, a suitable measure for the additive structure in $x$ is the least common multiple of the denominators of the coordinates. Generalizing this idea, when the coordinates of the vector $x$ are real, a notion termed as {\em least common denominator} (\abbr{LCD)} was introduced in \cite{RV1,RV2}, to capture the additive structure in $x$. In our current set-up of sparse matrices, adapting their definition, we have the following definition of \abbr{LCD}:
 \begin{dfn}\label{dfn:lcd}
  For $x \in S^{n-1}$, the \abbr{LCD} of $x$ is defined as
   \[
     D(x)
      := \inf \Big\{ \theta>0 : \dist(\theta x, \mathbb{Z}^n) < (\delta_0 p)^{-1/2} \sqrt{\log_+(\sqrt{\delta_0 p} \theta)}
             \Big\},
   \]
   where $\delta_0 \in (0,1)$ is an appropriate constant (see Remark \ref{rmk:choose_delta_0} below for the choice of $\delta_0$).
 \end{dfn}
  \begin{rmk}\label{rmk:choose_delta_0}
 We note that there exist  $\delta_0, \vep_0' \in (0,1)$, such that for any $\vep < \vep_0'$,  $\cL(\xi\delta, \vep) \le 1-\delta_0p$, where $\xi$ is a random variable with unit variance and finite fourth moment, and $\delta$ is a $\dBer(p)$ random variable, independent of each other (for more details see \cite[Lemma 3.3]{V}). We choose this $\delta_0$ in Definition \ref{dfn:lcd} above.
 \end{rmk}
 Using the \abbr{LCD} of a vector, one can improve the small ball probability estimates (cf.~ \cite[Theorem 6.3]{V}). Using this, and proceeding as in  \cite{RV1} vectors with large \abbr{LCD}  are taken care of. To deal with the vectors of small \abbr{LCD}, we split them into {\em level sets} first. Inside each level set we use the small ball probability estimate once again, and a careful $\vep$-net argument is carried out (based on the value of the \abbr{LCD} in that level set) to obtain necessary bounds. After which the result follows by a union bound.

In the dense set-up one can show that \abbr{LCD} on the set of incompressible vectors under consideration is $\Omega(\sqrt{n})$ (see \cite[Lemma 6.1]{R_lec_notes}). However, in the sparse set-up one cannot guarantee similar kinds of lower bounds on \abbr{LCD} due to weak control on the compressible vectors. To this end, we use a lower bound \abbr{LCD} depending on $\norm{\cdot}_\infty$ (see Proposition \ref{l: LCD bound}), demanding some control on $\norm{\cdot}_\infty$ on the incompressible vectors which requires the introduction of dominated vectors.

\section{Compressible and dominated vectors} \label{sec: compressible}
In this section we obtain a lower bound on the infimum of $\norm{(\bar{A}_n +D_n)x}_2$ over compressible and dominated vectors. More specifically, we will prove the following proposition in this section. Before stating the result let us recall that for any $\gamma \in \R$, $\lceil \gamma \rceil$ denotes the ceiling of $\gamma$, i.e.~it is smallest integer greater than or equal to $\gamma$.
\begin{prop}   \label{p: dominated and compressible}
  Let $p$ satisfy \eqref{eq:p_assumption}. Denote
  \[
   \ell_0= \left \lceil \frac{\log 1/(8 p)}{\log \sqrt{pn}} \right \rceil.
  \]
Let $\bar{A}_n$ be an $n \times n$ matrix with zeros on the diagonal and  off-diagonal entries $a_{i,j}=\delta_{i,j} \cdot \xi_{i,j}$, where $\delta_{i,j}$ are i.i.d. Bernoulli random variables with $\P(\delta_{i,j}=1)=p$, and $\xi_{i,j}$ are centered i.i.d. random variables with unit variance and finite fourth moment. Let $K, R \ge 1$, and assume that $D_n$ is a non-random diagonal matrix with real entries such that $\norm{D_n} \le R \sqrt{pn}$. Then there exist constants $0< c_{\ref{p: dominated and compressible}}, \ol{c}_{\ref{p: dominated and compressible}}, C_{\ref{p: dominated and compressible}},\ol{C}_{\ref{p: dominated and compressible}}, \wt{C}_{\ref{p: dominated and compressible}}< \infty$, depending only on $K, R$, and the fourth moment of $\{\xi_{ij}\}$, such that
for any $p^{-1} \le M \le  c_{\ref{p: dominated and compressible}} n$,
  \begin{align*}
   &\P(\exists x \in \text{\rm Dom}(M, (C_{\ref{p: dominated and compressible}}(K+R))^{-4})
    \cup \text{\rm Comp}(M, \rho)\\
    & \qquad  \norm{(\bar{A}_n+D_n)x}_2 \le \ol{C}_{\ref{p: dominated and compressible}}(K+R) \rho \sqrt{np}
   \text{ and } \norm{\bar{A}_n} \le K  \sqrt{pn})\le \exp(- \ol{c}_{\ref{p: dominated and compressible}} pn),
  \end{align*}
  where $ \rho=(\wt{C}_{\ref{p: dominated and compressible}}(K+R))^{-\ell_0-6}$.
 \end{prop}
 
 The proof splits into two steps. First, we consider vectors which are close to $(1/8p)$-sparse. As explained above, for such vectors, the small ball probability bound is too weak, which forces us to use a method specially designed for sparse matrices. At the second step of the proof, we consider vectors which are close to $M$-sparse, but not to $(1/8p)$-sparse. For such moderately sparse vectors, a better control of the L\'{e}vy concentration function is available.

\subsection{Vectors close to very sparse}  \label{subsec: very sparse}
We first establish a uniform lower bound for $\norm{(\bar{A}_n+ D_n)x}_2$ over the sets of unit vectors which are close to $(1/8p)$-sparse.
Our approach is based on the observation that for any such vector there is a relatively large number of rows of $\bar{A}_n$ which have exactly one non-zero entry in the columns corresponding to its support.
Unfortunately, this number is insufficient to use the union bound over all supports.
This forces us to use a simple chaining type argument.
The support of the vector is divided into blocks of increasing sizes.
We use one of these blocks carrying a substantial part of the $\ell_2$ norm of the vector to obtain the small ball probability bound, and show that the contribution of the other blocks does not destroy it.

To run this procedure efficiently, we need a combinatorial lemma about the structure of the set of rows having exactly one non-zero entry in the columns corresponding tho the chosen block. To this end, we divide the set of these columns in two parts, and look for those rows for which the first part has exactly one non-zero entry, and the second one has only zeros. Such zero rows would be useful in showing that the contributions of different coordinates within the selected block add up correctly. For ease of writing, for any positive integer $\gamma \le n$, let us denote $\binom{[n]}{\gamma}$ to be the collection of all subsets of $[n]$ of cardinality $\gamma$. Now we are ready to state the combinatorial lemma.

 \begin{lem}  \label{l: pattern}
 Let $\bar{A}_n$ be an $n \times n$ matrix with zeros on the diagonal, and has off-diagonal entries $a_{i,j}=\delta_{i,j}  \xi_{i,j}$, where $\delta_{i,j}$ are i.i.d. Bernoulli random variables with $\P(\delta_{i,j}=1)=p$, where $p$ satisfies \eqref{eq:p_assumption}, and $\xi_{i,j}$ are centered i.i.d. random variables with $\max\{\P(\xi_{i,j}\ge1),\P(\xi_{i,j}\le-1)\} \ge c_0$ for some positive constant $c_0$. 
For $\kappa \in \N$ and for $J, J' \subset[n]$, let ${\cA}_c^{J,J'}$ denotes the event that 
there are at least $c\kappa pn$ rows of the matrix $\bar{A}_n$ containing exactly one  non-zero entry $a_{i,j}$ in the columns corresponding to $J$, for which $|a_{i,j}| \ge1$, and all zero entries in the columns corresponding to $J'$.  Denote
 \[
   \gm=\gm(\kappa):=\kappa \sqrt{pn} \wedge \frac{1}{8p}.
  \]
\noindent
Then, there exist constants $0< c_{\ref{l: pattern}},\ol{c}_{\ref{l: pattern}}< \infty,$  depending only on $c_0$, such that

\[
\P\Bigg( \bigcap_{\kappa \le (8p \sqrt{pn})^{-1} \vee 1} \ \bigcap_{J \in \binom{[n]}{\kappa}} \ \bigcap_{J' \in \binom{[n]}{\gm}, \, J \cap J'=\emptyset} {\cA}^{J, J'}_{\ol{c}_{\ref{l: pattern}}}\Bigg) \ge 1- \exp (-c_{\ref{l: pattern}} p n).
\]
 \end{lem}

 Before going to the proof let us mention that we will often write $\gamma$ instead of $\lfloor  \gamma \rfloor$ (the floor of $\gamma$, i.e.~the largest integer less than or equal to $\gamma$), even when $\gamma$ is not an integer. This will not make any changes in the proof. We adopt this approach to simplify the presentation.

\begin{proof}
 Fix $\kappa \le (8p \sqrt{pn})^{-1} \vee 1$ and a set $J \in \binom{[n]}{\kappa}$. Let $I^1(J)$ be the set of all rows of $\bar{A}_n$ containing exactly one large entry in the columns corresponding to $J$:
 \[I^{1}(J):= \Big\{i \in [n]: |a_{i,j_i}|  \ge 1 \text{ for some } j_i \in J, \text { and } a_{i,j}=0 \text{ for all } j \in J \setminus \{j_i\}\Big\}.\]
 
Similarly for a set $J' \in \binom{[n]}{\gm}$ we define
\[I^0(J'):= \Big\{ i \in [n]: a_{i,j}=0 \text{ for all } j \in J'\Big\}. \]
 To prove the desired result we first show that the cardinality of the subset $I^{1}(J)$ must be at least $c \kappa pn$ with large probability, for some positive constant $c$. Then using Chernoff's bound we argue that $I:=|I^{1}(J)\cap I^0(J')|$ is also large with large probability. Finally taking union bounds over the set of choices of $J$, and over $\kappa$, we complete the proof. 

 \noindent
 To this end, we begin by obtaining a lower bound on $\P(i \in I^1(J))$ for every $i \in [n]$. Recall that the diagonal entries of $\bar{A}_n$ are zero, and therefore we need to consider the two cases $i \in [n] \setminus J$, and $i \in J$ separately. 
 
 \noindent
 Now, by the independence of the random variables $\{\delta_{i,j}\}$, and $\{\xi_{i,j}\}$, and the fact that $\max\{\P(\xi_{i,j}\ge1),\P(\xi_{i,j}\le-1)\} \ge c_0$, it follows that, for every $i \in [n] \setminus J$,
  \beq\label{eq: i in I_1(J)}
  \P(i \in I^{1}(J))  \ge c_0|J| \cdot p (1-p)^{|J|-1}
  \ge c_0\kappa p (1-\kappa p) \ge \frac{c_0}{2} \kappa p.
 \eeq
 Similarly for every $i \in J$, whenever $|J|=\kappa\ge 3$, we also have that
  \[
  \P(i \in I^{1}(J)) = \P(i \in I^1(J\setminus\{i\})) \ge c_0(|J|-1) \cdot p (1-p)^{|J|-2}
  \ge c_0(\kappa-1) p (1-\kappa p) \ge \frac{c_0}{2} \kappa p.
 \]
 When $|J|=2$, one can again show that $\P(i \in I^{1}(J)) \ge c_0 p = \f{c_0}{2}\kappa p$,  for any $i \in J$.
 Therefore, applying Chernoff's inequality, whenever $\kappa \ge 2$, we obtain
 \beq
  \P(|I^{1}(J)|\le \frac{c_0}{4}\kappa pn) \le \exp(-c_1\kappa pn),\label{eq:conc_small_set}
 \eeq
 for some positive finite constant $c_1$. For $|J|=\kappa=1$, we note that $J \cap I^1(J)=\emptyset$. Therefore shrinking $c_1$ if necessary, and applying Chernoff's inequality again, we also obtain that
 \[
  \P(|I^{1}(J)|\le \frac{c_0}{4} p(n-1))\le \P(|I^{1}(J)|\le \frac{c_0}{3} pn) \le \exp(-c_1 pn).
 \]
This establishes \eqref{eq:conc_small_set} for all values of $\kappa$. Next for a fixed set $J' \in \binom{[n]}{\gm}$, and for any $i \in [n]\setminus J'$, we have that
 \beq\label{eq: i in I_0(J)}
   \P(i \in I^0(J'))=(1-p)^{|J'|} \ge 1-p \cdot |J'| =1-p \cdot  \gm \ge \frac{3}{4}.
 \eeq
 Similarly, for $i \in J'$,
  \[
   \P(i \in I^0(J'))=(1-p)^{|J'|-1} \ge 1-p \cdot  (\gm-1) \ge \frac{3}{4}.
 \]
 Thus, for a given $I \subset [n]$, the random variable $|I \setminus I^0(J')|$ can be represented as the sum of independent Bernoulli variables taking value 1 with probability either $q_1$ or $q_2$, where $\max\{q_1,q_2\} \le  p \gm$. Note that $\E |I \setminus I^0(J')| \le  p \gm \cdot |I| \le |I|/4$ by the assumption on $\kappa,$ and $\gm$.
 Hence, by Chernoff's inequality
 \[
  \P(|I \setminus I^0(J')| \ge \frac{1}{2}|I| )
  \le \exp \left( - \frac{ |I|}{16} \log \left( \frac{1}{4 p \gm} \right)\right).
 \]
 Therefore, for any $I \subset [n]$ such that $|I| \ge \frac{c_0}{4}\kappa np$, we deduce that
 \begin{align*}
  & \P \Big( \exists J' \in \binom{[n]}{\gm} \text{ such that }
         |I^0(J') \cap I| \le \frac{c_0}{8} \kappa pn \Big) \\
   \le& \sum_{J' \in \binom{[n]}{\gm}} \P (|I \setminus I^0(J')| \ge \f{1}{2}|I| )
   \\
   \le& \binom{n}{ \gm} \cdot \exp \left( - \frac{ |I|}{16} \log \left( \frac{1}{4 p \gm} \right)\right)
   \le  \exp \left( \gm \cdot \log \left( \frac{en}{\gm} \right) - \frac{c_0}{64}\kappa pn \cdot \log \left( \frac{1}{4 p \gm} \right) \right) = \exp (- \kappa p n \cdot U),
 \end{align*}
 where
\[
 U:= \f{c_0}{64}\log\left(\f{1}{4p \gm}\right) - \f{\gm}{\kappa p n} \log \left(\f{en}{\gm}\right).
 \]
  We claim that $U \ge c_0/100$. To prove this, consider two cases. First, assume that $p \ge \frac{1}{4}n^{-1/3}$. In this case, $\kappa=1$ and $\gm=\frac{1}{8p}$. Therefore, for all large $n$,
 \[
  U=\frac{c_0}{64} \log 2- \frac{1}{8 p^2 n} \log (en \cdot 8p)
  \ge \frac{c_0}{64} \log 2- 2 n^{-1/3} \cdot \log (8e n)
  \ge \frac{c_0}{100},
 \]
 where the first inequality holds by the assumption on $p$.

 Now, assume that $\frac{C_{\ref{thm: smallest singular + norm}}\log n}{n} \le p \le \frac{1}{4}n^{-1/3}$. Then
   $1 \le \kappa \le  \frac{1}{8p \sqrt{pn}}$, and $\gm=\kappa \sqrt{pn}$.
 Denote
 \[
   \alpha= \frac{1}{4 \kappa p \sqrt{pn}}.
 \]
 The assumption on $\kappa$ implies that $\alpha \ge 2$. Hence,
 \begin{align*}
  U
  &=  \frac{c_0}{64}   \log \left( \frac{1}{4 \kappa p\sqrt{pn}}\right)
    - \frac{1}{\sqrt{pn}}  \log \left(\frac{en}{\kappa \sqrt{pn}} \right) \\
  &=  \frac{c_0}{64}  \log \alpha -  \frac{1}{\sqrt{pn}}  \log (4e pn \alpha) \\
  &=\frac{c_0}{64}\log\alpha-\f{1}{\sqrt{pn}}\log \alpha -\f{1}{\sqrt{pn}}\big(\log(4e)+ \log (pn)\big).
 \end{align*}
 Now noting that by the assumption on $p$ we have $pn \ra \infty$ as $n \ra \infty$, and using the fact that $x^{-1/2}\log x \ra 0$ as $x \ra \infty$, we conclude that $U \ge \f{c_0}{100}$, for all large $n$.
 This proves that, for any $I \subset [n]$ with $|I|\ge \f{c_0}{4}\kappa p n$, we have
 \beq
   \P \left( \exists J' \in \binom{[n]}{ \gm} \text{ such that } |I^0(J') \cap I| \le \frac{c_0}{8} \kappa pn \right)
   \le \exp (-c_2 \kappa pn),\notag
 \eeq
for some positive finite constant $c_2$. Now for a set $J \in \binom{[n]}{\kappa}$ define
 \[
  p_J:=\P \left( \exists J' \in \binom{[n]}{\gm} \text{ such that } J' \cap J= \emptyset, \
        |I^{1}(J) \cap I^0(J')| < \frac{c_0}{8} \kappa pn \right).
 \]
Since $J$ and $J'$ are disjoint, it is easy to note that the random subsets $I^{1}(J)$ and $I^0(J')$ are independent. Using \eqref{eq:conc_small_set} this now implies that
 \begin{align}
    p_J
    &\le \sum_{I \subset [n],\ |I| \le \frac{c_0}{4}\kappa pn } \P(I^{1}(J)=I) \notag\\
    &+ \sum_{I \subset [n] ,\ |I| > \frac{c_0}{4}\kappa pn } \P(I^{1}(J)=I)
        \P \Big( \exists J' \in \binom{[n]}{\gm} \text{ such that }
         |I^0(J') \cap I| \le \frac{c_0}{8} \kappa pn \Big)
        \notag\\
    &\le \P(|I^{1}(J)|\le \frac{c_0}{4}\kappa pn) +
     \exp(-c_2\kappa pn) \sum_{I \subset [n] ,\ |I| > \frac{c_0}{4}\kappa pn } \P(I^{1}(J)=I) \le \exp(-c_3 \kappa p n),\label{eq:conc_large_set}
 \end{align}
for all large $n$, where $c_3$ is another positive constant.

The rest of the proof consists of taking the union bounds.
 First, using the union bound over $J \in \binom{[n]}{\kappa}$, setting $\ol{c}_{\ref{l: pattern}}=c_0/8$, and enlarging $\ol{C}_{\ref{thm: smallest singular + norm}}$, if needed, we get that
 \[\P \bigg(\bigcup_{J \in \binom{[n]}{\kappa}} \ \bigcup_{J' \in \binom{[n]}{\gm}, \, J \cap J'=\emptyset} ({\cA}^{J, J'}_{\ol{c}_{\ref{l: pattern}}})^c \bigg) \le \binom{n}{\kappa} \exp(-\bar{c}_3 \kappa p  n) \le \exp( \kappa \log n - {c}_3 \kappa p n) \le \exp( - c_3' \kappa p n),\]
 for some positive finite constant $c_3'$. 
 The last inequality here follows from assumption \eqref{eq:p_assumption}.
 Finally taking another union bound over $\kappa$ we obtain the desired result.
 \end{proof}


\medskip
We  use  Lemma \ref{l: pattern} to establish a uniform small ball probability bound for the set of
dominated vectors.
Before formulating the result, note that the condition $p \le c(K+R)^{-2}$ introduced at the beginning of Section \ref{sec: proof outline} ensures that $1/(8p) >1$.

 \begin{lem}   \label{l: sparse vectors}
  Let $\bar{A}_n$ be the matrix defined in Proposition \ref{p: dominated and compressible}, and let $p$ satisfy \eqref{eq:p_assumption}.
     Denote
  \beq\label{eq:ell_0_dfn}
   \ell_0= \left \lceil \frac{\log 1/(8 p)}{\log \sqrt{pn}} \right \rceil.
  \eeq
 Fix $K,R \ge 1$, and let $D_n$ be a real diagonal matrix with $\norm{D_n} \le R \sqrt{n p}$. 
 Then there exist constants $0< c_{\ref{l: sparse vectors}}, C_{\ref{l: sparse vectors}}, \wt{C}_{\ref{l: sparse vectors}}< \infty$, depending only on the fourth moment of $\{\xi_{i,j}\}$, such that
  \begin{align*}
   \P&\Big(\exists x \in \text{\rm Dom}\big((8p)^{-1}, (C_{\ref{l: sparse vectors}}(K+R))^{-1}\big) \text{ such that } \norm{(\bar{A}_n+D_n)x}_2 \le (\wt{C}_{\ref{l: sparse vectors}}(K+R))^{-\ell_0} \sqrt{np} \\
   &\hskip 3in \text{ and } \norm{\bar{A}_n} \le K \sqrt{pn}\Big) \\
   &\le \exp(-c_{\ref{l: sparse vectors}}pn).
  \end{align*}
 \end{lem}

\vskip10pt


\noindent
 \begin{proof}
We first prove the result for $\text{\rm Sparse}((8p)^{-1})$ vectors of unit norm, and then we show that the estimates are automatically extended to $\text{\rm Dom}((8p)^{-1}, (C(K+R))^{-1})$ vectors, for some large constant $C$. Our proof strategy for sparse vectors depends on $p$. If $p \ge (1/4) n^{-1/3}$, we apply  Lemma \ref{l: pattern} with $\kappa=1$ and $\gm=\frac{1}{8p}$. The range $p \le (1/4) n^{-1/3}$ requires a different approach since the we cannot reach the level of sparsity $O(p^{-1})$ in one step. Instead we use Lemma \ref{l: pattern} with different values of  $\kappa$ depending on the distribution of coordinates of the vector. Assuming that the event described in this lemma occurs, we split the vector into blocks with disjoint support. One of these blocks has a  large $\ell_2$ norm. By the assertion of Lemma \ref{l: pattern}, a large number of rows of the matrix $\bar{A}_n$ have exactly one non-zero entry in columns corresponding to the support of this block. This will be enough to conclude that $\norm{(\bar{A}_n+D_n)x}_2$ is bounded below, for $x \in \text{\rm Sparse}((8p)^{-1})$. Note that while applying Lemma \ref{l: pattern} we need $\max \{\P(\xi_{i,j} \ge 1), \P(\xi_{i,j} \le -1)\} \ge c_0$. Since $\xi_{i,j}$'s are centered and have unit variance it is easy to check that $\max \{\P(\xi_{i,j} \ge 1/2), \P(\xi_{i,j} \le -1/2)\} >0$, and therefore without loss of generality we can work with a scaled version of $\xi_{i,j}$. Since the fourth moment of $\xi_{i,j}$'s are bounded, upon an application of the Paley-Zygmund inequality (see \cite[Lemma 3.5]{LPRT}), we further obtain a uniform lower bound on the value of $c_0$.

\noindent
We now begin with large values of $p$, that is, $p \ge (1/4)n^{-1/3}$.
   In this case, $\ell_0=1$, and we prove that there exist constants $\wt{c}_0$ and $c_0'$ such that 
  \begin{align}
   \P&(\exists x \in \text{\rm Sparse}((1/8p))\cap S^{n-1} \text{ such that } \norm{(\bar{A}_n+D_n)x}_2 \le \sqrt{\wt{c}_0 np}
   \text{ and } \norm{\bar{A}_n} \le K \sqrt{pn}) \notag\\
   &\qquad \le \exp(-c_0'pn). \label{eq:sparse_large_p}
  \end{align}
  For $k \in [n]$, set $J_k=\{k\}$ and $J_k'= \supp(x) \setminus J_k$. Let $\cA$ be the event that for each $k \in [n]$ there exists a set $I_k \subset [n]$ of rows such that $|I_k|=\ol{c}_{\ref{l: pattern}} p n$, and for any $i \in I_k$, $|a_{ik}|  \ge 1$ and $a_{ij}=0$ for $j \in \supp(x) \setminus\{k\}$.
   The definition of the sets $I_k$ immediately implies that $I_k \cap I_{k'} =\emptyset$ for $k \ne k'\in \supp(x)$.
   By Lemma \ref{l: pattern}, $\P(\cA) \ge 1- \exp(-c_{\ref{l: pattern}}pn)$.
   This shows that on this large set $\cA$, we have that
   \beq\label{eq:norm_bound_A_n+D_n}
    \norm{(\bar{A}_n +D_n)x}_2^2 \ge  \sum_{k \in \supp(x)} \sum_{i \in I_k} \Big|((\bar{A}_n +D_n)x)_i\Big|^2.
   \eeq
To get rid of the diagonal matrix $D_n$, let us consider only the coordinates $i \in I_k\setminus \supp(x)$. For these coordinates, $((\bar{A}_n+D_n)x)_i=(\bar{A}_n x)_i$.
The assumption on $p$ implies $|I_k| \gg |\supp(x)|=O(p^{-1})$, and so $|I_k\setminus \supp(x)| \ge \f{\ol{c}_{\ref{l: pattern}}pn}{2}$.
Hence,
   \beq\label{lower_bd_norm_A_n}
    \norm{(\bar{A}_n +D_n)x}_2^2 \ge
      \sum_{k \in \supp(x)} \sum_{i \in I_k\setminus \supp(x)} |(\bar{A}_n x)_i|^2\ge \sum_{k \in \supp(x)} \f{\ol{c}_{\ref{l: pattern}}pn}{2} |x(k)|^2
    =\f{\ol{c}_{\ref{l: pattern}}pn}{2}.
   \eeq
 Thus, setting $c_0'= c_{\ref{l: pattern}}$, and $\wt{c}_0=\f{\ol{c}_{\ref{l: pattern}}}{2}$ we have \eqref{eq:sparse_large_p}.
This estimate can be automatically extended to the set $\text{Dom}((8p)^{-1}, (C(K+R))^{-1})$ provided that the constant $C$ is large enough. Indeed, assume that
  \beq\label{eq:dominated_large_p_contradiction}
    \norm{(\bar{A}_n+D_n)x}_2 <  \frac{1}{2}\sqrt{\wt{c}_0pn}
  \eeq
  for some $x \in \text{Dom}((8p)^{-1}, (C(K+R))^{-1})$. Set $m=(8p)^{-1}$. Since $x \in S^{n-1}$, it is easy to note that $\norm{x_{[m+1:n]}}_{\infty} \le m^{-1/2}$. Hence,
  \[
    \norm{x_{[m+1:n]}}_2 \le (C(K+R))^{-1} \sqrt{m} \norm{x_{[m+1:n]}}_{\infty}
    \le (C(K+R))^{-1},
  \]
  and therefore
  \begin{align*}
    \norm{(\bar{A}_n +D_n) x_{[1:m]}}_2
    &\le \norm{(\bar{A}_n +D_n)x}_2
    + (\norm{\bar{A}_n}+ R \sqrt{np})  \norm{x_{[m+1:n]}}_2 \\
    &< \frac{1}{2} \sqrt{\wt{c}_0pn}+ (K+R) \sqrt{pn} \cdot (C(K+R))^{-1}
    <\f{3}{4} \sqrt{\wt{c}_0pn},
  \end{align*}
  when $C \ge \f{4}{\sqrt{\wt{c}_0}}$. Furthermore
  \begin{align*}
 \bigg| \norm{(\bar{A}_n +D_n) ({x_{[1:m]}}/{\norm{x_{[1:m]}}_2})}_2 - \norm{(\bar{A}_n +D_n) x_{[1:m]}}_2\bigg| & \le   (K+R) \Big|1-\norm{x_{[1:m]}}_2\Big|\\
 & \le \f{\sqrt{\wt{c}_0pn}}{4}.
  \end{align*}
  Since $x_{[1:m]}/\norm{x_{[1:m]}}_2 \in \text{Sparse}((8p)^{-1})\cap S^{n-1}$, combining the above steps we note that the inequality in \eqref{eq:dominated_large_p_contradiction} holds only in $\cA^c$. Therefore, setting $C_{\ref{l: sparse vectors}}= \wt{C}_{\ref{l: sparse vectors}}=\f{4}{\sqrt{\wt{c}_0}}$, and $c_{\ref{l: sparse vectors}}= c_{\ref{l: pattern}}$, we prove the lemma for $p\ge (1/4)n^{-1/3}$.
  \vskip 0.2in

\noindent
We now consider the more difficult case, $\frac{\ol{C}_{\ref{thm: smallest singular + norm}} \log n}{n} \le p < (1/4) n^{-1/3}$.  Note that for such values of $p$,
  \[
    \frac{1}{8 p \sqrt{pn}} >1.
  \]
  To simplify the notation in the proof below, assume in addition that  $(pn)^{\ell_0/2}=\frac{1}{8p}$, i.e. the integer part in the definition of $\ell_0$ is redundant.

  Consider  $ x \in \text{\rm Dom}((8p)^{-1},(C(K+R))^{-1})$.
  Let us rearrange the magnitudes of the coordinates of $x$ and group them in blocks of lengths $(pn)^{\ell/2}$, where $\ell=1 \etc l_0$.
  More precisely, set
  \[
    z_\ell=x_{[(pn)^{(\ell-1)/2}+1: (pn)^{\ell/2}]}\footnote{when $\ell =1$ by a slight abuse of notation we take $z_1=x_{[1: \sqrt{np}]}$.},
   \]
   and
   \[
    z_{\ell_0+1}=x_{[(pn)^{\ell_0/2}+1 : n]}.
   \]
   For simplicity of notation denote $m=(8p)^{-1}=(pn)^{\ell_0/2}$.
   Let us show that one of the blocks $z_1 \etc z_{\ell_0}$ has a substantial $\ell_2$ norm.
    Note that
   \begin{align}
    \norm{z_{\ell_0+1}}_2
    \le (C_{\ref{l: sparse vectors}}(K+R))^{-1} \sqrt{m} \norm{z_{\ell_0+1}}_{\infty}
    &\le \sqrt{2} (C_{\ref{l: sparse vectors}}(K+R))^{-1}  \norm{x_{[m/2:m]}}_2 \notag\\
    &\le \sqrt{2} (C_{\ref{l: sparse vectors}}(K+R))^{-1}  \norm{z_{\ell_0}}_2, \label{eq: z_l_0+1}
   \end{align}
   where in the last step we use the fact that $n p \ra \infty$, as $n \ra \infty$, and so the support of $z_{\ell_0}$ contains that of $x_{[m/2:m]}$.
   As $x \in S^{n-1}$ implies $\sum_{\ell=1}^{\ell_0+1} \norm{z_{\ell}}_2^2 =1$, we have
   \[
   \sum_{\ell=1}^{\ell_0} \norm{z_{\ell}}_2^2  \ge 1 - 2(C_{\ref{l: sparse vectors}}(K+R))^{-2}.
   \]
On the other hand, for any $K \ge 1$ and $R \ge 0$, if $C_{\ref{l: sparse vectors}} >2$ then
$3\sum_{\ell=1}^\infty (C_{\ref{l: sparse vectors}} (K+R))^{-2\ell} <1$. Thus

    \[
      \sum_{\ell=1}^{\ell_0} (C_{\ref{l: sparse vectors}} (K+R))^{-2\ell}< \sum_{\ell=1}^{\ell_0} \norm{z_\ell}_2^2,
    \]
    which implies that there exists $\ell \le \ell_0$ such that $\norm{z_\ell}_2 \ge (C_{\ref{l: sparse vectors}} (K+R))^{-\ell}$. Let $\ell_\star$ be the largest index having this property, and set $u=\sum_{\ell=1}^{\ell_\star} z_\ell, \ v=\sum_{\ell=\ell_\star+1}^{\ell_0+1} z_\ell$. First consider the case when $\ell_\star < \ell_0$. Then by the triangle inequality and \eqref{eq: z_l_0+1}, we have that
  \[
   \norm{v}_2 \le \sum_{m=\ell_\star+1}^{\ell_0+1} \norm{z_m}_2 \le 2\sqrt{2} ({C_{\ref{l: sparse vectors}}} (K+R))^{-(\ell_\star+1)}.
  \]
Let $\kappa=(pn)^{(\ell_\star-1)/2}$.
Note that
    \[
     \kappa \le (np)^{(\ell_0-1)/2} \le \frac{1}{8 p  \sqrt{pn}}.
    \]
We will apply Lemma \ref{l: pattern} with this choice of $\kappa$.
  Split the support of $u$ into $(pn)^{1/2}$ blocks of equal size $\kappa$.
 To this end, define $L_{\ell_\star}:= \pi_x^{-1}([1, (np)^{\ell_\star/2}])$, where $\pi_x$ is the permutation of absolute values of the coordinates of $x$ in an non-increasing order. For $s \in [(pn)^{1/2}]$, define $J_s:=\pi_x^{-1}([(s-1)\kappa+1, s\kappa])$, and set $J_s'=L_{\ell_\star} \setminus J_s$. Since $|J_s'| \le |L_{\ell_\star}| = \kappa \sqrt{pn}$, we apply Lemma \ref{l: pattern} to get a set $\cA$ with large probability, such that on $\cA$, there exists subset of rows $I_s$ with $|I_s| \ge \ol{c}_{\ref{l: pattern}} \kappa p n$ for all $s \in [\sqrt{pn}]$, such that for every $i \in I_s$, we have $|a_{i,j_0}|  \ge 1$ for only one index $j_0 \in J_s$ and $a_{i,j}=0$ for all $j \in J_s \cup J_s' \setminus \{j_0\}$. It can further be checked that $I_1,I_2,\cdots,I_{\sqrt{pn}}$ are disjoint subsets. Therefore, on set $\cA$ for any $i \in I_s$,
  \[
    |(\bar{A}_nu)_i|=|a_{i,{j_0}}u(j_0)|
    = |a_{i,j_0}| \cdot |u(j_0)| \ge |x(\pi^{-1}_x(s\kappa))|.
  \]
Here  we used  that $\pi_x$ is a non-increasing rearrangement. Now note that for $i \notin \text{supp}(u)$,
  \[
  ((\bar{A}_n+D_n)u)_i = (\bar{A}_nu)_i, \text{ and } \text{supp}(u)= \kappa \sqrt{np} \ll \ol{c}_{\ref{l: pattern}} \kappa np,
  \]
  as long as  $np \ra \infty$. Therefore,
  \begin{align}
    \norm{(\bar{A}_n+D_n)u}_2^2
    \ge \sum_{s=1}^{(pn)^{1/2}} \sum_{i \in I_s\setminus \supp(u)} \big( (\bar{A}_n u)_i \big)^2
   & \ge \f{\ol{c}_{\ref{l: pattern}}pn}{2} \sum_{s=1}^{(pn)^{1/2}}\kappa (x(\pi_x^{-1}(s\kappa)))^2 \notag\\
    &\ge \f{\ol{c}_{\ref{l: pattern}}pn}{2} \sum_{k=(pn)^{(\ell_\star-1)/2}}^{(pn)^{\ell_\star/2}} (x(\pi_x^{-1}(k)))^2 \notag \\
    &= \frac{\ol{c}_{\ref{l: pattern}}pn}{2} \norm{z_{\ell_\star}}_2^2  \ge \f{\ol{c}_{\ref{l: pattern}}pn}{2} \cdot ({C_{\ref{l: sparse vectors}}} (K+R))^{-2 \ell_\star},  \label{i: z_l}
\end{align}
  where the third inequality uses monotonicity of the sequence $\{|x(\pi_x^{-1}(k))|\}_{k=1}^n$.
  Combining this with the bound on $\norm{v}_2$, on the set $\cA$, we get that
  \begin{align*}
      \norm{(\bar{A}_n+D_n)x}_2
      &\ge     \norm{(\bar{A}_n+D_n)u}_2 - \norm{\bar{A}_n+D_n} \cdot \norm{v}_2  \\
      &\ge \sqrt{\f{\ol{c}_{\ref{l: pattern}}pn}{2} } ({C_{\ref{l: sparse vectors}}} (K+R))^{- \ell_\star}-  (K+R) \sqrt{pn} \cdot 2\sqrt{2} (C_{\ref{l: sparse vectors}} (K+R))^{-(\ell_\star+1)} \\
     & \ge \sqrt{pn}  (\wt{C}_{\ref{l: sparse vectors}} (K+R))^{- \ell_\star}\sqrt{pn},
  \end{align*}
  where the last inequality follows if the constants $C_{\ref{l: sparse vectors}}, \wt{C}_{\ref{l: sparse vectors}}$ are chosen large enough independently of $\ell_\star$.

  \noindent
Now it remains to consider the case when $\ell_\star=\ell_0$.  Note that in this case, using \eqref{i: z_l}, we have that
\[\norm{(\bar{A}_n+D_n)u}_2 \ge \sqrt{\f{\ol{c}_{\ref{l: pattern}}pn}{2} }  \norm{z_{\ell_0}}_2,\]
and from \eqref{eq: z_l_0+1}, we have $\norm{v}=\norm{z_{\ell_0+1}} \le \sqrt{2} (C_{\ref{l: sparse vectors}}(K+R))^{-1}  \norm{z_{\ell_0}}_2$. Now proceeding similarly as before, on $\cA$, we obtain that
\[\norm{(\bar{A}_n+D_n)x}_2 \ge \sqrt{pn}  (\wt{C}_{\ref{l: sparse vectors}} (K+R))^{- \ell_0}\sqrt{pn}.\]
Since by Lemma \ref{l: pattern}, $\P(\cA) \ge 1 - \exp(-c_{\ref{l: pattern}} pn)$, the proof is completed.
  \end{proof}

\medskip
\noindent
We now extend the result of Lemma \ref{l: sparse vectors} to compressible vectors. This step requires only simple approximation. Recall that $\text{Sparse}((8p)^{-1})\cap S^{n-1} \subset \text{Dom}((8p)^{-1}, (C_{\ref{l: sparse vectors}}(K+R))^{-1})$.
  \begin{lem}  \label{l: compressible}
  Let $\bar{A}_n$ be the matrix defined in Proposition \ref{p: dominated and compressible}, $p$ satisfy \eqref{eq:p_assumption}. Fix $K, R \ge 1$, and let $D_n$ a non-random diagonal matrix with real entries such that $\norm{D_n} \le R \sqrt{np}$. Set
  \[
    \rho:=(\wt{C}_{\ref{l: sparse vectors}} (K+R))^{-(\ell_0+1)},
  \]
 where $l_0$ is defined in \eqref{eq:ell_0_dfn}. Then
  \begin{align*}
   \P&\Big(\exists x \in \text{\rm Comp}((8p)^{-1}, \rho) \text{ such that } \norm{(\bar{A}_n +D_n)x}_2 \le \frac{\wt{C}_{\ref{l: sparse vectors}} (K+R) \rho}{2} \sqrt{np}\\
  &\hskip 5in \text{ and } \norm{\bar{A}_n} \le K \sqrt{pn}\Big) \\
   &\le \exp(-c_{\ref{l: sparse vectors}}pn).
  \end{align*}
  \end{lem}

  \begin{proof}
   Denote
   \[
   \Om_\rho:= \bigg\{\forall x \in \text{Sparse}(1/(8p))\cap S^{n-1} \ 
    \norm{(\bar{A}_n+D_n)x}_2 \ge \rho  \wt{C}_{\ref{l: sparse vectors}}(K+R) \sqrt{pn}
    \text{ and } \norm{\bar{A}_n} \le K \sqrt{pn} \bigg\}.
   \]
Then on the set $\Om_\rho$, for any $\bar{x} \in \text{Comp}((8p)^{-1}, \rho)$, we can find $x \in \text{Sparse}(1/(8p))$ such that $\norm{(\bar{A}_n+D_n)({x}/{\norm{x}_2})}_2 \ge \rho  \wt{C}_{\ref{l: sparse vectors}}(K+R) \sqrt{pn}$, and $\norm{x-\bar{x}}_2 \le \rho$. This also implies $|1-\norm{x}_2| \le \rho$. Therefore
   \begin{align*}
     \norm{(\bar{A}_n +D_n)\bar{x}}_2
      & \ge \norm{(\bar{A}_n +D_n)(x/\norm{x}_2)}_2 - \norm{\bar{A}_n+D_n}\norm{x - \f{x}{\norm{x}_2}}_2- \norm{\bar{A}_n+D_n}  \norm{x-\bar{x}}_2 \\
     & \ge \f{ \rho  \wt{C}_{\ref{l: sparse vectors}}(K+R) }{2}\sqrt{pn},
   \end{align*}
when $\wt{C}_{\ref{l: sparse vectors}} >4$. Since by Lemma \ref{l: sparse vectors}, $\P(\Om_\rho) \ge 1- \exp(-c_{\ref{l: sparse vectors}} pn)$, the result follows.
 \end{proof}

 \medskip

 \subsection{Vectors close to moderately sparse} \label{subsec: moderately sparse}
Lemma \ref{l: sparse vectors}, and Lemma \ref{l: compressible} provide uniform lower bound on $\norm{(\bar{A}_n+D_n)x}_2$ for vectors which are close to very sparse vectors. To prove Proposition \ref{p: dominated and compressible}, we need to uplift these estimates for vectors which are less sparse.
For such vectors, we employ a different strategy.
These vectors are sufficiently spread.
 This allows us to obtain a small ball probability estimate which is strong enough to use the $\vep$-net argument.
To this end, L\'{e}vy concentration function turns out to be useful. Recall the L\'{e}vy concentration function is given by
\[
\cL(Z,\vep):=\sup_{u \in \R^n}\P(\norm{ Z-u}_2 \le \vep).
\]
Below we prove several results about L\'{e}vy concentration function, which are subsequently used in the proof Lemma \ref{l: dominated vectors}, and eventually lead to the proof of Proposition \ref{p: dominated and compressible}.

 \begin{lem}  \label{l: spread vector}
  Assume that the  matrix $\bar{A}_n$ satisfies the conditions of Proposition \ref{p: dominated and compressible}. For any $x \in \R^n$, let us denote $x_{(i)}$ to be the vector obtained from $x$ by setting its $i$-th coordinate to be zero. Then there exists a positive constant $c_{\ref{l: spread vector}}$, depending only on the fourth moment of $\{\xi_{ij}\}$, such that for any $x \in \R^n$ and any $i \in [n]$,
  \[
   \cL ((\bar{A}_n x)_i, \frac{1}{4} \sqrt{p} \norm{x_{(i)}}_2) \le  1-\frac{c_{\ref{l: spread vector}}p}{\left(\norm{x_{(i)}}_\infty/\norm{x_{(i)}}_2\right)^2+p}.
  \]
 \end{lem}
 \begin{proof}
 We begin with the standard symmetrization.
  Let $\d_1' \etc \d_n'$ and $\xi_1' \etc \xi_n'$ be independent copies of $\d_1 \etc \d_n$ and $\xi_1 \etc \xi_n$. Since the diagonal entries of $\bar{A}_n$ are zero, for any $b \in \R$ and $t>0$, we have that
 \begin{align}  \label{i: symm}
   &\P^2 \left( |(\bar{A}_n x)_i-b| \le t \right) \\
   &= \P \left( \left|\sum_{j\in [n]\setminus\{i\}} \d_j \xi_j x_j-b\right| \le t \right) \cdot \P \left( \left|\sum_{j\in [n]\setminus\{i\}} \d_j' \xi_j' x_j-b\right| \le t \right) \notag \\
   &\le \P \left( \left|\sum_{j\in [n]\setminus\{i\}} (\d_j \xi_j- \d_j'\xi_j') x_j\right| \le 2t \right). \notag
 \end{align}
 Denote $\theta_j=\d_j \xi_j- \d_j'\xi_j'$. Then $\E \theta_j=\E \theta_j^3=0$, $\E \theta_j^2=2p$, and $\E \theta_j^4\le cp$, for some constant $c$, depending only on the fourth moment of $\{\xi_{i,j}\}$.
  Set $S=\sum_{j \in [n]\setminus\{i\}} \theta_j x_j$. Then $\E S^2 \ge p \norm{x_{(i)}}_2^2$,
  and
 \begin{align*}
   \E S^4
   &= \sum_{j \in [n]\setminus\{i\}} \E \theta_j^4 \cdot  x_j^4 + \sum_{j \neq l \in [n] \setminus \{i\}} \E \theta_j^2
   x_j^2 \cdot \E \theta_l^2  x_l^2  \\
   &\le  c p \norm{x_{(i)}}_{\infty}^2 \cdot \norm{x_{(i)}}_2^2 + 4 p^2
   \norm{x_{(i)}}_2^4.
 \end{align*}
  Then the  Paley--Zygmund inequality (cf.~\cite[Lemma 3.5]{LPRT})
 \[
  \P(|S| \le t) \le 1-  \frac{(\E S^2 -t^2)^2}{\E S^4}
 \]
 yields
 \[
   \P ( |S| \le \frac{1}{2} \sqrt{p} \norm{x_{(i)}}_2)
   \le 1-\frac{c'p}{(\norm{x_{(i)}}_\infty/\norm{x_{(i)}}_2)^2+p},
 \]
 for some constant $c'<1$, depending only on $c$. Combining this with \eqref{i: symm}, and setting $c_{\ref{l: spread vector}}=c'/2$, we obtain
 \begin{align*}
   \cL ((\bar{A}_n x)_i, \frac{1}{4} \sqrt{p} \norm{x_{(i)}}_2)
   &\le  \sqrt{1-\frac{c'p}{(\norm{x_{(i)}}_\infty/\norm{x_{(i)}}_2)^2+p}} \\
   &\le  1-\frac{c_{\ref{l: spread vector}}p}{(\norm{x_{(i)}}_\infty/\norm{x_{(i)}}_2)^2+p}.
 \end{align*}
 \end{proof}

\noindent
To pass from an estimate for one coordinate to estimate for the norm, we need the following elementary lemma.
 \begin{lem}  \label{l: tensorization}
 Let $V_1 \etc V_n$ be non-negative  independent random variables such that $\P(V_i>1) \ge q$, for all $i \in [n]$, and for some  $q \in (0,1/2)$.
 Then there exist constants $0< c_{\ref{l: tensorization}}, c'_{\ref{l: tensorization}} < \infty$,  such that
 \[
  \P \left( \sum_{j=1}^n V_j \le \frac{c_{\ref{l: tensorization}}q n}{\log(1/q)} \right) \le \exp (-c'_{\ref{l: tensorization}}qn).
 \]
 \end{lem}
 \begin{proof}
  For a positive constant $\beta$,  denote $L(\beta):=\frac{ \beta q n}{\log(1/q)}$.
   Let $J:=\{j \in [n]| \ V_j>1\}$. If $\sum_{j=1}^n V_j \le L(\beta)$, then $|J|\le L(\beta)$.
  Hence,
  \begin{align}
    \P \left( \sum_{j=1}^n V_j \le L(\beta) \right)
    &\le \binom{n}{L(\beta)}(1-q)^{n-L(\beta)}
    \le \exp \left(L(\beta) \log \frac{en}{L(\beta)} -\frac{n}{2} \cdot \log \frac{1}{1-q} \right). \label{eq:tensorization}
     \end{align}
Since $\f{L(\beta)}{n} \log \Big(\f{e n}{L(\beta)}\Big) \ra 0$ uniformly in $q \in (0,1/2)$ as $\beta \ra 0$, we can choose $\beta$ small enough such that the RHS of \eqref{eq:tensorization} can be made smaller than $\exp(-c'qn)$ for some positive constant $c'$. This completes the proof.
 \end{proof}

 \noindent
 Combining Lemma \ref{l: spread vector} and Lemma \ref{l: tensorization}, we obtain the following corollary.
 \begin{cor}     \label{c: spread vector}
 Let $\bar{A}_n$ be as in Proposition \ref{p: dominated and compressible}. For every $x \in \R^n$ and $i \in [n]$, define $x_{(i)}$ to be the vector obtained from $x$ by setting its $i$-th coordinate to be zero. Then  for any $\alpha >1$, there exist $\be, \gamma >0$, depending on $\alpha$ and the fourth moment of $\{\xi_{ij}\}$, such that for $x \in \R^n$, satisfying
  $\sup_{i \in [n]}\left(\norm{x_{(i)}}_\infty/\norm{x_{(i)}}_2\right) \le \alpha \sqrt{p}$, we have
 \[
   \cL \left(\bar{A}_nx, \beta \cdot \sqrt{pn} \inf_{i \in [n]}\norm{x_{(i)}}_2\right)
   \le \exp (-\gamma n ).
 \]
 \end{cor}
 \begin{proof}
 Fix any $y \in \R^n$, and let $V_j=\frac{16}{p \norm{x_{(j)}}_2^2}((\bar{A}_nx)_j-y_j)^2$. Since by our assumption,
 \[
 \inf_{i \in [n]} \frac{c_{\ref{l: spread vector}}p}{(\norm{x_{(i)}}_\infty/\norm{x_{(i)}}_2)^2+p} \ge \frac{c_{\ref{l: spread vector}}}{\a^2+1},
 \]
 the claim then follows from Lemma \ref{l: spread vector} and Lemma \ref{l: tensorization} applied with
 \[
 q= \frac{c_{\ref{l: spread vector}}}{\a^2+1} \wedge \frac{1}{4}.
 \]
 \end{proof}

 \noindent
 Equipped with these results on L\'{e}vy concentration we now prove uniform lower bound on

 \noindent
 $\norm{(\bar{A}_n+D_n)x}_2$ for vectors in $\text{Dom}(M, C_{\ref{p: dominated and compressible}}(K+R)^{-4})$.

 \begin{lem}   \label{l: dominated vectors}
Let $\bar{A}_n$ be the matrix defined in Proposition \ref{p: dominated and compressible}, $p$ satisfy \eqref{eq:p_assumption}, and let $\ell_0$ be as in \eqref{eq:ell_0_dfn}. Fix $K, R \ge 1$, and let $D_n$ be any non-random diagonal matrix with real entries such that $\norm{D_n} \le R \sqrt{n p}$. Further denote
\[
    \rho:=(\wt{C}_{\ref{l: sparse vectors}} (K+R))^{-(\ell_0+1)}.
  \]
 There exist positive  constants $c_{\ref{l: dominated vectors}}, \ol{c}_{\ref{l: dominated vectors}}, C_{\ref{l: dominated vectors}}, \ol{C}_{\ref{l: dominated vectors}}$, depending on $\E[\xi_{ij}^4]$, $K$, and $R$, such that 
 for any $p^{-1} \le M \le c_{\ref{l: dominated vectors}} n$,
  \begin{align*}
   \P&\Big(\exists x \in \text{\rm Dom}(M, (C_{\ref{l: dominated vectors}}(K+R))^{-4}) \text{ such that } \norm{(\bar{A}_n +D_n)x}_2 \le (\ol{C}_{\ref{l: dominated vectors}}(K+R))^{-4} \rho \sqrt{np}\\
   &\hskip 5in \text{ and } \norm{\bar{A}_n} \le K \sqrt{pn}\Big) \\
   &\le \exp(- \ol{c}_{\ref{l: dominated vectors}} pn).
  \end{align*}
 \end{lem}
  \begin{proof}
  Let $c<1$. Denote for shortness $m=(8p)^{-1}$, so $m < M/2$.
   By Lemma \ref{l: sparse vectors} and Lemma \ref{l: compressible}, it is enough to obtain a uniform lower bound for all vectors from the set
   \[
    W:= \text{Dom}(M, (\wt{C}_{\ref{l: sparse vectors}}(K+R))^{-4}) \setminus
    \Big( \text{Comp}((8p)^{-1}, \rho) \cup \text{Dom}((8p)^{-1}, ({C}_{\ref{l: sparse vectors}}(K+R))^{-1}) \Big).
   \]
   We begin with a smaller set
   \[
    V:= \text{Sparse}(M)\cap S^{n-1} \setminus
    \Big( \text{Comp}((8p)^{-1}, \rho) \cup \text{Dom}((8p)^{-1}, ({C}_{\ref{l: sparse vectors}}(K+R))^{-1}) \Big).
   \]
First let us consider the case, $p \ge (1/4)n^{-1/3}$. In this case the proof is based on the straightforward $\vep$-net argument. Note that in this regime of $p$ as above, $\ell_0=1$, and so $\rho=(\wt{C}_{\ref{l: sparse vectors}} (K+R))^{-2}$.
   Since for any $x \in V$, $x \notin \text{Dom}((8p)^{-1}, (C_{\ref{l: sparse vectors}}(K+R))^{-1})$ we have that
   \[
     \frac{\norm{x_{[m+1:M]}}_{\infty}}{\norm{x_{[m+1:M]}}_2} \le
     {C}_{\ref{l: sparse vectors}}(K+R) \sqrt{8p}.
   \]
However, to apply Corollary   \ref{c: spread vector} we need to find $\sup_{i \in [n]} \left(\norm{x_{[m+1:M]\setminus\{i\}}}_\infty/\norm{x_{[m+1:M]\setminus\{i\}}}_2\right)$. This can be obtained easily. Note that $\norm{x_{[m+1:M]}}_\infty \le 1/\sqrt{m}$, and we have $x \notin \text{Comp}(m, \rho)$, which in turn implies that $\norm{x_{[m+1:M]}}_2 \ge \rho$. Therefore,
\beq\label{eq:lower_bd_inf_i}
\norm{x_{[m+1:M]\setminus\{i\}}}_2 \ge \norm{x_{[m+1:M]}}_2 - {1}/{\sqrt{m}} \ge \f{1}{2}\norm{x_{[m+1:M]}}_2.
\eeq
Here the last inequality follows from the assumption $p \le c (K+R)^{-2}$, for sufficiently small $c$, which we made at the beginning of Section \ref{sec: proof outline}.
 Therefore we have
\[
\sup_{i \in [n]} \f{\norm{x_{[m+1:M]\setminus\{i\}}}_\infty}{\norm{x_{[m+1:M]\setminus\{i\}}}_2} \le  4 {C}_{\ref{l: sparse vectors}}(K+R) \sqrt{p}.
\]
Now by Corollary \ref{c: spread vector}, enlarging ${\wt{C}}_{\ref{l: sparse vectors}}$ if needed, we deduce that
   \begin{align*}
    & \,\, \cL((\bar{A}_n+D_n)x, (\wt{C}_{\ref{l: sparse vectors}}(K+R))^{-3} \sqrt{pn} \inf_{i \in [n]}\norm{x_{[m+1:M]\setminus\{i\}}}_2) \\
    & \hskip1in  \le \cL(A_nx, (\wt{C}_{\ref{l: sparse vectors}}(K+R))^{-3} \sqrt{pn} \inf_{i \in [n]}\norm{x_{[m+1:M]\setminus\{i\}}}_2) \le \exp(-c' n),
   \end{align*}
for some constant $c'$ depending on $K$ and $R$. Using \eqref{eq:lower_bd_inf_i} again, and enlarging $\wt{C}_{\ref{l: sparse vectors}}$ again, we further deduce that
\beq\label{eq:levy_modified_bound}
\cL((\bar{A}_n+D_n)x, (\wt{C}_{\ref{l: sparse vectors}}(K+R))^{-3} \sqrt{pn}\norm{x_{[m+1:M]}}_2) \le \exp(-c'n).
\eeq
Now we will use this estimate of the L\'{e}vy concentration function to show that infimum over $V$ is well controlled. To this end, we will use a $\vep$-net argument. Since $V \subset \text{Sparse}(M)$, we begin by noting that the set $V$ is contained in $S^{n-1}$ intersected with the union of coordinate subspaces of dimension $M$.
Hence, for $\vep=(\wt{C}_{\ref{l: sparse vectors}}(K+R))^{-4} \rho$, there exists an $\vep$-net $\cN \subset V$ of cardinality less than
   \beq\label{eq:net_estimate}
    \binom{n}{M}  \left(\frac{3}{\vep} \right)^{M}
    \le \exp \left( c_{\ref{l: dominated vectors}}n \log \left(\frac{3  e}{c_{\ref{l: dominated vectors}} \vep} \right) \right).
   \eeq
   Here we used the assumption $M \le c_{\ref{l: dominated vectors}} n$.
  We can choose the constant $c_{\ref{l: dominated vectors}}$ sufficiently small so that the $|\cN| \le \exp((c'/2)n)$. Therefore using the union bound, we show that,
   \[
    \P (\exists \, x \in \cN \mid \norm{(\bar{A}_n+D_n)x}_2 \le  (\wt{C}_{\ref{l: sparse vectors}}(K+R))^{-3}  \sqrt{pn} \norm{x_{[m+1:M]}}_2)
    \le \exp(-(c'/2) n).
   \]
   The proof in this case is finished by approximating any point of $W$ by a point of $\cN$. Indeed, assume that for any $x \in \cN$,
   \[
     \norm{(\bar{A}_n+D_n)x}_2 \ge  (\wt{C}_{\ref{l: sparse vectors}}(K+R))^{-3} \sqrt{pn} \norm{x_{[m+1:M]}}_2.
   \]
Let $x' \in W$, then we can find $x \in \cN$ such that $\|{(x'_{[1:M]}}/{\|x'_{[1:M]}\|_2})-x\|_2 \le \vep$.
Let us show that $x$ approximates $x'$.
 Using $m \le M/2$ and the fact that all coordinates of $x'_{[M+1:n]}$ have smaller absolute values than those of $x'_{[m+1:M]}$,
we conclude that
\[
  \sqrt{M} \norm{x'_{[M+1:n]}}_{\infty} \le \sqrt{2} \norm{x'_{[m+1:M]}}_2.
\]
Now recalling that $x' \in \text{Dom}(M, (\wt{C}_{\ref{l: sparse vectors}}(K+R))^{-4})$, we have
\[    
\norm{x'_{[M+1:n]}}_2
    \le (\wt{C}_{\ref{l: sparse vectors}}(K+R))^{-4} \sqrt{M} \norm{x'_{[M+1:n]}}_{\infty}
    \le \sqrt{2} (\wt{C}_{\ref{l: sparse vectors}}(K+R))^{-4} \norm{x'_{[m+1:M]}}_2. 
\]
Next using the fact that $\|{(x'_{[1:M]}}/{\|x'_{[1:M]}\|_2})-x\|_2 \le \vep$, applying the triangle inequality, we also obtain
\[
\norm{x'_{[m+1:M]}}_2 \le \norm{x'_{[1:M]}}_2 \left(\norm{x_{[m+1:M]}}_2+\vep \right) \le  \norm{x_{[m+1:M]}}_2+ \vep.
\]
For any $x \in \cN$,  $x \notin \text{Comp}(m,\rho)$, we further have $\norm{x_{[m+1:M]}}_2 \ge \rho = (\wt{C}_{\ref{l: sparse vectors}}(K+R))^{4}\vep$.
Using the two previous inequalities we further deduce
   \begin{align*}
    \norm{x'_{[M+1:n]}}_2
    &
    \le \sqrt{2} (\wt{C}_{\ref{l: sparse vectors}}(K+R))^{-4} \norm{x'_{[m+1:M]}}_2  \\
    &\le 2 (\wt{C}_{\ref{l: sparse vectors}}(K+R))^{-4} (\norm{x_{[m+1:M]}}_2   +\vep)
    \le 4 (\wt{C}_{\ref{l: sparse vectors}}(K+R))^{-4} \norm{x_{[m+1:M]}}_2
   \intertext{and}
    \norm{x-x'}_2
    &\le \norm{x-(x'_{[1:M]}/\|x'_{[1:M]}\|_2)}_2+\left| 1 - \|x'_{[1:M]}\|_2\right|+ \norm{x'_{[M+1:n]}}_2\\
    & \le \vep + 2 \norm{x'_{[M+1:n]}}_2 \le  \vep+  8 (\wt{C}_{\ref{l: sparse vectors}}(K+R))^{-4} \norm{x_{[m+1:M]}}_2 \\
    & \qquad \quad   \qquad \qquad \qquad \le  9 (\wt{C}_{\ref{l: sparse vectors}}(K+R))^{-4} \norm{x_{[m+1:M]}}_2.
   \end{align*}
   Thus, choosing $\wt{C}_{\ref{l: sparse vectors}}$ sufficiently large, by the triangle inequality,
   \begin{align*}
    \norm{(\bar{A}_n+D_n)x'}_2
    &\ge \norm{(\bar{A}_n+D_n) x}_2 - (\norm{\bar{A}_n}+\norm{D_n}) \cdot \norm{x-x'}_2  \\
    &\ge \frac{1}{2} (\wt{C}_{\ref{l: sparse vectors}}(K+R))^{-3} \sqrt{pn} \norm{x_{[m+1:M]}}_2
    \ge \frac{1}{2} (\wt{C}_{\ref{l: sparse vectors}}(K+R))^{-3} \sqrt{pn} \rho.
   \end{align*}

   \vskip 0.2in

\noindent
   Assume now that $\frac{\ol{C}_{\ref{thm: smallest singular + norm}} \log n}{n} \le p < (1/4) n^{-1/3}$. In this case, the proof uses a more delicate $\vep$-net argument. To this end we combine two nets: a coarser one for small coordinates, and a finer one for large ones.

  \noindent
  Let $I, J \subset [n]$ be disjoint sets such that $|I|=m, \ |J|=M-m$, where $m$ and $M$ are the same as in the previous case.
    Let $\vep, \tau >0$ be numbers to be chosen later.
     The sets
     \begin{align*}
       &B_I:=\{u \in B_2^{n} \big| \supp(u) \subset I\},
       \intertext{and}
       &R_J:=\{u \in S^{n-1} \big| \supp(u) \subset J \text{ and }
      \norm{u}_{\infty} \le  {4C_{\ref{l: sparse vectors}}(K+R)} \sqrt{p}\},
     \end{align*}
      admit an $\vep$-net $\cN_I \subset B_I$ and a $\tau$-net $\cN_J \subset R_J$ of cardinalities
  \[
   |\cN_I|
   \le \left( \frac{3}{\vep} \right)^{|I|} \quad \text{ and }
   |\cN_J|
   \le \left( \frac{3}{\tau} \right)^{|J|}.
  \]
  Let $\cN_0$ be an $\vep$-net in $[\rho/\sqrt{2},1] \subset \R$, and let
  \[
   \cM_{I,J}:= \{u+ \l w \mid u \in \cN_I, \ w \in \cN_J, \l \in \cN_0 \},
  \]
    and $$\cM:= \Bigg(\bigcup_{\substack{I: I \subset [n], \\ |I|=m}}\bigcup_{\substack{J: J \subset [n],\\ |J|=M-m, I \cap J = \emptyset}}\cM_{I,J} \Bigg).$$
We now show that $\cM$ serves as an appropriate net for $W$. To this end, decompose  $x \in W$ as $x=u_x+v_x+r_x$,  where   $u_x=x_{[1:m]} $ contains $m$ coordinates of $x$ with largest absolute values, $v_x=x_{[m+1:M]}$ the intermediate ones, and $r_x=x_{[M+1:n]}$ the rest. The assumption $x \notin \text{Comp}(m, \rho) \cup \text{Dom}(m, ({C}_{\ref{l: sparse vectors}}(K+R))^{-1})$ implies that
   \beq\label{eq:v_x_r_x}
     \norm{x_{[m+1:n]}}_2=\sqrt{\norm{v_x}_2^2 + \norm{r_x}_2^2} \ge \rho,  \text{ and }
    \norm{\frac{v_x}{\norm{x_{[m+1:n]}}_2}}_{\infty} \le  C_{\ref{l: sparse vectors}}(K+R) m^{-1/2}= C_{\ref{l: sparse vectors}}(K+R) \sqrt{8p}.
   \eeq
   Since $x \in W$,
   \[
    \norm{r_x}_2
    \le (\wt{C}_{\ref{l: sparse vectors}}(K+R))^{-4} \sqrt{M} \norm{r_x}_{\infty}
    \le 2 (\wt{C}_{\ref{l: sparse vectors}}(K+R))^{-4} \norm{v_x}_2,
   \]
   where as in the previous case, the last inequality follows from the facts that the coordinates of $r_x$ have smaller magnitudes than the non-zero coordinates of $v_x$ and $m \le M/2$.   For $\wt{C}_{\ref{l: sparse vectors}} >2$ this in particular implies that $\norm{v_x}_2 \ge \norm{r_x}_2$. Therefore from \eqref{eq:v_x_r_x} we further deduce that
  \beq\label{eq:bound_on_v_x}
  \norm{v_x}_2 \ge \rho/\sqrt{2} \qquad \text{ and } \qquad \norm{\frac{v_x}{\norm{v_x}_2}}_{\infty} \le   4C_{\ref{l: sparse vectors}}(K+R) \sqrt{p}.\eeq
 Assume that $\supp(u_x) \subset I, \ \supp(v_x) \subset J$, for some $I , J \subset [n]$. Choose $\bar{u} \in \cN_I, \ \bar{v} \in \cN_J, \ \l \in \cN_0$  such that
   \beq\label{eq:approximating_net}
    \norm{u_x-\bar{u}}_2 \le \vep, \quad \norm{\frac{v_x}{\norm{v_x}_2}-\bar{w}} \le \tau,
    \quad \text{ and } |\l - \norm{v_x}_2| \le \vep.
   \eeq
   and consider $\bar{x}=\bar{u}+ \l \bar{w} \in \cM$.   For $\vep<\rho/\sqrt{2}$, we also have that
   \beq\label{eq:approximating_net_r}
    \norm{r_x}_2 \le 2 (\wt{C}_{\ref{l: sparse vectors}}(K+R))^{-4} \norm{v_x}_2
    \le 2 (\wt{C}_{\ref{l: sparse vectors}}(K+R))^{-4} (\l+\vep) \le 4 (\wt{C}_{\ref{l: sparse vectors}}(K+R))^{-4} \l.
   \eeq
 Thus we see from \eqref{eq:approximating_net}-\eqref{eq:approximating_net_r}  that $\bar{x}$ approximates $x$. Now using Corollary \ref{c: spread vector}, we would have liked to show that for any $\bar{x}\in \cM$, an inequality similar to \eqref{eq:levy_modified_bound} hold. However, such inequality is not possible for any $\bar{x} \in \cM$, as the conditions required for Corollary \ref{c: spread vector} does not hold for all $\bar{x} \in \cM$. We solve this issue by modifying the net $\cM$.  We construct the modification $\cM' \subset W$ as follows: If for an $\bar{x} \in \cM$, there exists an $x \in W$ such that \eqref{eq:approximating_net} holds, then we keep that $x$ in $\cM'$ (if there is more than one choice we choose any one of them arbitrarily). Note that this construction ensures that $|\cM'| \le |\cM|$, and moreover by the triangle inequality it follows that, for any $x \in W$, there exists $\bar{x} \in \cM'$ such that
 \beq \label{eq:approximating_net_triangle}
 \norm{u_x - u_{\bar{x}}}_2 \le 2\vep, \quad \norm{\frac{v_x}{\norm{v_x}_2}-\frac{v_{\bar{x}}}{\norm{v_{\bar{x}}}_2}} \le 2\tau,
    \quad \text{ and } |\norm{v_x}_2 - \norm{v_{\bar{x}}}_2| \le 2\vep.
 \eeq
 Proceeding analogous to \eqref{eq:approximating_net_r}, we also deduce that
 \beq\label{eq:approximating_net_r_triangle}
    \norm{r_x}_2 \le 2 (\wt{C}_{\ref{l: sparse vectors}}(K+R))^{-4} \norm{v_x}_2
    \le 2 (\wt{C}_{\ref{l: sparse vectors}}(K+R))^{-4} (\norm{v_{\bar{x}}}_2+2\vep) \le 6 (\wt{C}_{\ref{l: sparse vectors}}(K+R))^{-4} \norm{v_{\bar{x}}}_2,
   \eeq
as $\norm{v_{\bar{x}}}_2 \ge \rho/\sqrt{2}$.

Now fix $\bar{x} \in \cM'$. Then, using \eqref{eq:bound_on_v_x},  and proceeding as in \eqref{eq:lower_bd_inf_i}-\eqref{eq:levy_modified_bound}, from Corollary \ref{c: spread vector}, we deduce that
   \begin{align*}
  &  \cL((\bar{A}_n+D_n) \bar{x}, (\wt{C}_{\ref{l: sparse vectors}}(K+R))^{-3} \sqrt{pn} \norm{v_{\bar{x}}}_2)\\
   & \hskip3in \le \cL ( \bar{A}_n v_{\bar{x}},  (\wt{C}_{\ref{l: sparse vectors}}(K+R))^{-3}  \sqrt{pn}  \norm{v_{\bar{x}}}_2)
    \le e^{-c'n}.
   \end{align*}

  Assume that the parameters $\vep, \tau>0$ are chosen so that
 { \begin{equation}\label{MM__IJ}
 |\cM'| \le    |\cM| \le \binom{n}{m} \binom{n-m}{M-m} \frac{1}{\vep} \cdot \left( \frac{3}{\vep} \right)^{|I|}
    \cdot \left( \frac{3}{\tau} \right)^{|J|}
    \le e^{c'n/2}.
  \end{equation}
}
Then by the union bound,
\begin{align}\label{eq:dominated_net}
& \P \Big(\exists \, \bar{x} \in \cM' \text{ such that } \norm{(\bar{A}_n+D_n)\bar{x}}_2 \le  (\wt{C}_{\ref{l: sparse vectors}}(K+R))^{-3}  \sqrt{pn} \norm{v_{\bar{x}}}_2\Big)\notag\\
&   \hskip4.5in   \le \exp(-(c'/2) n).
\end{align}
 We now extend the uniform lower bound in \eqref{eq:dominated_net} for all $x \in W$.
 In the process of this extension, we select the parameters $\vep$ and $\tau$. Finally, we will check that this selection satisfies \eqref{MM__IJ}.

 Assume that  the complement of the set appearing in the \abbr{LHS} of \eqref{eq:dominated_net} occurs. Now we recall that for every $x \in W$, there exists an $\bar{x} \in \cM'$ such that \eqref{eq:approximating_net_triangle} holds. Therefore
   \begin{align}
    \norm{(\bar{A}_n +D_n)x}_2
    &\ge \norm{(\bar{A}_n+D_n) \bar{x}}_2 \notag\\
    &\quad - (\norm{\bar{A}_n}+\norm{D_n})  \Big( \norm{u_x-u_{\bar{x}}}_2 +
           \norm{v_x-v_{\bar{x}}}_2 +\norm{r_x}_2 +\norm{r_{\bar{x}}}_2
            \Big).
           \label{eq:lower_bound_dom_non_sparse}
   \end{align}   
To obtain a lower bound on the \abbr{RHS} of \eqref{eq:lower_bound_dom_non_sparse}, we use \eqref{eq:approximating_net_triangle} to note that
\begin{align*}
 \norm{v_x-v_{\bar{x}}}_2 \le  \norm{\f{v_x}{\norm{v_x}_2}-\f{v_{\bar{x}}}{\norm{v_{\bar{x}}}_2}}_2 \norm{v_{\bar{x}}}_2 + \norm{v_{{x}}}_2 \left| 1- \frac{\norm{v_{\bar{x}}}_2}{\norm{v_{{x}}}_2}\right| \le 2(\vep +\tau \norm{v_{\bar{x}}}_2).
\end{align*} 
Further using \eqref{eq:approximating_net_r} and \eqref{eq:approximating_net_r_triangle} we also obtain that    
\[
\norm{r_x}_2 + \norm{r_{\bar{x}}}_2 \le 8 (\wt{C}_{\ref{l: sparse vectors}}(K+R))^{-4} \norm{v_{\bar{x}}}_2.
\]
 Denoting $\mu'=(\wt{C}_{\ref{l: sparse vectors}}(K+R))^{-3}$, applying the previous two estimates, and \eqref{eq:approximating_net_triangle}, from \eqref{eq:lower_bound_dom_non_sparse}, we therefore deduce that 
   \begin{equation}    
  \norm{(\bar{A}_n +D_n)x}_2  \ge \mu' \norm{v_{\bar{x}}}_2 \sqrt{pn} - 2(K+R) \sqrt{pn}  \cdot \big(\vep+ \norm{v_{\bar{x}}}_2 \tau+ \vep
     + 4 (\wt{C}_{\ref{l: sparse vectors}}(K+R))^{-4} \norm{v_{\bar{x}}}_2 \big) 
        \label{eq:dom_W_bound}
   \end{equation}
      Setting
      \beq\label{eq:tau_vep_choice}
      \tau= \frac{\mu'}{16(K+R)}, \qquad \vep= \frac{\mu' \rho}{16 (K+R)},
      \eeq
  enlarging   $\wt{C}_{\ref{l: sparse vectors}}$ further, if necessary,  and recalling the fact that $\norm{v_{\bar{x}}}_2 \ge \rho/\sqrt{2}$, from the inequality \eqref{eq:dom_W_bound} we further deduce that
   \[
         \norm{(\bar{A}_n +D_n)x}_2 \ge \frac{\mu'}{8} \norm{v_{\bar{x}}}_2 \sqrt{pn}
         \ge \frac{\mu'}{8 \sqrt{2}} \rho \sqrt{pn}.
   \]
   It thus remains to check that \eqref{MM__IJ} holds for the choice of parameters in \eqref{eq:tau_vep_choice}.
    To this end, recall that $m= (8p)^{-1}$ and $M \le cn$.
    Substituting this in \eqref{MM__IJ}, we obtain
   \[\binom{n}{m}\binom{n-m}{M-m} \le \binom{n}{m}\binom{n}{M} \le \Big(\f{en}{m}\Big)^m \Big(\f{en}{M}\Big)^M \le ({8epn})^{(8p)^{-1}} \Big(\f{e}{c}\Big)^{cn}.
   \]
 Therefore \eqref{MM__IJ} yields
      \[
  |\cM'| \le |\cM| \le      \left( \frac{48 e (K+R)}{c \mu'} \right)^{cn}
       \left( \frac{384 (K+R) p n}{\mu' \rho} \right)^{(8p)^{-1}}.
   \]
   Thus, we have to show that
   \beq\label{eq:MM__IJ_ineq}
    \left( \frac{48 e (K+R)}{c \mu'} \right)^{cn}
       \left( \frac{384 (K+R) p n}{\mu' \rho} \right)^{(8p)^{-1}}
    \le e^{c'n/2}.
   \eeq
   To this end, we claim that
   \[
   \left( \frac{384 (K+R) p n}{\mu' \rho} \right)^{(8p)^{-1}} \le  \left( \frac{48 e (K+R)}{c \mu'} \right)^{cn},
   \]
   from which it is easy to see that the bound in \eqref{eq:MM__IJ_ineq} follows if $c$ is chosen small enough with respect to $c'$. Turning to prove our claim, we note that it is enough to prove that
   \[
    p^{-1} \log \left( \frac{pn}{\rho} \right) \ll n,
   \]
   which is immediate since $np \ra \infty$, and $\ell_0 \ll np$. This shows that \eqref{MM__IJ} holds and thus the proof is completed.
 \end{proof}

 Finally we are ready to prove Proposition \ref{p: dominated and compressible}.

 \begin{proof}[Proof of Proposition \ref{p: dominated and compressible}]
  Since $\text{Sparse}(M)\cap S^{n-1} \subset \text{Dom}(M, (C_{\ref{l: dominated vectors}}(K+R))^{-4})$, with the help of Lemma \ref{l: dominated vectors}, the proof is completed using the same arguments as in the proof of Lemma \ref{l: compressible}. The details are omitted.
 \end{proof}

 \begin{rmk}\label{rmk:comp_dom_modification}
 In the proof of Theorem \ref{thm: smallest singular + norm} we will also need a modification of Proposition \ref{p: dominated and compressible}, where the matrix under consideration is not a $n \times n$ matrix, but a $(n-1) \times n$ matrix. One can check that if  some modified versions of Lemma \ref{l: pattern} and Corollary \ref{c: spread vector}, applicable to $(n-1) \times n$ matrices are available, then the rest of the proof remains exactly same. Moreover, for $(n-1) \times n$ matrices one can easily  reprove Lemma \ref{l: pattern} and Corollary \ref{c: spread vector} with slightly worse bounds. Therefore the proof of the required modification of Proposition \ref{p: dominated and compressible} is straightforward, and hence all the details are omitted.
 \end{rmk}


\begin{rmk}\label{rmk:omega_complex}
In Proposition \ref{p: dominated and compressible} we computed a probability bound for the infimum of

\noindent
$\norm{(\bar{A}_n+D_n)x}_2$ over dominated and compressible vectors $x \in \R^n$. This treatment of the infimum for general real-valued diagonal matrix $D_n$ such that $\norm{D_n} \le R\sqrt{np}$ for some finite positive $R$, is motivated by the analysis of the limiting spectral distribution of $\bar{A}_n$. It is well known that a key step to such analysis is the control on $s_{\min}(\bar{A}_n - \omega \sqrt{np}I_n)$ for $\omega \in B(0,R) \subset \C$, for some $R$ finite (see \cite{bordenave_chafai}).

It can be easily checked the proof of Lemma \ref{l: sparse vectors} and Lemma \ref{l: compressible} remains same when we allow $D_n$ to be complex-valued diagonal matrix, and the infimum is now taken over compressible and dominated vectors in $\C^n$. Corollary \ref{c: spread vector} also continues to hold for vectors in $\C^n$. However, the proof of Lemma \ref{l: dominated vectors} uses some estimates of $\gamma$-net in $\R^n$. Therefore those steps need some modifications. To this end, note that \eqref{eq:net_estimate} becomes
\[
    \binom{n}{M}  \left(\frac{3}{\vep} \right)^{2cn},
  \]
and \eqref{MM__IJ} becomes
\[
    |\cM| \le \binom{n}{m} \binom{n-m}{M-m}\left( \frac{1}{\vep}\right)^2 \cdot \left( \frac{3}{\vep} \right)^{2|I|}
    \cdot \left( \frac{3}{\tau} \right)^{2|J|},
  \]
and rest of the estimates remains same. Shrinking the constant $c$, if necessary, repeating the same steps one can deduce the conclusion of Lemma \ref{l: dominated vectors}, may be with a slightly worse constants. Building on this one can then extend the result of Proposition \ref{p: dominated and compressible}, where $D_n$ is now complex-valued diagonal matrix and the infimum is taken over complex vectors.

To obtain the necessary bound on $s_{\min}(A_n - \omega \sqrt{np}I_n)$ for $\omega \in B(0,R) \subset \C$, we also need an modified version of Proposition \ref{p: norm on S_L} for complex vectors. However, this is not a straightforward extension from the real case.  See also Remark \ref{rmk:omega_complex_small_lcd}.
\end{rmk}

 \section{Vectors with a small LCD} \label{sec: small LCD}


 Bounding the smallest singular value of a random matrix $\bar{A}_n$ depends crucially on a strong estimate of the L\'evy concentration function of $\bar{A}_nx$ for $x \in S^{n-1}$. Such estimate, however is impossible to achieve for a vector having a rigid arithmetic structure. As such structure is measured by the \abbr{LCD} (recall Definition \ref{dfn:lcd}), we have to treat the vectors with a small \abbr{LCD} separately. Fortunately, the set of vectors with a smaller \abbr{LCD} has a smaller complexity, i.e. a smaller $\vep$-net size.
 We encounter two opposite effects: a larger \abbr{LCD} means a better L\'evy concentration function bound, and at the same time, a larger complexity of the set. We show below that these two effects compensate each other precisely.
 To this end, we partition the set of vectors with a small \abbr{LCD} into the sets $S_L$ for which the \abbr{LCD} roughly equals $L$. Since the \abbr{LCD} is roughly constant in $S_L$, we obtain a uniform bound on the L\'{e}vy concentration function, and thereby  using an $\vep$-net we show that the infimum over $S_L$ is well controlled.

 Since we have already  obtained a lower bound on the infimum over compressible and dominated vectors in Proposition \ref{p: dominated and compressible}, we will consider vectors which are neither compressible nor dominated.
 For $p^{-1} \le M \le c_{\ref{p: dominated and compressible}}n$, and $\rho$ as in Proposition \ref{p: dominated and compressible}, define
 \[
   W:=\{ x \in S^{n-1} \mid x \notin \text{Comp}(M, \rho) \cup \text{Dom}(M,(C_{\ref{p: dominated and compressible}}(K+R))^{-4}) \}.
 \]
 Next for $v \in \R^n$, let $I(v):=\text{Supp}\big(v_{[M+1:n]}\big)$ be the set of small coordinates, and let $v_{I(v)}=v_{[M+1:n]}$.
Recall that for $x \in S^{n-1}$, its \abbr{LCD} is defined as
   \[
     D(x)
      := \inf \Big\{ \theta>0 : \dist(\theta x, \mathbb{Z}^n) < (\delta_0 p)^{-1/2} \sqrt{\log_+(\sqrt{\delta_0 p} \theta)}
             \Big\},
   \]
   where $\delta_0 \in (0,1)$ is chosen as in Remark \ref{rmk:choose_delta_0}.
As mentioned above we need to define level sets $S_L$. Since the diagonal entries of $\bar{A}_n$ are zero, we need to work with the following modified definition of level sets. For any $L \ge 1$, we define
 \[
  S_L:=\left\{v \in W \mid L \le \inf_{i \in [n]} D(v_{I(v)\setminus\{i\}}/\norm{v_{I(v)\setminus\{i\}}}_2) <2L \right\}.
 \]
 We are now ready to state the main result of this section.
 \begin{prop}    \label{p: norm on S_L}
  Fix $K, R \ge 1$, and let $D_n$ be a non-random diagonal matrix with real entries such that $\norm{D_n} \le R \sqrt{ np}$. Let $\wt{A}_n^{D,m}$ be the  $m \times n$ matrix obtained from $(\bar{A}_n +D_n)^{{\mathsf T}}$ by collecting its last $m$ rows, where $\bar{A}_n$ is the matrix defined in  Theorem \ref{thm: smallest singular + norm}. When $D_n=0$, we write $\wt{A}_n^m$ instead of $\wt{A}_n^{0,m}$. Fix a positive real $r\ge 1$. Then there exist small positive constants $ c_{\ref{p: norm on S_L}}, c'_{\ref{p: norm on S_L}}, \ol{c}_{\ref{p: norm on S_L}}$, and a large positive constant $\ol{C}_{\ref{p: norm on S_L}}$, depending only on  $\E\xi_{ij}^4$, and small positive constants $\wt{c}_{\ref{p: norm on S_L}},c''_{\ref{p: norm on S_L}}, c^*_{\ref{p: norm on S_L}}$, depending on  $\E\xi_{i,j}^4$ and  $r$, such that, if $r \ge (\ol{C}_{\ref{p: norm on S_L}}(K+R))^2$, then for $r p^{-1/2} \le L \le \exp(c''_{\ref{p: norm on S_L}} pn/(K+R)^2)$,  $m \ge n - c^*_{\ref{p: norm on S_L}}(K+R)^2/p$, we have
    \[
    \P\Big( \inf_{v \in S_L} \norm{\wt{A}_n^{D,m} v}_2 \le \ol{c}_{\ref{p: norm on S_L}} \rho \vep_0 \sqrt{pn} \text{ and } \norm{\wt{A}_n^m} \le K \sqrt{pn}\Big)
    \le \exp(-\wt{c}_{\ref{p: norm on S_L}}n),
   \]
   where
    \[
     \vep_0 =  \min(c_{\ref{p: norm on S_L}}/\sqrt{r}, c'_{\ref{p: norm on S_L}} \sqrt{n}/L).
   \]
 \end{prop}

 \noindent
 Similar to Section \ref{sec: compressible} a crucial tool here would be bounds on L\'{e}vy concentration function (recall Definition \ref{dfn:levy}). However, the estimate obtained in Corollary \ref{c: spread vector} is not sufficient for incompressible vectors. To this end, we find estimates in terms of the \abbr{LCD}, see Definition \ref{dfn:lcd}. For $\delta_0$ as in Remark \ref{rmk:choose_delta_0}, from \cite[Theorem 6.3]{V} we get the following result:
 \begin{prop}\label{prop:lcd}
 Let $S \in \R^n$ be a random vector with i.i.d.~coordinates of the form $S_j =\delta_j \xi_j$, where $\P(\delta_j=1)=p$, and $\xi_j$'s are random variables with  unit variance, and finite fourth moment, which are independent of $\delta_j$.
   Then for any $v \in S^{n-1}$
   \begin{equation} \label{eq: LCD}
    \cL \left( \sum_{j=1}^n S_j v_j, \sqrt{p} \vep \right) \le C_{\ref{prop:lcd}} \left( \vep+\frac{1}{\sqrt{p} D(v)} \right),
   \end{equation}
   for some constant $C_{\ref{prop:lcd}}$, depending only on $\E|\xi_k|$ and $\E\xi_k^4$.
 \end{prop}
Let $I \subset [n]$, and for any $v \in \R^n$, let $v_I \in \R^n$ be the vector with coordinates $v_I(j)=v(j) \cdot \bI(j \in I)$. Since the diagonal entries of $\bar{A}_n$ are zero, depending on the value of $m$, for every $i \in [m]$, there exists a $j \in [n]$ such that $(\wt{A}_n^m)_{ij}=0$. Thus  applying Proposition \ref{prop:lcd} we deduce that
 \begin{align*}
 \cL\left((\wt{A}_n^m v)_i, \inf_{j \in [n]}\norm{v_{I\setminus\{j\}}}_2 \sqrt{p} \vep \right) & \le  \cL\left((\wt{A}_n^m v_I)_i, \norm{v_{I\setminus\{j\}}}_2 \sqrt{p} \vep \right) \\
 & \le C_{\ref{prop:lcd}} \left( \vep+\frac{1}{\sqrt{p} D(v_{I\setminus\{j\}}/\norm{v_{I\setminus\{j\}}}_2)} \right)\\
 & \le C_{\ref{prop:lcd}} \left( \vep+\frac{1}{\sqrt{p} \inf_{j \in [n]}D(v_{I\setminus\{j\}}/\norm{v_{I\setminus\{j\}}}_2)} \right).
 \end{align*}
 Now a direct application of \cite[Remark 3.5]{V} gives the following result on tensorization, which allows to transfer the bound on L\'evy concentration function from random variables to random vector:
 \begin{prop}\label{prop: lcd_tensorize}
 Let $\wt{A}_n^m$ be the matrix defined in Proposition \ref{p: norm on S_L}.
   Then for any $\vep>0$, and any $I \subset [n]$ we have,
   \begin{equation}   \label{eq: tensorized}
    \cL(\wt{A}_n^mv, \vep \inf_{j \in [n]}\norm{v_{I\setminus\{j\}}}_2 \sqrt{pm})
    \le C_{\ref{prop: lcd_tensorize}}^m  \left( \vep+\frac{1}{\sqrt{p} \inf_{j \in [n]}D(v_{I\setminus\{j\}}/\norm{v_{I\setminus\{j\}}}_2)} \right)^m,
   \end{equation}
   where $C_{\ref{prop: lcd_tensorize}}$ is some constant, depending only on $\E|\xi_{i,j}|$ and $\E\xi_{i,j}^4$.
 \end{prop}

 \noindent
 Setting the parameter $L=(\delta_0 p)^{-1/2}$ in \cite[Definition 6.1]{V}, we note that the definition of \abbr{LCD} there matches our definition of \abbr{LCD}. Therefore from \cite[Lemma 6.2]{V}, we immediately obtain:
 \begin{prop}
 \label{l: LCD bound}
  Let $x \in S^{n-1}$. Then
  \[
   D(x) \ge \frac{1}{2 \norm{x}_\infty}.
  \]
 \end{prop}





 \noindent
 We now proceed to the proof of Proposition \ref{p: norm on S_L}.

 \begin{proof}[Proof of Proposition \ref{p: norm on S_L}]
   The proof relies on a covering argument. The lower bound for the \abbr{LCD} is used to obtain the uniform estimate for the L\'evy concentration function. Then we construct a special $\vep_0$-net of a small cardinality, and extend the L\'evy concentration function estimate from one point to the whole net by the union bound. Finally, we use approximation to extend this bound to the set $S_L$.

 \begin{step}
   Recall \[
   r p^{-1/2} \le L \le \exp(c''_{\ref{p: norm on S_L}} pn/(K+R)^2) \text{ and }
   \vep_0 =  \min(c_{\ref{p: norm on S_L}}/\sqrt{r}, c'_{\ref{p: norm on S_L}} \sqrt{n}/L).\]
   Since $L\sqrt{p} \ge r$, and $np \ra \infty$, we have that $\vep_0 \ge \frac{1}{\sqrt{p}L}$. Thus, for $v \in S_L$, by \eqref{eq: tensorized}, we immediately obtain that
   \[
   \cL\big(\wt{A}_n^{D, m}v, \inf_{j \in [n]}\norm{v_{I(v)\setminus\{j\}}}_2  c \vep_0 \sqrt{pm}\big)\le \cL\big(\wt{A}_n^mv, \inf_{j \in [n]}\norm{v_{I(v)\setminus\{j\}}}_2  c \vep_0 \sqrt{pm}\big) \le \vep_0^m,
 \]
 where $c= (2 C_{\ref{prop: lcd_tensorize}})^{-1}$. 
   \end{step}
 \begin{step}
  To make the approximation possible, we have to approximate the large and the small coordinates of $v$ differently. Since $v \in S_L$, for some $j \in [n]$, we have $D(v_{I(v)\setminus\{j\}}/\norm{v_{I(v)\setminus\{j\}}}_2) \le 2L$. For this $j \in [n]$, a scaled copy of the vector $v_{I(v)\setminus\{j\}}$ is close to an integer point.
  We will use a scaled copy of this point to approximate $v_{I(v)\setminus\{j\}}$.
  We do not have any information about the vector $v_{I^c(v)\cup \{j\}}$ besides $\norm{v_{I^c(v)\cup\{j\}}}_2 \le 1$.  Therefore, this vector will be approximated in the $\ell_2$ norm using the standard volumetric estimate.

 \noindent 
Now, we pass to the details of this construction. To this end, fixing $I \subset [n]$  a set of cardinality $n-r^2p^{-1}$, we denote
  \[
    Z_I :=\{ z \in \mathbb{Z}^n \mid \supp(z) \subset I \text{ and } 0<\norm{z}_2 \le 2L \},
  \]
  and let $\cN_I:= \{ z/ \norm{z}_2 \mid z \in Z_I\}$.
  A simple volumetric calculation shows that
  \[
   |\cN_I| \le \left(2 + \frac{\ol{c}L}{\sqrt{n}} \right)^{n-r^2 p^{-1}},
  \]
  for a universal constant $\ol{c}$. Also, there exists a $(c\vep_0 \rho/10(K+R))$-net (the constant $c$ is the constant obtained in Step 1) $\cN_I'$ in $\{x \in B_2^n \mid \supp(x) \subset I^c \}$ of cardinality
  \[
   |\cN_I'| \le \left( \frac{30 (K+R)}{c\vep_0 \rho} \right)^{r^2 p^{-1}}.
  \]
  Let $\cN_0$ be a $(c\vep_0 \rho/10(K+R))$-net in $[\rho/2,1]$ of cardinality
  \[
   |\cN_0| \le  \frac{30 (K+R)}{c\vep_0 \rho}.
  \]
  Set
  \[
    \cM^{(1)}:= \bigcup_{\substack{I \subset [n] \\ |I|=n-r^2 p^{-1}}} \{x+t y \mid x \in \cN_I', \ y \in \cN_I, \ t \in \cN_0 \}.
  \]
This set $\cM^{(1)}$ does not quite serve as an appropriate $\vep$-net of $S_L$, because we also need to consider those $v\in S_L$ for which $j \in I(v)$. For such $v$, the cardinality of $I(v)\setminus\{j\}$ is $n-r^2p^{-1}-1$. Thus we need a modification of the set $\cM^{(1)}$. Namely, we denote
\[
    \cM^{(2)}:= \bigcup_{\substack{I \subset [n] \\ |I|=n-r^2 p^{-1}-1}} \{x+t y \mid x \in \cN_I', \ y \in \cN_I, \ t \in \cN_0 \},
\]
where the estimates on the cardinality of $\cN_I$ and $\cN_I'$ now changes to
\[
  |\cN_I| \le \left(2 + \frac{\ol{c}L}{\sqrt{n}} \right)^{n-r^2 p^{-1}-1}, \quad \text{ and } \quad    |\cN_I'| \le \left( \frac{30 (K+R)}{c\vep_0 \rho} \right)^{r^2 p^{-1}+1}.
\]
Set $\cM:=\cM^{(1)} \cup \cM^{(2)}$. Therefore the previous estimates now yield
  \begin{align}
    |\cM | \le |\cM^{(1)}|+|\cM^{(2)}|
    &\le 2 c^{-n}\binom{n}{r^2p^{-1}+1}  \left( \frac{30 (K+R)}{\vep_0 \rho} \right)^{r^2 p^{-1}+2}
     \left(2 + \frac{\ol{c}L}{\sqrt{n}} \right)^{n-r^2 p^{-1}} \notag\\
    &\le 2 c^{-n} \left( \frac{C(K+R)n}{r^2 p^{-1} \vep_0 \rho}
          \left(2 + \frac{\ol{c}L}{\sqrt{n}} \right)^{-1}\right)^{r^2 p^{-1}+1}
          \left( \frac{30 (K+R)}{\vep_0 \rho} \right)
          \left(2 + \frac{\ol{c}L}{\sqrt{n}} \right)^{n+1}, \label{eq:cM_cardinality}
  \end{align}
  where $C$ is some absolute constant.
  Recall that $\rho= (C_{ \ref{p: dominated and compressible}}K)^{-\ell_0 -6}$ (see Proposition  \ref{p: dominated and compressible}), where $\ell_0$ is defined in \eqref{eq:ell_0_dfn}.  Thus $ \log (1/\rho) \ll np$, and therefore we can choose a constant $c_1$ arbitrarily small such that $\rho^{-1} \le \exp(c_1 np)$ for all large $n$. Hence the third term in the \abbr{RHS} of \eqref{eq:cM_cardinality} is bounded above by $\f{1}{\vep_0}\exp(\ol{c}_1 pn)$, where $\ol{c}_1$ is another arbitrarily small positive finite constant.
  Similarly, we conclude that
  \[
     \left( \frac{C(K+R)n}{r^2  \rho} \right)^{r^2 p^{-1}+1}
     \le \exp(2 r^2 \ol{c}_1 n).
  \]
   Next,  from the upper bound of $L$, and the definition of $\vep_0$ it follows that
  \[
  \f{np}{\vep_0(2+\f{\ol{c}L}{\sqrt{n}})} \le {(c'_{\ref{p: norm on S_L}})^{-1}}\f{np}{\f{\sqrt{n}}{L}(2+\f{\ol{c}L}{\sqrt{n}})}\le \f{\sqrt{n} p L}{c'_{\ref{p: norm on S_L}}}\le \f{\exp(2{c}''_{\ref{p: norm on S_L}}pn)}{c'_{\ref{p: norm on S_L}}}.
  \]
  The last inequality follows from the assumption \eqref{eq:p_assumption}, and the upper bound on $L$.
  Therefore combining all the estimates we get
  \beq\label{eq:cM_cardinality_new}
  |\cM| \le \f{\exp(2r^2\ol{c}_1 n+ \ol{c}_1 pn) \exp(4r^2 c''_{\ref{p: norm on S_L}}n)}{\vep_0(c'_{\ref{p: norm on S_L}})^{2r^2p^{-1}} c^n} \left(2 + \frac{\ol{c}L}{\sqrt{n}} \right)^{n+1}.
  \eeq

\noindent
Now we will show that $\cM$ serves as an appropriate $\vep$-net for $S_L$. To this end, let $v \in S_L$. Then there exists $j \in [n]$ such that $D(v_{I(v)\setminus\{j\}}/\norm{v_{I(v)\setminus\{j\}}}_2) <2L$. Let us assume that $j \in I(v)$, and write $v=v_{I(v)\setminus\{j\}}+v_{I^c(v)\cup \{j \}}$. We claim that there exists $v'=\bar{x}+t \bar{y} \in \cM^{(2)}$ such that
  \begin{align}  \label{approx v}
    \norm{v_{I^c(v)\cup\{j\}}-\bar{x}}_2 \le \frac{c\rho \vep_0}{10 (K+R)}, \quad
      &\norm{\frac{v_{I(v)\setminus\{j\}}}{\norm{v_{I(v)\setminus\{j\}}}_2}-\bar{y}}_2 \le \frac{2\sqrt{\log(\sqrt{\delta_0 p} \cdot 2L)}}{\sqrt{\delta_0 p}L},
      \notag \\
       \text{ and }
      &\big| t- \norm{v_{I(v)\setminus\{j\}}}_2 \big| \le  \frac{c\rho \vep_0}{10 (K+R)}.
  \end{align}
  Indeed, choose $\bar{x} \in \cN'_{I(v)\setminus\{j\}}$ such that $\norm{v_{I^c(v)\cup\{j\}}-\bar{x}}_2< \f{c\rho \vep_0}{10(K+R)}$. By the definition of the \abbr{LCD}, we can find $z \in \mathbb{Z}^n$ such that
  \[
   \norm{\theta \frac{v_{I(v)\setminus\{j\}}}{\norm{v_{I(v)\setminus\{j\}}}_2}-z}_2 < \frac{\sqrt{\log (\sqrt{\delta_0 p} \cdot \theta)}}{\sqrt{\delta_0 p}}.
  \]
  Since $L \le D(v_{I(v)\setminus\{j\}}/\norm{v_{I(v)\setminus\{j\}}}_2) < 2L$, we have $L \le \theta < 2L$, which implies
  \[
   \norm{ \frac{v_{I(v)\setminus\{j\}}}{\norm{v_{I(v)\setminus\{j\}}}_2}-\frac{z}{\theta}}_2 < \frac{\sqrt{\log (\sqrt{\delta_0 p}  2L)}}{\sqrt{\delta_0 p}L}.
  \]
  Thus setting $\bar{y}=z/\norm{z}_2 \in \cN_{I(v)\setminus\{j\}}$ we obtain
  \[
  \norm{\norm{\bar{y}}_2-\f{\norm{z}_2}{\theta}}_2 = \norm{\norm{\f{v_{I(v)\setminus\{j\}}}{\norm{v_{I(v)\setminus\{j\}}}_2}}_2 - \f{\norm{z}_2}{\theta}}_2 \le \norm{ \frac{v_{I(v)\setminus\{j\}}}{\norm{v_{I(v)\setminus\{j\}}}_2}-\frac{z}{\theta}}_2,
  \]
  and therefore
  \[
   \norm{ \frac{v_{I(v)\setminus\{j\}}}{\norm{v_{I(v)\setminus\{j\}}}_2}-\bar{y}}_2 \le
   2\norm{ \frac{v_{I(v)\setminus\{j\}}}{\norm{v_{I(v)\setminus\{j\}}}_2}-\frac{z}{\theta}}_2 < \frac{2\sqrt{\log (\sqrt{\delta_0 p}  2L)}}{\sqrt{\delta_0 p}L}.
  \]
Finally, noting that $S_L \subset (\text{Comp}(M,\rho))^c$, we have that  for any $j \in [n]$
 \beq\label{eq:lower_bd_v_I(v)}
 \norm{v_{I(v)\setminus\{j\}}}_2 \ge \norm{v_{I(v)}}_2 - \norm{v_{I(v)}}_\infty \ge \norm{v_{I(v)}}_2- \f{1}{\sqrt{M}} \ge \frac{1}{2}\norm{v_{I(v)}}_2\ge \f{\rho}{2},
 \eeq
 where the second last step follows from upon choosing $r$ sufficiently large. Therefore we can choose $t \in \cN_0$ so that $\big| t- \norm{v_{I(v)}}_2 \big| \le  \frac{c\rho \vep_0}{10 (K+R)}$. 

In the proof of \eqref{approx v} we have assumed that $j \in I(v)$. One can repeat the same proof above even when $j \notin I(v)$, to conclude that in this case, there exists $\bar{v} \in \cM^{(1)}$ such that such that \eqref{approx v} still holds. Hence, combining these two arguments we obtain that for every $v \in S_L$, there exists a $\bar{v}\in \cM$ such that \eqref{approx v} holds.

  The deficiency of this construction is that $\cM \not \subset S_L$, so we cannot use the small ball estimates we obtained for the points of $S_L$ in Step 1. This however, can be easily corrected. For any point $v'=\bar{x}+t\bar{y} \in \cM$, choose one point $v \in S_L$ satisfying \eqref{approx v}, whenever it exists. If such a point does not exist, we skip the point $v'$. These points $v$ form a set $\cM'$, which can be used instead of $\cM$. Indeed, the triangle inequality implies that for any $w \in S_L$, there exists $v=\bar{x}+t\bar{y} \in \cM'$, and $j \in [n]$, such that
  \begin{align}  \label{approx v inside}
    \norm{w_{I^c(w)\cup\{j\}}-\bar{x}}_2 \le \frac{c\rho \vep_0}{5 (K+R)}, \quad
      &\norm{\frac{w_{I(w)\setminus\{j\}}}{\norm{w_{I(w)\setminus\{j\}}}_2}-\bar{y}}_2 \le \frac{4\sqrt{\log(\sqrt{\delta_0 p} \cdot 2L)}}{\sqrt{\delta_0 p}L},
      \notag \\
       \text{ and }
      &\big| t- \norm{w_{I(w)\setminus\{j\}}}_2 \big| \le  \frac{c\rho \vep_0}{5 (K+R)}.
  \end{align}
  Obviously, $|\cM'| \le |\cM|$.
 \end{step}

 \begin{step}
 By Step 1, for any $v \in \cM'$,
  \[
   \cL\big(\wt{A}_n^{D, m}v, \inf_{j \in [n]}\norm{v_{I(v)\setminus\{j\}}}_2  c \vep_0 \sqrt{pm}\big) \le \vep_0^m.
 \]
Now from \eqref{eq:lower_bd_v_I(v)}, we also have that,
 \[
 \inf_{j \in [n]}  \norm{v_{I(v)\setminus\{j\}}}_2 \ge \f{1}{2} \norm{v_{I(v)\setminus\{j^*\}}}_2, \quad \forall j^* \in [n].
 \]
 Hence absorbing the factor $1/2$ in $c$, we deduce the following estimate on the L\'{e}vy  concentration function:
 \[
   \cL\big(\wt{A}_n^{D, m}v, \norm{v_{I(v)\setminus\{j\}}}_2  c \vep_0 \sqrt{pm}\big) \le \vep_0^m, \quad \forall j\in [n], \forall v \in \cM'.\]
Thus by the union bound,
  \begin{align}
  \P\Big( \exists v=\bar{x}+t \bar{y} \in \cM' \text{ such that } \norm{\wt{A}_n^{D, m}(\bar{x}+t\bar{y})}_2 \le c \vep_0  \sqrt{pn} \cdot t\Big)   &\le  |\cM|  \max_{\bar{x}+t \bar{y} \in \cM'}  \cL(\wt{A}_n^{D,m}v, t c \vep_0 \sqrt{pn})\notag\\
  &\le  |\cM|  \vep_0^m.\label{eq:cM_union_bound}
  \end{align}
  Assume first that $\frac{\bar{c}L}{\sqrt{n}} \le 1$.
 Using \eqref{eq:cM_cardinality_new}, we see that the \abbr{RHS} of \eqref{eq:cM_union_bound} is bounded by
  \beq\label{eq:cM_bound_2}
    \f{1}{\vep_0} \exp \left(-n \left[ \frac{m}{n} \log \frac{1}{\vep_0} - \log 6 -\log \f{1}{c}- \frac{2r^2}{np} \log \f{1}{c'_{\ref{p: norm on S_L}}} - 4r^2 (\ol{c}_1+ c''_{\ref{p: norm on S_L}})\right] \right).
  \eeq
  Shrinking, if necessary, the constants $c_{\ref{p: norm on S_L}}$, and $c'_{\ref{p: norm on S_L}}$, we have $\f{1}{4}\log (1/\vep_0) \ge \log (6/c)$.  Choose $c''_{\ref{p: norm on S_L}}$ and $\ol{c}_1$ small enough such that $4r^2 (\ol{c}_1 +c''_{\ref{p: norm on S_L}}) \le \f{1}{16}\log (1/\vep_0)$. Finally, noting that $np \ra \infty$, and $m/n \ge 1/2$, we deduce that \eqref{eq:cM_bound_2} is bounded by $\exp(-\ol{c}'n)$, for some small positive constant $\ol{c}'$.

\vskip10pt

\noindent
 Otherwise, if $\frac{\bar{c}L}{\sqrt{n}} >1 $, using the facts that $(n-m)/n \le c^*_{\ref{p: norm on S_L}}(K+R)^2/(pn)$ and $L \le \exp(c''_{\ref{p: norm on S_L}} pn/(K+R)^2)$, and choosing $c^*_{\ref{p: norm on S_L}}$ sufficiently small, and shrinking $c''_{\ref{p: norm on S_L}}$, if necessary, we deduce that the right hand side of \eqref{eq:cM_union_bound}, is bounded by
 \begin{align}
  &\left( \frac{3\ol{c}L}{\sqrt{n}} \right)^{n+1} \left(\frac{c'_{\ref{p: norm on S_L}} \sqrt{n}}{L} \right)^m\notag\\
  = & \, \exp \left(-(n+1) \cdot \left[\log \frac{1}{3\ol{c} c'_{\ref{p: norm on S_L}}}- \frac{n+1-m}{n+1} \log \left( \frac{L}{c'_{\ref{p: norm on S_L}} \sqrt{n}} \right) \right] \right) \notag\\
  \le  & \, \exp \left(-n \cdot \left[\log \frac{1}{3\ol{c} c'_{\ref{p: norm on S_L}}} - c^*_{\ref{p: norm on S_L}} c''_{\ref{p: norm on S_L}} + \f{c^*_{\ref{p: norm on S_L}}}{np}  \log \f{1}{c'_{\ref{p: norm on S_L}}} +\f{c''_{\ref{p: norm on S_L}}p}{(K+R)^2} \right]\right) \le \exp(-\ol{c}''n),
    \end{align}
  for another small positive constant $\ol{c}''$.

 \noindent
  Therefore we have obtained that
  \[
 \P\Big(  \forall \bar{x}+ t \bar{y} \in \cM' \  \norm{\wt{A}_n^{D, m}(\bar{x}+t \bar{y})}_2 \ge c \vep_0  \sqrt{pn} \cdot t ) \ge 1 - \exp(-c''n),
  \]
  where $c''=\min\{\ol{c}',\ol{c}''\}$. Now we restrict ourselves on this set with very large probability. Consider any $w \in S_L$. By our construction of $\cM'$ there exists $\bar{x}+ t \bar{y} \in \cM'$, and $j \in [n]$, such that it satisfies \eqref{approx v inside}. Therefore, on this set of large probability we have that
  \begin{align*}
    \norm{\wt{A}_n^{D,m}w}_2
    &\ge \norm{\wt{A}_n^{D,m}(\bar{x}+t \bar{y})}_2 \\
    &- (\|{\wt{A}_n^m}\|+\norm{D_n}) \cdot
        \Big( \norm{w_{I^c(w)\cup \{j\}}-\bar{x}}_2+ \norm{w_{I(w)\setminus\{j\}}- \norm{w_{I(w)\setminus\{j\}}}_2 \bar{y}}_2\\
        & \qquad \qquad \qquad \qquad \qquad \qquad \qquad \qquad \qquad \qquad\qquad \qquad \qquad
        + \Big| \norm{w_{I(w)\setminus\{j\}}}_2  -t \Big|\Big) \\
    &\ge c \vep_0 t \sqrt{pn}
    -(K+R) \sqrt{pn}  \cdot
        \left( \frac{2c\rho \vep_0}{5 (K+R)}+ \frac{4\sqrt{\log(\sqrt{\delta_0 p} \cdot 2L)}}{\sqrt{\delta_0 p}L} \cdot \norm{w_{I(w)}}_2
        \right)
  \end{align*}
   Recall that for $w \in S_L \subset W$, we have $\norm{w_{I(w)\setminus\{j\}}}_2 \ge \rho/2$. This and \eqref{approx v inside} imply
  \[
    t- \f{2 \rho }{5} \ge \norm{w_{I(w)\setminus\{j\}}}_2 - \frac{c\rho \vep_0}{5 (K+R)}- \f{2 \rho }{5} \ge \frac{1}{20} \norm{w_{I(w)\setminus\{j\}}}_2.
  \]
 Combining this with the previous inequality, we obtain
  \[
    \norm{\wt{A}_n^{D, m}w}_2
    \ge  {\sqrt{pn}  \norm{w_{I(w)\setminus\{j\}}}_2}
              \left( \f{c \vep_0}{20} - (K+R)  \frac{4\sqrt{\log(\sqrt{\delta_0 p} \cdot 2L)}}{\sqrt{\delta_0 p}L}
        \right).
  \]
  If $\frac{c_{\ref{p: norm on S_L}}' \sqrt{n}}{L} \ge \f{c_{\ref{p: norm on S_L}}}{\sqrt{r}}$, then $\vep_0=\f{c_{\ref{p: norm on S_L}}}{\sqrt{r}}$. In this case using the fact that $L \ge r p^{-1/2}$, and observing that for positive constants $\alpha_1, \alpha_2$, the function $x \mapsto \f{\sqrt{\log(\alpha_1 x)}}{\alpha_2 x}$ is a decreasing function for large values of $x$, we obtain
  \[
   \f{c \vep_0}{20} - (K+R) \cdot \frac{4\sqrt{\log(\sqrt{\delta_0 p} \cdot 2L)}}{\sqrt{\delta_0 p}L}
   \ge \f{c c_{\ref{p: norm on S_L}}}{20\sqrt{r}}- (K +R) \frac{4\sqrt{\log(2 \sqrt{\delta_0} r)}}{\sqrt{\delta_0}r}.
  \]
 Now choosing $r \ge (C_{\ref{p: norm on S_L}}(K+R))^2 $, for a sufficiently large constant $C_{\ref{p: norm on S_L}}$ we obtain
  \[
     \norm{\wt{A}_n^{D,m}w}_2
    \ge  \sqrt{pn}  \cdot \norm{w_{I(w)\setminus\{j\}}}_2 \cdot \frac{cc_{\ref{p: norm on S_L}}}{40\sqrt{r}}  \ge \f{c \vep_0 \rho \sqrt{pn}}{80}.
  \]

\noindent
Otherwise, when  $\frac{c_{\ref{p: norm on S_L}}' \sqrt{n}}{L} \le \f{c_{\ref{p: norm on S_L}}}{\sqrt{r}}$, we have $\vep_0 =\frac{c'_{\ref{p: norm on S_L}} \sqrt{n}}{L}$. Since  $L \le \exp({c}''_{\ref{p: norm on S_L}}pn/(K+R)^2) $, choosing ${c}''_{\ref{p: norm on S_L}}$ sufficiently small, we have
  \begin{align*}
      \f{c \vep_0}{20} - (K+R)  \frac{4\sqrt{\log(\sqrt{\delta_0 p}  2L)}}{\sqrt{\delta_0 p}L}
      &= \vep_0\Big(\f{c}{20}- \f{4(K+R)\sqrt{\log(\sqrt{\delta_0 p}  2L)}}{\vep_0\sqrt{\delta_0  p}L}\Big)\\
      & = \vep_0\Big(\f{c}{20}- \f{4(K+R)\sqrt{\log(\sqrt{\delta_0 p}  2L)}}{c'_{\ref{p: norm on S_L}}\sqrt{\delta_0   n p}}\Big)\\
       &\ge \vep_0\bigg(\f{c}{20}-\f{4(K+R)\sqrt{\log(\sqrt{\delta_0 p }  2\exp( {c}''_{\ref{p: norm on S_L}}pn /(K+R)^2))}}{c'_{\ref{p: norm on S_L}}\sqrt{\delta_0 n p}} \bigg)\\
       &\ge \vep_0\bigg(\f{c}{20}-\f{4(K+R)\sqrt{\log(\sqrt{\delta_0 p }  2)+  {c}''_{\ref{p: norm on S_L}}pn /(K+R)^2)}}{c'_{\ref{p: norm on S_L}}\sqrt{\delta_0 n p}} \bigg)\\
       &\ge \vep_0\Big(\f{c}{20}-\f{4\sqrt{{c}''_{\ref{p: norm on S_L}}}}{c'_{\ref{p: norm on S_L}}\sqrt{\delta_0}}\Big) \ge \f{c}{40} \vep_0.
  \end{align*}
  Therefore, in this case
  \[
     \norm{\wt{A}_n^{D, m}w}_2
    \ge  \sqrt{pn}  \cdot \norm{w_{I(w)\setminus\{j\}}}_2 \cdot \f{c}{40} \vep_0 \ge \f{c}{80} \vep_0 \rho \sqrt{pn}.
  \]
  Thus combining both the cases, and setting $\wt{c}_{\ref{p: norm on S_L}}=\f{c}{80}$, the proof is completed.
 \end{step}
 \end{proof}


 \begin{rmk}\label{rmk:omega_complex_small_lcd}
Proof of  Proposition \ref{p: norm on S_L} crucially uses $\vep$-net argument. If we allow $D_n$ to be a complex-valued diagonal matrix in Proposition \ref{p: norm on S_L}, then the sets $S_L$ become subsets of the complex unit sphere, whose real dimension is $2n-1$ instead of $n-1$.
 hence, \eqref{eq:cM_cardinality} changes to
 \begin{align}
    |\cM |
    &\le 2 c^{-n}\binom{n}{r^2p^{-1}+1}  \left( \frac{30 (K+R)}{\vep_0 \rho} \right)^{2(r^2 p^{-1}+2)}
     \left(2 + \frac{\ol{c}L}{\sqrt{n}} \right)^{2(n-r^2 p^{-1})} \notag.
  \end{align}
  The reader can easily convince her/himself that $|\cM| \vep_0^m$ is not exponentially small anymore. Thus the proof breaks down in the complex case.
  Since the extension of Proposition \ref{p: norm on S_L} to complex-valued $D_n$ is quite involved, we defer it to \cite{BR} where we use it to derive the circular law.
 \end{rmk}

 \section{Proof of Theorem \ref{thm: smallest singular + norm}}\label{sec:main_thm_proof}

 \noindent
 In this section we combine the results from Section \ref{sec: compressible}, and Section \ref{sec: small LCD} to prove Theorem \ref{thm: smallest singular + norm}.

 \begin{proof}[Proof of Theorem \ref{thm: smallest singular + norm}]
Recalling that $\Omega_K= \{\norm{\bar{A}_n} \le K \sqrt{np}\}$, we note that for any $\vartheta>0$,
 \begin{align}\label{eq:inf_split}
 & \P\Big( \{s_{\min}(\bar{A}_n+D_n) \le \vartheta\}  \cap  \Omega_K \Big) \notag\\
  \le& \,  \P\Big( \Big\{\inf_{x \in V^c} \norm{(\bar{A}_n+D_n)x}_2 \le \vartheta \Big\}  \cap\Omega_K \Big)
  + \P\Big( \Big\{\inf_{x \in V} \norm{(\bar{A}_n+D_n)x}_2 \le \vartheta  \Big\} \cap  \Omega_K \Big),
 \end{align}
 where
  \[
  V:=S^{n-1} \setminus \Big( \text{Comp}(c_{\ref{p: dominated and compressible}}n, \rho) \cup \text{Dom}(c_{\ref{p: dominated and compressible}}n, (C_{\ref{p: dominated and compressible}}(K+R))^{-4}) \Big),
 \]
 and $\rho$ as in Proposition \ref{p: dominated and compressible}. Using Proposition \ref{p: dominated and compressible} with $M=c_{\ref{p: dominated and compressible}}n$,  we obtain that
 \[
   \P\Big( \inf_{x \in V^c} \norm{(\bar{A}_n+D_n)x}_2 \le \ol{C}_{\ref{p: dominated and compressible}}(K+R) \rho \sqrt{np},  \,  \norm{\bar{A}_n} \le K \sqrt{pn} \Big)
    \le \exp(-\ol{c}_{\ref{p: dominated and compressible}}np).
 \]
Therefore it only remains to find an upper bound on the second term in the \abbr{RHS} of \eqref{eq:inf_split}. Now using Lemma \ref{l: via distance}, we see that to find an upper bound of
\[
 \P\Big( \Big\{\inf_{x \in V} \norm{(\bar{A}_n+D_n)x}_2 \le \vep \rho^2 \sqrt{\frac{p}{n}}\Big\} \cap \Omega_K\Big)
 \]
 is enough to find the same for
 \[\P \Big( \Big\{\dist(\bar{A}_{n,j},H_{n,j}) \le \rho \sqrt{p} \vep\Big\} \cap \Omega_K \Big) \text{ for a fixed } j,\]
 where $\bar{A}_{n,j}$ are now columns of $(\bar{A}_n+D_n)$ (see also Remark \ref{rmk:l via distance}). As these estimates are the same for different $j$, so we consider only  $j=1$.
 Let $\wt{A}_n^D$ be the $(n-1) \times n$ matrix whose rows are the columns $\bar{A}_{n,2} \etc \bar{A}_{n,n}$. Note that it is the matrix $\wt{A}_n^{D,m}$ defined in Proposition \ref{p: norm on S_L}, for $m=n-1$. For ease of writing, hereafter we omit the superscript $m$. Let $v \in S^{n-1} \cap \text{Ker}(\wt{A}_n^D)$, where $\text{Ker}(\wt{A}_n^D):=\{x\in \R^n | \wt{A}_n^D x=0\}$. Since
 \[
    \dist(\bar{A}_{n,1},H_{n,1}) \ge |\pr{v}{\bar{A}_{n,1}}|,
 \]
it is enough to prove that
 \[
   \P\Big(\Big\{\exists v \in S^{n-1} \text{ such that } \wt{A}_n^\omega v=0  \ \text{and} \ |\pr{A_{n,1}}{v}| \le \rho\vep \sqrt{p}\Big\}\cap  \Omega_K \Big) \le   \vep+ \exp (-cpn/(K+R)^2),
 \]
 for some positive constant $c$.
 We partition $S^{n-1}$ into the set of compressible and dominated vectors and its complement again.
 {Setting $Q=(2\ol{C}_{\ref{p: norm on S_L}}(K+R))^{12} p^{-1}$, denote
 \[
  \ol{W}=S^{n-1} \setminus \Big( \text{Comp}(Q, \rho) \cup \text{Dom}(Q, (\ol{C}_{\ref{p: norm on S_L}}(K+R))^{-4}) \Big).
 \]
 }
 Then
 \begin{align}
    \P\Big(\Big\{\exists v \in S^{n-1} \text{ such that } &\wt{A}_n^D v=0  \ \text{and} \ |\pr{\bar{A}_{n,1}}{v}| \le \rho\vep \sqrt{p}\Big\}\cap  \Omega_K \Big) \notag\\
  & \le \P\Big(\Big\{\exists v \in \ol{W}^c \text{ such that } \wt{A}_n^D v=0\Big\}\cap  \Omega_K \Big) \notag\\
  &  +  \P\Big(\Big\{\exists v \in \ol{W} \text{ such that } \wt{A}_n^D v=0  \ \text{and} \ |\pr{\bar{A}_{n,1}}{v}| \le \rho \vep \sqrt{p}\Big\}\cap  \Omega_K \Big). \label{eq:W_W^c}
 \end{align}
This time, we apply Proposition \ref{p: dominated and compressible} with $M=Q$.
 It yields that the first term in the \abbr{RHS} of \eqref{eq:W_W^c} does not exceed $\exp(-\ol{c}_{\ref{p: dominated and compressible}}np)$. Although this proposition was proved for $n \times n$ matrices, the same proof would work for $(n-1) \times n$ matrices as well (see also Remark \ref{rmk:comp_dom_modification}).

  Let $w \in \ol{W}$. Since $w \notin \text{Dom}(Q, (\ol{C}_{\ref{p: norm on S_L}}(K+R))^{-4}) $,
 \[
  \norm{w_{I(w)}}_2 \ge (\ol{C}_{\ref{p: norm on S_L}}(K+R))^{-4} \sqrt{Q} \norm{w_{I(w)}}_{\infty} \ge 4(\ol{C}_{\ref{p: norm on S_L}}(K+R))^{2} p^{-1/2} \norm{w_{I(w)}}_{\infty}.
 \]
 We also recall that, for $p \le c (K+R)^{-2}$, and a sufficiently small $c$ (see \eqref{eq:lower_bd_inf_i}), we have
 \[
 \norm{w_{I(w)\setminus\{i\}}}_2 \ge \f{1}{2}\norm{w_{I(w)}}_2, \quad \text{ for } i \in [n].
 \]
 Hence, Proposition \ref{l: LCD bound} yields
 $$\inf_{i \in [n]}D\bigg(\frac{w_{I(w)\setminus\{i\}}}{\norm{w_{I(w)\setminus\{i\}}}_2}\bigg)
 \ge (\ol{C}_{\ref{p: norm on S_L}}(K+R))^{2}p^{-1/2} \text{ for all } w \in \ol{W}.$$

  To estimate the second term in \eqref{eq:W_W^c}, decompose $\ol{W}$ as $\ol{W}=W_1 \cup W_2$, where
 \[
  W_1:=\Big\{w \in W \mid \inf_{i \in [n]} D\bigg(\frac{w_{I(w)\setminus\{i\}}}{\norm{w_{I(w)\setminus\{i\}}}_2}\bigg) \le \exp (c''_{\ref{p: norm on S_L}}pn/(K+R)^2) \Big\} \quad \text{and} \quad W_2:=\ol{W}\setminus W_1.
 \]
 Decompose $W_1$ further as
 \beq\label{eq: W_1 via S_L}
  W_1 = \bigcup_{(\ol{C}_{\ref{p: norm on S_L}}(K+R))^{2}p^{-1/2} \le L \le \exp (c''_{\ref{p: norm on S_L}}pn/(K+R)^2)}^\star S_L,
 \eeq
 where the $\star$ denotes that the union is taken over $L=2^k$ for $k \in \N$.
 Then by Proposition \ref{p: norm on S_L},
 \begin{align*}
  &\P \Big(\Big\{\exists v \in W_1 \text{ such that } \wt{A}_n^D v=0 \Big\} \bigcap \Omega_K \Big) \\
  &\le \sum_{(\ol{C}_{\ref{p: norm on S_L}}(K+R))^{2}p^{-1/2} \le L \le \exp (c''_{\ref{p: norm on S_L}}pn/(K+R)^2)}^\star
   \P \Big(\Big\{\exists v \in S_L \text{ such that } \wt{A}_n^D v=0    \Big\} \bigcap \Omega_K \Big) \\
  &\le \frac{c''_{\ref{p: norm on S_L}} pn}{ (K+R)^2} \cdot \exp(-{\wt{c}_{\ref{p: norm on S_L}}}n)
   \le  \exp(-\f{\wt{c}_{\ref{p: norm on S_L}}}{2}n).
 \end{align*}
 Thus, to finish the proof of Theorem \ref{thm: smallest singular + norm}, it is enough to estimate
 \[
  \P \Big(\exists v \in W_2 \text{ such that } \wt{A}_n^D v=0 \text{ and } |\pr{\bar{A}_{n,1}}{v}| \le \vep \rho \sqrt{p} \Big)
 \]
 Note that $v$ is defined by $\bar{A}_{n,2} \etc \bar{A}_{n,n}$, so it is independent of $\bar{A}_{n,1}$. Condition on $\bar{A}_{n,2} \etc \bar{A}_{n,n}$ such that $v \in W_2$ for the matrix $\wt{A}_n^D$ formed by these columns. We may now consider $v$ as a fixed vector satisfying
 \[
   \inf_{j \in [n]}D\bigg(\frac{v_{I(v)\setminus\{j\}}}{\norm{v_{I(v)\setminus\{j\}}}_2}\bigg) \ge \exp (c''_{\ref{p: norm on S_L}}pn/(K+R)^2).
 \]
 Let $\bar{A}_{n,1}^{I(v)\setminus\{1\}}$ be vector obtained from $\bar{A}_{n,1}$ by keeping the coordinates corresponding to the set $I(v)\setminus\{1\}$. Since $v \notin \text{Comp}(M,\rho)$, we have $\norm{v_{I(v)\setminus\{1\}}}_2 \ge \rho/2$. Thus using Proposition \ref{prop:lcd} we obtain that
 \begin{align*}
 \P ( |\pr{\bar{A}_{n,1}}{v}| \le \vep \rho \sqrt{p})
 & \le \sup_{y \in \R} \P\Big( \Big|\pr{A_{n,1}^{I(v)\setminus\{1\}}}{v_{I(v)\setminus\{1\}}}-y\Big| \le \vep \rho \sqrt{p}\Big)\\
  &\le2 C_{\ref{prop:lcd}} \left( \vep+\frac{1}{\sqrt{p} D(v_{I(v)\setminus\{1\}}/\norm{v_{I(v)\setminus\{1\}}}_2)} \right)\\
  &
  \le 2 C_{\ref{prop:lcd}} \left( \vep+\frac{1}{\sqrt{p} } \exp (-c''_{\ref{p: norm on S_L}}pn/(K+R)^2) \right) \le 2C_{\ref{prop:lcd}}  \vep+ \exp \left(-\f{c''_{\ref{p: norm on S_L}}pn}{2(K+R)^2}\right),
 \end{align*}
where the last inequality here follows from the assumption $p \ge \frac{ \log n}{n}$.
 Replacing $\vep$ by $\vep/(2C_{\ref{prop:lcd}})$, we obtain
 \[
 \P ( |\pr{\bar{A}_{n,1}}{v}| \le (2C_{\ref{prop:lcd}})^{-1}\vep \rho \sqrt{p})
  \le   \vep+ \exp \left(-\f{c''_{\ref{p: norm on S_L}}pn}{2(K+R)^2}\right),
 \]
 which completes the proof of the theorem.
 \end{proof}

\section{Estimates of the spectral norm}\label{sec:norm}

In this section we prove bounds on the spectral norm of sparse random matrices with heavy-tailed entries (Theorem \ref{lem:norm_heavy_tail}) and with sub-Gaussian entries (Theorem \ref{lem:norm_subgaussian}). Building on those theorems we complete the proof of Corollary \ref{thm: smallest singular heavy tail} and Corollary \ref{thm: smallest singular}. We then provide an outline of the proof for the spectral norm of sparse random matrices, with entries satisfying \eqref{eq:beta}, in Remark \ref{rmk:norm_general_beta}. 

To prove Theorem \ref{lem:norm_heavy_tail}
we use the following result of Seginer \cite{S}.
 \begin{thm}{\em {\bf (Seginer)}}\label{thm:seginer}
  Let ${A}_n$ be a random matrix with i.i.d. centered  entries whose columns are denoted by ${A}_{n,1} \etc {A}_{n,n}$. Then, there exists an absolute constant $C_{\ref{thm:seginer}}$, such that for $1 \le q \le 2 \log n$
  \[
    \E \norm{{A}_n}^q \le C_{\ref{thm:seginer}}^q \E \max_{j \in [n]} \norm{{A}_{j,n}}_2^q.
  \]
 \end{thm}

\begin{proof}[Proof of Theorem \ref{lem:norm_heavy_tail}]
 Fix $t  \ge 1$. Then  Markov's inequality yields
 \[
  \P(|a_{ij}| > t \sqrt{np}) \le \frac{\E |a_{ij}|^q}{(t\sqrt{np})^q}
  \le \left( \frac{\bar{K}}{t \sqrt{np}} \right)^q \cdot p.
 \]
 Hence, using the fact that $p = \Omega(n^{-\alpha})$, and using the definition of $q$, we have
 \beq\label{eq:Lambda_n}
  \P(\exists i,j \in [n] \text{ such that} \ |a_{ij}|> t \sqrt{np})
  \le n^2 \cdot \left( \frac{\bar{K}}{t \sqrt{np}} \right)^q \cdot p
  \le C' \left( \frac{\bar{K}}{t } \right)^q,
 \eeq
 for some constant $C'$, depending only on $\alpha$.
 Now setting $y_{ij}= a_{ij} \cdot \mathbf{1}(|a_{ij}| \le t \sqrt{np})$, we define
 the random variables
 \[
   z_{ij}:= \left( \frac{y_{ij}}{2t\sqrt{np}} \right)^2
    - \E  \bigg[\left( \frac{y_{ij}}{2t\sqrt{np}} \right)^2\bigg].
 \]
Upon observing $q \ge 4$, we note that
 \[
  \E z_{ij}=0, \quad
  \E z_{ij}^2 \le \left( \frac{1}{2t\sqrt{np}} \right)^4 \cdot B p,
  \quad \text{and } |z_{ij}| \le 1 \text{ a.s.,}
 \]
for some constant $B$, depending only on the fourth moment of $\{\xi_{i,j}\}$. Then by Bennett's inequality, for any $s \ge 1$,
 \begin{align*}
  \P \left( \sum_{i=1}^n y_{ij} ^2 \ge 8s^2 t^2 np \right)
 & \le \P \left( \sum_{i=1}^n \left( \frac{y_{ij}}{2t\sqrt{np}} \right)^2 - n \E \bigg[\left( \frac{y_{ij}}{2t\sqrt{np}} \right)^2\bigg]\ge  s^2 \right) \\
   & = \P \left( \sum_{i=1}^n z_{ij} \ge  s^2 \right)\\
  &\le \exp \left[- \frac{s^2}{8} \cdot \log \left( \frac{s^2}{2 n \E z_{ij}^2} \right) \right]
  \le \left( 8 B^{-1}s^2 t^4 np \right)^{-s^2/8}.
 \end{align*}
 Recalling that $p=\Omega(n^{-\a})$, we obtain that for some positive constants $c$ and $C$, depending on $\alpha$, and the fourth moment of $\{\xi_{i,j}\}$,
 \begin{align*}
  \P \left( \exists j \in [n] \text{ such that }  \sum_{i=1}^n y_{ij} ^2 \ge 8s^2 t^2 np \right) & \le n  \left( 8B^{-1} s^2 t^4 np \right)^{-s^2/8} \\
  & \le n^{-(1-\a)s^2/8+1}
  \le n^{-cs^2}
 \end{align*}
 for all $s \ge C$.
 Therefore,
 \begin{align*}
  &\P \left( \exists j \in [n] \text{ such that }  \sum_{i=1}^n a_{ij} ^2 \ge 8s^2 t^2 np \right)\\
   &\le n^{-cs^2}+ \P (\exists i,j \in [n]\text{ such that }  a_{ij} \neq y_{ij}) \le n^{-cs^2}+  C' \left( \frac{\bar{K}}{t } \right)^q.
 \end{align*}
 Let $r<q$, and choose $\b>0$ so that $\frac{q}{1+\b}>r$. Setting $s=\t^{\b/(1+\b)}, \ t=\t^{1/(1+\b)}$, we conclude that for any $\t \ge C^{(1+\b)/\b}$,
 \[
   \P \left( \exists j \in [n] \text{ such that }  \norm{\bar{A}_{j,n}}_2 \ge 2 \t \sqrt{np} \right)
   \le n^{-c \t^{2\b/(1+\b)}}+C' \bar{K}^q \t^{-q/(1+\b)}.
 \]
 Upon using integration by parts from the last inequality we deduce $\E \max_{j \in [n]} \norm{A_j}_2^r \le \bar{C} (\sqrt{np})^r$, for some $\bar{C}$ depending on $\a, \bar{K}, r$, and the fourth moment of $\{\xi_{i,j}\}$, which in combination with Seginer's theorem proves part (i).

 \vskip 0.2in
 To prove (ii), take $\rho \in (r,q)$, and  let $\xi$ be a symmetric random variable  such that $\P(|\xi|>t) \sim t^{-\rho}$ as $ t \to \infty$.
 Let $\xi_{ij}, \ i,j \in [n]$ be independent copies of $\xi$.

 By Chernoff's inequality, for any $c \in (0,1)$ there exists positive constants $c', c''$ such that
 \[
  \P \left( \sum_{i,j=1}^n \d_{ij} \le c n^2 p \right) \le \exp(-c' n^2 p)
 \le\exp(-c'' n^{2-\a}).
 \]
 Now conditioning  on the event $E$ that $\sum_{i,j=1}^n \d_{ij} \ge c n^2 p$, we see that for a sufficiently large $t$,
 \begin{align*}
  &\P \left(\max_{i,j \in [n]}  |a_{ij}| \le t \sqrt{np} \Big| E \right)\\
  \le& \left(1-C (t \sqrt{np})^{-\rho} \right)^{c p n^2} \le \exp \left( -C (t \sqrt{np})^{-\rho} c p n^2 \right)
  \le \exp \left( -C' t^{-\rho} n^{-\rho(1-\a)/2+2-\a} \right),
 \end{align*}
where $C'=Cc$. Removing the conditioning, we obtain
 \begin{align*}
  \P (\norm{\bar{A}_n} \le t \sqrt{np})
  &\le \P (\max_{i,j \in [n]} |a_{ij}| \le t \sqrt{np}) \\
  &\le \exp \left( -C' t^{-\rho} n^{-\rho(1-\a)/2+2-\a} \right)
  + \exp(-c'' n^{2-\a}),
 \end{align*}
 Setting $t=n^\nu$ for some $\nu>0$ such that
 \[
  \rho \left( \nu+ \frac{1-\a}{2} \right)<2-\a
 \]
 completes the proof.
 \end{proof}
 
 \bigskip
Since in Theorem \ref{thm: smallest singular + norm} we consider matrix with i.i.d.~off-diagonal entries, and zero diagonal entries, we cannot directly apply Theorem \ref{lem:norm_heavy_tail} to prove Corollary \ref{thm: smallest singular heavy tail}. If we are able to show that $\norm{\Lambda_n}=O(\sqrt{np})$, with large probability, then conditioning on $\Lambda_n$ we can complete the proof of Corollary \ref{thm: smallest singular heavy tail}. To prove $\norm{\Lambda_n}=O(\sqrt{np})$, with large probability, we note that the operator norm of any diagonal matrix is the maximum of its entries. Thus the proof completes using Markov's inequality, and the union bound (for example, one can proceed as in \eqref{eq:Lambda_n}). 

\begin{proof}[Proof of Theorem \ref{lem:norm_subgaussian}]
Let $\xi'_{ij}, \ i,j \in [n]$ be independent copies of $\xi_{ij}, \ i,j \in [n]$, and let ${A}_n'$ and $B_n$ be the matrices with entries $a'_{ij}=\d_{ij}\xi_{ij}'$, and $b_{ij}= \d_{ij} \eta_{ij}$, respectively, where $\eta_{ij}:=\xi_{ij}-\xi'_{ij}$. Further let us denote by $\E_{\xi}$ the expectation with respect to $\xi$, conditioned on $\bm{\d}:=(\d_{ij})_{i,j \in [n]}$.
 Let $q \ge 1$ be an even integer. By Jensen's inequality,
 \begin{equation}  \label{eq: symm}
   \E_\xi \norm{{A}_n}^q = \E_\xi \norm{{A}_n- \E_{\xi'} {A}_n'}^q
   \le \E_\eta \norm{B_n}^q.
 \end{equation}
 Let $g_{ij}, \ i,j \in [n]$ be independent $N(0,1)$ random variables. Since $\xi_{ij}-\xi_{ij}'$ is a sub-Gaussian random variable, there exists a constant $C_1>0$, depending on the sub-Gaussian norm of $\{\xi_{ij}\}$, such that $\E |\eta_{ij}|^q \le \E |C_1g_{ij}|^q$ for all $q \ge 1$.
 Let $W_n$ be the $n \times n$ random matrix with entries $w_{ij}=\d_{ij} g_{ij}$.
 Since
 \begin{equation}  \label{eq: norm to trace}
   \E_\eta \norm{B_n}^q \le \E_\eta \text{Tr} \left( (B_nB_n^*)^{q/2} \right),
 \end{equation}
 where the last expression is a polynomial of the even moments of $\eta_{ij}$ with non-negative coefficients, we have that
 \begin{equation}  \label{eq: trace to norm}
    \E_\eta \text{Tr} \left( (B_nB_n^*)^{q/2} \right)
    \le  C_1^q \E_g \text{Tr} \left( (W_nW_n^*)^{q/2} \right)
    \le C_1^q \cdot n \E_g \norm{W_n}^q.
 \end{equation}
 The last inequality above uses that $\text{Tr} \left( (W_nW_n^*)^{q/2} \right)=\sum_{j=1}^n \lambda_j^{q/2}(W_nW_n^*)$, and all eigenvalues $\lambda_j(W_nW_n^*)$ satisfy $|\lambda_j(W_nW_n^*)| \le \norm{W_n}^2$.

  To estimate $\E \norm{W}^q$, we use the following recent result of Bandeira and van Handel \cite{BH}.
\begin{lem}[{\cite[Theorem 3.1]{BH}}]\label{lem:norm_BH}
Let $\bm{X}$ be a $n \times m$ rectangular matrix with $\bm{X}_{ij}= b_{ij}g_{ij}$, where $g_{ij}$ are i.i.d.~$N(0,1)$. Let
\[
\sigma_1:=\max_i \sqrt{\sum_jb_{ij}^2}, \quad \sigma_2:=\max_j \sqrt{\sum_i b_{ij}^2}, \quad \sigma_*:= \max_{i,j} |b_{ij}|.
\]
Then
\[
\E\norm{\bm{X}} \le (1+\vep) \left\{\sigma_1 + \sigma_2 + \f{5}{\sqrt{\log (1+\vep)}}\sigma_* \sqrt{\log (n \wedge m)}\right\},
\]
for any $\vep \in (0,1/2)$.
\end{lem}

  First, let us denote $\Omega$ to be the event that for all $i \in [n]$,
  \[
    \sum_{j=1}^n \d_{ij} \le \ol{C} pn \quad \text{and} \quad
    \sum_{j=1}^n \d_{ji} \le \ol{C} pn,
  \]
 for some $\ol{C} \ge 2$. Since $p \ge C_0 \f{\log n}{n}$, by Chernoff's inequality, and using the union bound, we see that we can choose the constant $C_0$ large enough, such that $\P(\Omega^c) \le e^{-cpn}$, for some $c >0$, depending only on $C_0$.
  Assuming that $\bm{\d} \in \Omega$ and conditioning on $\bm{\d}$, using Lemma \ref{lem:norm_BH}, we have
    \[
    \E \left[ \norm{W_n} \mid \bm{\d} \right] \le 2 \left(\sqrt{\ol{C}pn}+ C^*\sqrt{\log n}\right)
    \le  \sqrt{C' pn},
  \]
where $C^*$ is some absolute constant , and $C'= 4(C^*)^2 \ol{C}$.  Conditionally on $\bm{\d}$, $\norm{W_n}$ can be viewed as a 1-Lipschitz function on $\R^{n^2}$ equipped with the standard Gaussian measure.
  Using the Gaussian concentration inequality (for example, see \cite{ledoux}), we obtain
  \[
    \P \left[ \norm{W_n} \ge \E \left[ \norm{W_n} \mid \bm{\d} \right]+ t \mid \bm{\d} \right]
    \le \wt{C}\exp(-c't^2)
  \]
  for some absolute constants $\wt{C}, c'>0$, and any $t>0$.
  Hence,
  \begin{align*}
     \E_g \norm{W_n}^q
     &\le (C' pn)^{q/2}+ \int_{\sqrt{C' pn}}^\infty q s^{q-1} \P \left[ \norm{W_n} \ge s \mid \bm{\d} \right] \, ds \\
    &\le (C' pn)^{q/2}+ (C^{''}q)^{q/2},
  \end{align*}
   for some absolute constant $C''$. Setting $q=pn$ (or taking the closest even number), we see that this inequality in combination  with \eqref{eq: symm}, \eqref{eq: norm to trace}, and \eqref{eq: trace to norm} yields
  \[
    \E_\xi \norm{A_n}^{pn} \le n \cdot (C_2pn)^{pn/2}
     \le   (C_2^2 pn)^{pn/2},
  \]
where we used the condition $p \ge  \frac{\log n}{n}$ to absorb the factor $n$, and $C_2$ is a positive constant depending on $C_0$ and the sub-Gaussian norm of $\{\xi_{ij}\}$. Now if we choose $C_{\ref{lem:norm_subgaussian}} >C_2^2$, then Markov's inequality implies that for any $\bm{\d} \in \Omega$, there exists a small positive constant $c_{\ref{lem:norm_subgaussian}}$, depending on $C_{\ref{lem:norm_subgaussian}}$, such that
  \[
    \P \left[ \norm{A_n} \ge C_{\ref{lem:norm_subgaussian}} \sqrt{pn}  \mid \bm{\d} \right]
    \le \exp(-c_{\ref{lem:norm_subgaussian}}pn).
  \]
Finally, shrinking $c_{\ref{lem:norm_subgaussian}}$ further we have
  \[
   \P \left( \norm{A_n} \ge C_{\ref{lem:norm_subgaussian}} \sqrt{pn}   \right)
   \le \max_{\bm{\d} \in \Omega} \P \left[ \norm{A_n} \ge C_{\ref{lem:norm_subgaussian}} \sqrt{pn}  \mid \bm{\d} \right]
   + \P(\Omega^c)
   \le \exp(-c_{\ref{lem:norm_subgaussian}}pn).
  \]
  This completes the proof.
  \end{proof}
  
  \bigskip
  We have already seen that we cannot apply Theorem \ref{thm: smallest singular + norm} and Theorem \ref{lem:norm_subgaussian} directly to prove Corollary \ref{thm: smallest singular}. We also need to show that $\norm{\Lambda_n}=O(\sqrt{np})$, with large probability. This can be done very easily repeating steps in the proof of Theorem \ref{lem:norm_subgaussian}. We omit the details. Then proceeding as in the proof Corollary \ref{thm: smallest singular heavy tail}, we complete the proof of Corollary \ref{thm: smallest singular}.

\begin{rmk}\label{rmk:norm_general_beta}
One can extend the results of Theorem \ref{lem:norm_subgaussian} for random variables satisfying \eqref{eq:beta}. To this end, we will use the following result of Latala \cite{L1}.
\begin{lem}
Fix $q \ge 1$, and let $\{\zeta_i\}_{i=1}^n$ be i.i.d.~copies of a non-negative random variable $\zeta$. Then
\[
\norm{\sum_{i=1}^n \zeta_i}_q \sim \sup\left\{\f{q}{s}\left(\f{n}{q}\right)^{1/s}\norm{\zeta}_s: \max\left(1, \f{q}{n}\right) \le s \le q\right\}.
\]
\end{lem}
Fixing $q=  \log n$, for each $j \in [n]$, we apply the above result for $\zeta_i=\xi_{i,j}^2 \delta_{i,j}$. Thus denoting $A_{n,j}$ to be the $j$-th column of $A_n$ we have
\[
\norm{\norm{A_{n,j}}_2}_{2q}^2= \norm{\sum_{i=1}^n \xi_{i,j}^2 \delta_{i,j}}_q \le \sup\left\{C{\log n}\left(\f{np}{\log n}\right)^{1/s} s^{2\be-1}: 1 \le s \le \log n\right\},
\]
 for some absolute constant $C$. Analyzing $f(s):= \f{1}{s}\log \left(\f{np}{\log n}\right)+ (2\be-1) \log s$, for $s \in [1,\log n]$, we note that
\[
\norm{\norm{A_{n,j}}_2}_{2q}^2 \le C \log n \max \left\{\f{np}{\log n}, \left(\f{np}{\log n}\right)^{\f{1}{\log n}} (\log n)^{2 \be -1}\right\} \le e C np,
\]
if $np =\Omega ((\log n)^{2 \be})$. Thus applying Seginer's theorem now for $q=2\log n$ we get that
\[
\E\norm{A_n}^q = n O(\sqrt{np})^q, \text{ when } p=\Omega\left( \f{(\log n)^{2\be}}{n}\right).
\]
Finally applying Markov's inequality we get that for every $s >0$, there exists $K:=K(s)$ such that
\[
\P(\norm{A_n} \ge K \sqrt{np}) \le n^{-s}.
\]

\end{rmk}

\section{Proof of Theorem \ref{thm:bernoulli}}
 \label{sec:smallest singular+non-centered}

In this section we prove Theorem \ref{thm:bernoulli}. Since the entries of the adjacency matrix of an Erd\H{o}s-R\'{e}yni graph have non-zero mean, we first extend Theorem \ref{thm: smallest singular + norm} to allow non-centered random variables.

\begin{thm}  \label{thm: smallest singular + norm + non-centered}
  Let $\bar{A}_n$ be an $n \times n$ matrix with zero on the diagonal and has i.i.d.~off-diagonal entries $a_{i,j}= \delta_{i,j} \xi_{i,j}$, where $\delta_{i,j}, \ i,j \in [n], i \ne j,$ are independent Bernoulli random variables taking value 1 with probability $p_n \in (0,1]$, and $\xi_{i,j}, \ i,j \in [n], i \ne j$ are i.i.d. random variables with unit variance, and finite fourth moment. 
  Fix $K \ge 1$, and let $\bar{\Omega}_K:= \Big\{\norm{\bar{A}_n- \E \bar{A}_n} \le K \sqrt{np_n}\Big\}$. Further fix $R \ge 1$ and let $D_n$ be a real diagonal matrix with $\norm{D_n} \le R\sqrt{np_n}$.
 Then there exist constants $0< c_{\ref{thm: smallest singular + norm + non-centered}}, c'_{\ref{thm: smallest singular + norm + non-centered}}, C_{\ref{thm: smallest singular + norm + non-centered}}, \ol{C}_{\ref{thm: smallest singular + norm + non-centered}} < \infty$, depending on $K, R$, and the fourth moment of $\xi_{i,j}$, such that for any $\vep>0$, and 
 \beq
p_n \ge  \frac{\ol{C}_{\ref{thm: smallest singular + norm + non-centered}} \log n}{n},\label{eq:p_assumption_non-centered}
 \eeq
 \[
  \P \bigg( \Big\{s_{\min}(\bar{A}_n+D_n) \le C_{\ref{thm: smallest singular + norm + non-centered}} \vep \exp \left(-c_{\ref{thm: smallest singular + norm + non-centered}} \frac{\log (1/p_n)}{\log (np_n)} \right) \sqrt{\frac{p_n}{n}} \Big\} \bigcap   \bar{\Omega}_K \bigg)
  \le \vep +  \exp(-c'_{\ref{thm: smallest singular + norm + non-centered}}np_n).
 \]
\end{thm}

\bigskip

The key ingredients in the proof of Theorem  \ref{thm: smallest singular + norm} are  Proposition \ref{p: dominated and compressible}, and Proposition \ref{p: norm on S_L}. Thus to prove Theorem \ref{thm: smallest singular + norm + non-centered}, we need analogues of Proposition \ref{p: dominated and compressible} and Proposition \ref{p: norm on S_L} for the non-centered case. To this end, we start with the following generalizations of those two results. Before stating these results, for the ease of writing, let us denote $\mu:=\E \xi_{i,j}$, and let $\bm{U}_n^m$ to be the $m \times n$ matrix of all ones. Note in passing that $|\mu|$ is bounded in terms of the fourth moment of $\xi_{i,j}$. Now we are ready to state the results. 

  \begin{prop}    \label{p: norm on S_L y}
  Let $\bar{A}_n, D_n, K$, and $p$ be as in Theorem \ref{thm: smallest singular + norm + non-centered}. Define $\wt{A}_n^{D,m}$, $\wt{A}_n^{m}$, and $S_L$ as in Proposition \ref{p: norm on S_L}. Fix a vector $y \in \R^m$, and $r \ge 1$.  Then there exist small positive constants $ c_{\ref{p: norm on S_L}}, c'_{\ref{p: norm on S_L y}}, \ol{c}_{\ref{p: norm on S_L y}}$, and a large positive constant $\ol{C}_{\ref{p: norm on S_L y}}$, depending on the fourth moment of $\xi_{ij}$, $K$ and $R$, and small positive constants $\wt{c}_{\ref{p: norm on S_L y}},c''_{\ref{p: norm on S_L y}}, c^*_{\ref{p: norm on S_L y}}$, depending on the fourth moment of $\xi_{ij}$, $K$, $R$, and also on $r$, such that, if $r \ge (\ol{C}_{\ref{p: norm on S_L y}}(K+R))^2$, then for $r p^{-1/2} \le L \le \exp(c''_{\ref{p: norm on S_L y}} pn/(K+R)^2)$,  $m \ge n - c^*_{\ref{p: norm on S_L y}}(K+R)^2/p$, we have 
    \[
    \P\Big( \inf_{v \in S_L} \norm{(\wt{A}_n^{D,m} - \mu p \bm{U}_n^m)v -y}_2 \le \ol{c}_{\ref{p: norm on S_L y}} \rho \vep_0 \sqrt{pn} \text{ and } \norm{\wt{A}_n^{D,m} - \E \wt{A}_n^{D,m}} \le K \sqrt{pn}\Big)
    \le \exp(-\wt{c}_{\ref{p: norm on S_L y}}n),
   \]
   where
    \[
     \vep_0 =  \min(c_{\ref{p: norm on S_L y}}/\sqrt{r}, c'_{\ref{p: norm on S_L y}} \sqrt{n}/L).
   \]
 \end{prop}
\begin{proof}
 Recall that the proof of Proposition \ref{p: norm on S_L} is based on an estimate on the L\'{e}vy concentration function, followed by a special $\vep_0$-net argument. That required estimate on the L\'{e}vy concentration function follows from Proposition \ref{prop: lcd_tensorize}, the key to which is Proposition \ref{prop:lcd}. Since, Proposition \ref{prop:lcd} does not require the zero mean condition, it continues to hold in this set-up, and therefore so does Proposition \ref{prop: lcd_tensorize}. Furthermore, we note that the L\'{e}vy concentration function is not affected by the translation of a fixed vector. Therefore, Eqn.~\eqref{eq: tensorized}, can be strengthened to the following,
 \[
     \cL\left((\wt{A}_n^m- \mu p \bm{U}_n^m)v-y, \vep \inf_{j \in [n]}\norm{v_{I\setminus\{j\}}}_2 \sqrt{pm}\right)
    \le C_{\ref{prop: lcd_tensorize}}^m  \left( \vep+\frac{1}{\sqrt{p} \inf_{j \in [n]}D(v_{I\setminus\{j\}}/\norm{v_{I\setminus\{j\}}}_2)} \right)^m,
 \]
 for any $I \subset [n]$. The remaining part of the proof of Proposition \ref{p: norm on S_L} uses $\vep$-net argument. To carry out the same argument here, we need to bound on the operator norm of the matrix under consideration, i.e.~we need a bound on $\|{\wt{A}_n^{D,m}- \mu p \bm{U}_n^m}\|$. However, noting that $\mu p \bm{U}_n^n - \E\bar{A}_n = \mu p I_n$, the required bound is immediate on the event $\|{\wt{A}_n^{D,m} - \E \wt{A}_n^{D,m}}\| \le K \sqrt{pn}$. The rest of the argument remains exactly same, and hence we omit the details.
 \end{proof}
 
 Now we turn to find an analogue of Proposition \ref{p: dominated and compressible} in the non-centered case. Recall that the proof of Proposition \ref{p: dominated and compressible} can be split into two major parts. In the first part we control the infimum over very sparse vectors (and vectors close to those sparse vectors) by showing that there are large blocks inside $\bar{A}_n$ which have only one non-zero element per  row, and in the second part, where we focus on moderately sparse vectors, the proof is carried out by obtaining necessary estimates on the L\'{e}vy concentration function and an $\vep$-net argument. To extend Proposition \ref{p: dominated and compressible} in the non-centered set-up, one would like to extend this scheme for $\bar{A}_n - \E \bar{A}_n$. The first part of the proof of Proposition \ref{p: dominated and compressible}, in particular Lemma \ref{l: pattern}, crucially uses the fact that the entries of $\bar{A}_n$ are of the form $\xi_{i,j}\delta_{i,j}$, where $\delta_{i,j}\sim\dBer(p)$, and $\{\xi_{i,j}\}$ are centered random variable with unit variance. However, the entries of $\bar{A}_n - \E\bar{A}_n$ do not have this required product structure. So, one cannot directly extend Proposition \ref{p: dominated and compressible} in this case.
 
 We overcome this obstacle by using a ``folding'' trick. More elaborately, given any $\bar{A}_n$, a $n \times n$ matrix, we construct two $\lfloor n/2 \rfloor \times n$ matrices, denoted hereafter by $\hat{A}_n^{(1)}$ and $\hat{A}_n^{(2)}$, consisting of the first $\lfloor n/2\rfloor$, and the next $\lfloor n/2\rfloor$ rows of the matrix $\bar{A}_n$, respectively. Further, denote $\hat{A}_n:=\hat{A}_n^{(1)}-\hat{A}_n^{(2)}$. Using the triangle inequality, one can note that $\|{\bar{A}_n x}\|_2^2 \ge \f{1}{2}\|{\hat{A}_n x}\|_2^2$. Therefore, it is enough to control the infimum of $\|{\hat{A}_n x}\|_2$. As we will show below, the advantage of working with $\hat{A}_n$ is that its entries have the required product structure. Therefore, one can hope to use the ingredients of the proof of Proposition \ref{p: dominated and compressible} to obtain the necessary lower bound on the infimum. However, we should note that the number of rows of the matrix under consideration is reduced by one half from the centered case, which worsens the probability bounds. Nevertheless, we can carry out the above approach for treating sparse vectors as well as the vectors close to sparse since the sizes of the nets for such sets depend on the size of the support which is much smaller than $n$. 
 
 Before formally stating the result, let us introduce one more notation: For $D_n$ a $n \times n$ diagonal matrix, define $\hat{D}_n^{(1)}$, and $\hat{D}_n^{(2)}$ to be the matrices consisting of the first $\lfloor n/2\rfloor$, and next $\lfloor n/2\rfloor$ rows of $D_n$. Further, denote $\hat{D}_n:=\hat{D}_n^{(1)}-\hat{D}_n^{(2)}$. Now we are ready to state the result for compressible and dominated vectors.

\begin{prop}   \label{p: dominated and compressible y}
  Let $\bar{A}_n, D_n, K$, and $p$ be as in Theorem \ref{thm: smallest singular + norm + non-centered}, and $\ell_0$ be as in Proposition {\ref{p: dominated and compressible}}.  Then there exist constants $0< c_{\ref{p: dominated and compressible y}}, \ol{c}_{\ref{p: dominated and compressible y}}, C_{\ref{p: dominated and compressible y}},\ol{C}_{\ref{p: dominated and compressible y}}, \wt{C}_{\ref{p: dominated and compressible y}}< \infty$, depending only on $K, R$, and the fourth moment of $\{\xi_{ij}\}$, such that
for any $p^{-1} \le M \le  c_{\ref{p: dominated and compressible y}} n$,
  \begin{align*}
   &\P(\exists x \in \text{\rm Dom}(M, (C_{\ref{p: dominated and compressible y}}(K+R))^{-4})
    \cup \text{\rm Comp}(M, \rho)\\
    & \qquad  \norm{(\hat{A}_n+\hat{D}_n)x}_2 \le \ol{C}_{\ref{p: dominated and compressible y}}(K+R) \rho \sqrt{np}
   \text{ and } \|{\hat{A}_n }\| \le K  \sqrt{pn})\le \exp(- \ol{c}_{\ref{p: dominated and compressible y}} pn),
  \end{align*}
  where $ \rho=(\wt{C}_{\ref{p: dominated and compressible y}}(K+R))^{-\ell_0-6}$.
 \end{prop}

 
 \begin{proof}
 We proceed as in the proof of Proposition \ref{p: dominated and compressible}. As in the proof of Proposition \ref{p: dominated and compressible}, we first need to control infimum over vectors close to $1/(8p)$-sparse vectors. That is, we need to show that there exists some constants $0< c, C, \wt{C}< \infty$, depending on the fourth moment of $\xi_{i,j}$,  $K$, and $R$, such that
  \begin{align}
   \P&\Big(\exists x \in \text{\rm Dom}\big((8p)^{-1}, (C(K+R))^{-1}\big) \text{ such that } \norm{(\hat{A}_n+\hat{D}_n)x}_2 \le (\wt{C}(K+R))^{-\ell_0} \sqrt{np} \notag\\
   &\hskip 3in \text{ and } \|{\hat{A}_n}\| \le K \sqrt{pn}\Big)\notag\\
   &\le \exp(-cpn). \label{eq:adaptation}
  \end{align}
 The analogue of \eqref{eq:adaptation} in the proof of Proposition \ref{p: dominated and compressible} crucially uses Lemma  \ref{l: pattern}. We therefore need to find a version of Lemma \ref{l: pattern} applicable to $\hat{A}_n$. To this end, we show that the entries of $\hat{A}_n$ have the required product structure. 
 
   Define i.i.d.~random variables $\hat{\theta}_{i,j} \in \{1,2,3\}$ independent of $\bar{A}_n$ such that
  \[
   \P (\hat{\theta}_{i,j}=1)=\P (\hat{\theta}_{i,j}=2)= \frac{1-p}{2-p} \quad
   \text{and} \quad  \P (\hat{\theta}_{i,j}=3) =\frac{p}{2-p}.
  \]  
  Set 
  \[
   \hat{\xi}_{i,j}:=\xi_{i,j} \cdot \mathbf{1}_{\hat{\theta}_{i,j}=1} 
   - \xi_{i+\lfloor n/2 \rfloor, j}  \cdot \mathbf{1}_{\hat{\theta}_{i,j}=2}
   +(\xi_{i,j}- \xi_{i+\lfloor n/2 \rfloor, j})  \cdot \mathbf{1}_{\hat{\theta}_{i,j}=3}.
  \]
  Let $\{\hat{\delta}_{i,j}\}$ be another family of i.i.d.~Bernoulli random variables independent of $\bar{A}_n$ taking value $1$ with probability $p(2-p)$. Then the random variable $\hat{a}_{i,j}= \delta_{i,j} \xi_{i,j}- \delta_{i+\lfloor n/2 \rfloor, j} \xi_{i+\lfloor n/2 \rfloor, j}$ has the same distribution as $\hat{\delta}_{i,j} \hat{\xi}_{i,j}$ and these entries are independent for all $i,j$. This is the desired product structure, and therefore we can proceed as in the proof of Lemma \ref{l: pattern}.
  
More elaborately, recall that the key to the proof of Lemma \ref{l: pattern} is bounds on $\P(i \in I^1(J))$, and $\P(i \in I^0(J'))$ for any $i \in [n]$ (for example, see \eqref{eq: i in I_1(J)} and \eqref{eq: i in I_0(J)}). We have the same inequalities here, using the product structure shown above. Now applying Chernoff's inequality, and proceeding same as there we obtain the an analogue of Lemma \ref{l: pattern} for $\hat{A}_n$. The only difference from Lemma \ref{l: pattern} is that the constants appearing there get reduced by one half, as we now have a matrix with $\lfloor n /2\rfloor$ rows, instead of $n$ rows. 

Equipped with this analogue of Lemma \ref{l: pattern} we then proceed as in the proof of Lemma \ref{l: sparse vectors}. Considering the case $p \ge (1/4) n^{-1/3}$, similar to \eqref{eq:norm_bound_A_n+D_n} we obtain
 \[
    \norm{(\hat{A}_n +\hat{D}_n)x}_2^2 \ge  \sum_{k \in \supp(x)} \sum_{i \in I_k} \Big|((\hat{A}_n +\hat{D}_n)x)_i\Big|^2.
 \]
 To get rid of $\hat{D}_n$ from the above expression, we lower bound the sum over $i \in I_k$ by a sum over $i \in I_k \setminus \widehat{\supp}(x)$, where $\widehat{\supp}(x):=\{j \in \lfloor n /2\rfloor: x_j\ne 0, \text{ or } x_{j +\lfloor n/2 \rfloor}  \ne 0\}$. Since $|\widehat{\supp}(x)| \ll |I_k|$, we can proceed as in \eqref{lower_bd_norm_A_n}, and obtain that 
 \[
    \norm{(\hat{A}_n+\hat{D}_n)x}_2^2 \ge
      \sum_{k \in \supp(x)} \sum_{i \in I_k\setminus \widehat{\supp}(x)} |(\hat{A}_n x)_i|^2\ge \sum_{k \in \supp(x)} \f{\ol{c}_{\ref{l: pattern}}pn}{2} |x_k|^2
    =\f{\ol{c}_{\ref{l: pattern}}pn}{2}.
   \]
  
 Next, repeating the same steps as in Lemma \ref{l: pattern}, we establish \eqref{eq:adaptation}. Proof of \eqref{eq:adaptation}, when $\frac{\ol{C}_{\ref{thm: smallest singular + norm}} \log n}{n} \le p < (1/4) n^{-1/3}$ requires a similar adaptation. Details are omitted.
 
We then need to extend  \eqref{eq:adaptation} for $\text{\rm Comp}((8p)^{-1},\rho)$ vectors, and this can be done repeating the proof of Lemma \ref{l: compressible}. Finally one needs to extend \eqref{eq:adaptation} for $\text{\rm Dom}(M, (C(K+R))^{-4})$ vectors, where $p^{-1} \le M \le cn$, and $c, C$ are some positive constants. 
For $\bar{A}_n$ this was done in Lemma \ref{l: dominated vectors} using L\'{e}vy concentration function, $\vep$-net argument, and the union bound. The estimate on the L\'{e}vy concentration function in Corollary \ref{c: spread vector} was derived from Lemma \ref{l: spread vector}. Note that Lemma \ref{l: spread vector} continues to hold for $\hat{A}_n$. This implies we also obtain Corollary \ref{c: spread vector} for $\hat{A}_n$, except for the constants appearing there are decreased by one half, as $\hat{A}_n$ has only $\lfloor n /2\rfloor$ rows. Shrinking $c_{\ref{p: dominated and compressible y}}$, as needed, we also argue that the $\vep$-net here is not too big. Therefore, one can carry out the same steps as in Lemma \ref{l: dominated vectors} to complete the proof. We omit the details.
 \end{proof}
 

\bigskip
Next we combine Proposition \ref{p: dominated and compressible y} and Proposition \ref{p: norm on S_L y} to prove Theorem \ref{thm: smallest singular + norm + non-centered}.

\begin{proof}[Proof of Theorem \ref{thm: smallest singular + norm + non-centered}]
As noted in the proof of Theorem \ref{thm: smallest singular + norm}, for any $\vartheta>0$,
 \begin{align}
 & \P\Big( \{s_{\min}(\bar{A}_n+D_n) \le \vartheta\}  \cap  \bar{\Omega}_K \Big) \notag\\
  \le& \,  \P\Big( \Big\{\inf_{x \in V^c} \norm{(\bar{A}_n+D_n)x}_2 \le \vartheta \Big\}  \cap\bar{\Omega}_K \Big)
  + \P\Big( \Big\{\inf_{x \in V} \norm{(\bar{A}_n+D_n)x}_2 \le \vartheta  \Big\} \cap  \bar{\Omega}_K \Big)\notag,
 \end{align}
 where $\bar{\Omega}_K:= \{\norm{\bar{A}_n - \E \bar{A}_n} \le K \sqrt{np}\}$, 
  \[
  V:=S^{n-1} \setminus \Big( \text{Comp}(c_{\ref{p: dominated and compressible y}}n, \rho) \cup \text{Dom}(c_{\ref{p: dominated and compressible y}}n, (C_{\ref{p: dominated and compressible y}}(K+R))^{-4}) \Big),
 \]
 and $\rho$ as in Proposition \ref{p: dominated and compressible y}. Now note that using triangle inequality we have that $\|\hat{A}_n \| \le 3K \sqrt{ np}$ on the event $\bar{\Omega}_K$. Next we observe that $\|(\bar{A}_n+D_n)x\|_2^2 \ge \|(\hat{A}_n^{(1)}+\hat{D}_n^{(1)})x\|_2^2+\|(\hat{A}_n^{(2)}+\hat{D}_n^{(2)})x\|_2^2$ for any $x \in \R^n$, and therefore we deduce that $\sqrt{2}\|(\bar{A}_n+D_n)x\|_2 \ge \|(\hat{A}_n+\hat{D}_n)x\|_2$. Thus applying Proposition \ref{p: dominated and compressible y} we obtain that 
 \begin{align*}
   &\P( \inf_{x \in V^c}  \norm{(\bar{A}_n+{D}_n)x}_2 \le \ol{C}_{\ref{p: dominated and compressible y}}(K+R) \rho \sqrt{np}
   \text{ and } \norm{\bar{A}_n -\E \bar{A}_n} \le K  \sqrt{pn})\le \exp(- \ol{c}_{\ref{p: dominated and compressible y}} pn).
  \end{align*}
It therefore remains to bound
 \[
  \P\Big( \Big\{\inf_{x \in V} \norm{(\bar{A}_n+{D}_n)x}_2 \le \vartheta  \Big\} \cap  \bar{\Omega}_K \Big).
 \]
To this end, proceeding as in the proof of Theorem \ref{thm: smallest singular + norm} we note that we need to bound 
 \[
 p_1:=\P\Big(\Big\{\exists v \in \ol{W}^c \text{ such that } \wt{A}_n^D v=0\Big\}\cap  \bar{\Omega}_K \Big)
 \]
 and
 \[
p_2:= \P\Big(\Big\{\exists v \in \ol{W} \text{ such that } \wt{A}_n^D v=0  \ \text{and} \ |\pr{\bar{A}_{n,1}}{v}| \le \rho \vep \sqrt{p}\Big\}\cap  \bar{\Omega}_K \Big),
 \]
where $Q=(\ol{C}_{\ref{p: norm on S_L y}}(K+R))^{12} p^{-1}$, and
 \[
  \ol{W}=S^{n-1} \setminus \Big( \text{Comp}(Q, \rho) \cup \text{Dom}(Q, (C_{\ref{p: dominated and compressible}}(K+R))^{-4}) \Big).
 \]
 To bound $p_1$ we again apply the same folding argument to the matrix $\wt{A}_n^D$. That is, we define the matrix $\hat{A}_n^m$ from the matrix $\wt{A}_n^m$, and then apply Proposition \ref{p: dominated and compressible y}.
  To bound $p_2$, as in the proof of Theorem \ref{thm: smallest singular + norm}, we decompose $\ol{W}$ into $W_1$ and $W_2$, where
 \[
  W_1:=\Big\{w \in W \mid \inf_{i \in [n]} D\bigg(\frac{w_{I(w)\setminus\{i\}}}{\norm{w_{I(w)\setminus\{i\}}}_2}\bigg) \le \exp (c''_{\ref{p: norm on S_L y}}pn/(K+R)^2) \Big\} \quad \text{and} \quad W_2:=\ol{W}\setminus W_1.
 \]

As in the proof of Theorem \ref{thm: smallest singular + norm}, we show that the probability that there exists $v \in W_1$ such that $ \wt{A}_n^D v=0$ is small. To this end, we decompose $W_1$ in the union of the sets $S_L$ as in \eqref{eq: W_1 via S_L}.
We will show that
   \[
    \P\Big( \inf_{v \in S_L} \norm{\wt{A}_n^{D,m} v}_2 \le \f{\ol{c}_{\ref{p: norm on S_L y}} \rho \vep_0 \sqrt{pn}}{2} \text{ and } \norm{\wt{A}_n^{D,m} - \E \wt{A}_n^{D,m}} \le K \sqrt{pn}\Big)
    \le \exp\left(-\f{\wt{c}_{\ref{p: norm on S_L y}}n}{2}\right).
   \]
To establish this bound, we combine Proposition \ref{p: norm on S_L y} with an additional $\vep$-net argument.
Note that, the set $\mu p \mathbf{U}_n^m S^{n-1}$ is contained in the interval of length $n^{O(1)}$ in the direction of $\mathbf{1}$, where ${\bf 1}$ is the vector of ones of length $m$. This interval has a small net.
 Let $\cY_n:=\{\gamma {\bf 1}: |\gamma| \le \sqrt{n}p|\mu|\}$.
  We claim that
 \begin{multline*}
   \P\Big( \inf_{y \in \cY_n}\inf_{v \in S_L} \norm{(\wt{A}_n^{D,m} - \mu p \bm{U}_n^m)v-y}_2 \le \f{\ol{c}_{\ref{p: norm on S_L y}} \rho \vep_0 \sqrt{pn}}{2},  \,  \norm{\wt{A}_n^{D,m} - \E \wt{A}_n^{D,m}} \le K \sqrt{pn} \Big)  \\
    \le \exp\left(-\f{\wt{c}_{\ref{p: norm on S_L y}}}{2}n\right).
 \end{multline*}
 To see this first note that, using triangle inequality we can deduce
 \[
\left| \inf_{v \in S_L} \norm{(\wt{A}_n^{D,m} -  \mu p \bm{U}_n^m)v-y}_2 - \inf_{v \in S_L} \norm{(\wt{A}_n^{D,m} -  \mu p \bm{U}_n^m)v-y'}_2\right| \le \norm{y-y'}_2.
 \]
 Choose an $\f{\ol{c}_{\ref{p: norm on S_L y}} \rho \vep_0 \sqrt{pn}}{2}$-net $\wt{\cY}_n$ of  the set $\cY_n$. 
 We proceed by applying Proposition \ref{p: norm on S_L y} for $y \in \wt{\cY}_n$ and taking the union bound. Recalling the definition of $\vep_0$, and using the fact that $L \le \exp({c}''_{\ref{p: norm on S_L y}}pn /(K+R)^2)$, where $K, R \ge 1$, we note that $|\wt{\cY}_n| \le \exp(\f{{c}''_{\ref{p: norm on S_L y}}pn}{4})$. Thus shrinking  ${c}''_{\ref{p: norm on S_L y}}$, if necessary, the claim now follows from a union bound.
 

We further note that 
\[
 \inf_{y \in \cY_n}\inf_{v \in S_L} \norm{(\wt{A}_n^{D,m} - \mu p \bm{U}_n^m)v-y}_2 \le \inf_{v \in S_L} \norm{\wt{A}_n^{D,m} v}_2,
\]
which establishes the claim.

The infimum over $W_2$ is dealt with using the L\'{e}vy concentration function. This part  remains the same. This yields the desired bound on $p_2$ completing the proof.

\end{proof}

We now apply Theorem \ref{thm: smallest singular + norm + non-centered} to prove Theorem \ref{thm:bernoulli}. 

\begin{proof}[Proof of Theorem \ref{thm:bernoulli}]
Recall that the adjacency matrix $\Adj_n$ of a directed Erd\H{o}s-R\'{e}yni graph with edge connectivity probability $p$, is a matrix with zero diagonal, and has i.i.d.~off-diagonal entries $a_{i,j}\sim \dBer(p)$. So, if we are able to express $a_{i,j}$ as a product two random variables $\xi_{i,j}$, and $\delta_{i,j}$, where $\xi_{i,j}$ is a random variable with unit variance, and bounded fourth moment, and $\delta_{i,j}$ is a Bernoulli random variable, then we can use Theorem \ref{thm: smallest singular + norm + non-centered} to obtain the desired result. To this end, we split the proof into two different cases, $p \le 1/2$ and $p > 1/2$.

First let us consider the case $p \le 1/2$. There we note that $a_{i,j}$ has the same distribution as $\xi_{i,j}\delta_{i,j}$ where $\xi_{i,j} \sim \dBer(1/2)$, and $\delta_{i,j} \sim \dBer(\bar{p})$ with $\bar{p}=2p$. Thus applying Theorem \ref{thm: smallest singular + norm + non-centered}, we obtain that, there exist constants $0 < c, \ol{c}, C, \ol{C} < \infty$, depending only on $K$ and $R$, such that for
\[
 \frac{\ol{C} \log n}{n}\le p \le  \f{1}{2},
\]
and any $\vep >0$,
 \begin{align*}
 & \P \left( s_{\min}(\Adj_n+D_n) \le {C} \vep \exp \left(-c \frac{\log (1/p)}{\log (np)} \right)  \sqrt{\frac{p}{n}}, \norm{\Adj_n - \E \Adj_n} \le K \sqrt{np} \right )\\
  & \qquad \qquad \qquad \qquad \qquad \qquad \qquad \qquad \qquad \qquad \qquad \qquad \qquad \qquad \qquad  \le \vep +  \exp(-\ol{c}np).
 \end{align*}
 Thus it only remains to show that there exists $K \ge 1$ such that 
 \[
 \P(\norm{\Adj_n - \E \Adj_n} \ge K \sqrt{pn}) \le \exp(-c_0 pn),
 \]
 for some small positive absolute constant $c_0$. Using the triangle inequality, we see that it is enough to prove the same for $A_n$ with i.i.d.~entries $a_{i,j} \sim \dBer(p)$. Since the function $\norm{A_n - \E A_n}$, when viewed as a function from $\R^{n^2}$ to $\R$ is a $1$-Lipschitz, quasi-convex function using Talagrand's inequality (see \cite[Theorem 7.12]{BLM}) we note that  
 \beq\label{eq:talagrand_median}
 \P\left(\left|\norm{A_n - \E A_n} - \M_n\right| \ge t\right) \le 4 \exp(-t^2/4),
 \eeq
 for all $t >0$, where $\M_n$ is the median of $\norm{A_n -\E A_n}$. From \eqref{eq:talagrand_median}, using integration by parts one also obtains that $\left|\E\norm{A_n - \E A_n} - \M_n \right|\le C_0$, for some absolute constant $C_0$. Thus it only remains to show that $\E\norm{A_n -\E A_n} \le C_1 \sqrt{ np}$, for some another absolute constant $C_1$.

Turning to prove the above, we use Seginer's theorem. Since for every $i,j \in [n]$, 
\[
\Var[(a_{i,j}-p)^2] \le \E[(a_{i,j}-p)^4] \le p(1-p), \quad \text{ and } \quad  |(a_{i,j}-p)^2 - \E[(a_{i,j}-p)^2] | \le 2,
\]
using Bennett's inequality, we obtain that there exists some $t_0>0$, and a small positive constant ${c}''$, such that 
\beq
\P(\norm{A_{j,n} - \E A_{j,n}}_2^2 \ge t np) \le \exp(-{c}''t np), \notag
\eeq
for all $t \ge t_0$. Now using the union bound, and integration by parts, upon applying Seginer's theorem, we obtain $\E\norm{A_n -\E A_n} \le C_1 \sqrt{ np}$. This completes the proof of the theorem for $p \le 1/2$.

For $p >1/2$ we cannot use the same trick as above to produce the desired product structure. Instead, we note that $1- a_{i,j} \sim \dBer(1-p)$. We use this observation to create the desired product structure. More precisely, we denote $\bar{A}_n'$ to be the matrix with zero diagonal, and has i.i.d.~off-diagonal entries $1-a_{i,j}$. Then, we have $\bar{A}_n +D_n = \bm{U}_n + D_n' - \bar{A}_n'$, where $D_n'$ is another diagonal matrix such that $(D_n')_{i,i}= (D_n)_{i,i}-1$, for $i \in [n]$, and $\bm{U}_n$ is the $n \times n$ matrix of all ones. Therefore, now it is enough to find quantitative estimates on the smallest singular value of  $\bm{U}_n + D_n' - \bar{A}_n'$, where the entries of $\bar{A}_n'$ have the desired product structure. Due to the presence of $\bm{U}_n$, we cannot directly apply Theorem \ref{thm: smallest singular + norm + non-centered} here. However, $\rank(\bm{U}_n)$ being one, the set $\bm{U}_n S^{n-1}$ admits an $\vep$-net of small cardinality. Therefore, we can modify the proof of Theorem \ref{thm: smallest singular + norm + non-centered} accordingly.

To this end, recall that the proof of Theorem \ref{thm: smallest singular + norm + non-centered} can be broadly divided into two parts. In the first part we control the infimum over compressible and dominated vectors (see Proposition \ref{p: dominated and compressible y}), and in the second part we consider incompressible vectors (see Proposition \ref{p: norm on S_L y}). Since in Proposition \ref{p: dominated and compressible y}, we use folding trick we do not feel the presence of $\bm{U}_n$. There, the proof remains unchanged. In Proposition   \ref{p: norm on S_L y} it calls for an additional $\vep$-net. Since the cardinality of such a net is small, it does not ruin the proof, and it only worsens the constants. Thus the proof of this theorem is complete.
\end{proof}

\subsection*{Acknowledgements}
We thank the anonymous referee for her/his suggestion to include the result on sparse directed Erd\H{o}s-R\'{e}yni graph. 

\end{document}